\documentclass[12pt, reqno]{amsart}

\usepackage[margin=1in]{geometry}  
\usepackage{graphicx}              
\usepackage{amsmath}               
\usepackage{amsfonts}              
\usepackage{amsthm}                
\usepackage{amssymb}
\usepackage{enumerate}
\usepackage{tikz-cd}
\usepackage{mathtools}
\usepackage{hyperref}
\usepackage[matrix,arrow,curve,frame]{xy}

\tikzcdset{scale cd/.style={every label/.append style={scale=#1},
    cells={nodes={scale=#1}}}}

\definecolor{my-linkcolor}{rgb}{0.75,0,0}
\definecolor{my-citecolor}{rgb}{0.1,0.57,0}
\definecolor{my-urlcolor}{rgb}{0,0,0.75}
\hypersetup{
	colorlinks,
	linkcolor={my-linkcolor},
	citecolor={my-citecolor},
	urlcolor={my-urlcolor}
}

\newtheorem{Proposition}{Proposition}[section]
\newtheorem{Theorem}[Proposition]{Theorem}
\newtheorem{Corollary}[Proposition]{Corollary}

\newtheorem{Lemma}[Proposition]{Lemma}

\theoremstyle{definition}\newtheorem{Remark}[Proposition]{Remark}
\theoremstyle{definition}\newtheorem{Definition}[Proposition]{Definition}

\begin{document}

\title[Logarithmic minimal models]{Fusion and (non)-rigidity of Virasoro Kac modules in logarithmic minimal models at $(p,q)$-central charge
}

\author{Robert McRae and  Valerii Sopin}

\address{Yau Mathematical Sciences Center, Tsinghua University, Beijing 100084, China}

\email{rhmcrae@tsinghua.edu.cn, vvs@myself.com}


%

\numberwithin{equation}{section}

\begin{abstract}
Let $\mathcal{O}_c$ be the category of finite-length modules for the Virasoro Lie algebra at central charge $c$ whose composition factors are irreducible quotients of reducible Verma modules. For any $c\in\mathbb{C}$, this category admits the vertex algebraic braided tensor category structure of Huang--Lepowsky--Zhang. Here, we begin the detailed study of $\mathcal{O}_{c_{p,q}}$ where $c_{p,q} = 1-\frac{6(p-q)^2}{pq}$ for relatively prime integers $p, q \geq 2$; in conformal field theory, $\mathcal{O}_{c_{p,q}}$ corresponds to a logarithmic extension of the central charge $c_{p,q}$ Virasoro minimal model. We particularly focus on the Virasoro Kac modules $\mathcal{K}_{r,s}$, $r,s\in\mathbb{Z}_{\geq 1}$, in $\mathcal{O}_{c_{p,q}}$ defined by Morin-Duchesne--Rasmussen--Ridout, which are finitely-generated submodules of Feigin--Fuchs modules for the Virasoro algebra. We prove that $\mathcal{K}_{r,s}$ is rigid and self-dual when $1\leq r\leq p$ and $1\leq s\leq q$, but that not all $\mathcal{K}_{r,s}$ are rigid when $r>p$ or $s>q$. That is, $\mathcal{O}_{c_{p,q}}$ is not a rigid tensor category. 
We also show that all Kac modules and all simple modules in $\mathcal{O}_{c_{p,q}}$ are homomorphic images of repeated tensor products of $\mathcal{K}_{1,2}$ and $\mathcal{K}_{2,1}$, and we determine completely how $\mathcal{K}_{1,2}$ and $\mathcal{K}_{2,1}$ tensor with Kac modules and simple modules in $\mathcal{O}_{c_{p,q}}$. In the process, we prove some fusion rule conjectures of Morin-Duchesne--Rasmussen--Ridout.
\end{abstract}

\maketitle

\tableofcontents

\allowdisplaybreaks

\section{Introduction}

The Virasoro Lie algebra $\mathfrak{Vir}$ is the unique central extension of the Witt algebra, the Lie algebra of differential operators on the circle \cite{GF}. It appears often in theoretical physics, including string theory (beginning in 1970 with Virasoro's work \cite{Vir}) and two-dimensional conformal field theory \cite{BPZ}. The representation theory of the Virasoro algebra was heavily studied during the 1980s (see \cite{IK} and the references cited there), but open problems remain, especially those related to fusion of Virasoro modules in logarithmic conformal field theory.

Logarithmic conformal field theories (LCFTs) are characterized by the appearance of logarithmic singularities in their correlation functions. Such singularities arise from logarithmic $\mathfrak{Vir}$-modules, on which the Virasoro zero-mode $L_0$ is not diagonalizable. LCFTs have applications in string theory and condensed matter physics, especially in the behavior of statistical models at a critical point; see for example the reviews \cite{Fl, Gab, Kawai, CrRi} and the references cited there. Basic examples of chiral LCFTs include the logarithmic minimal models $\mathcal{LM}(p,q)$ for relatively prime $p,q\in\mathbb{Z}_{\geq 1}$, which are based on the representation theory of the Virasoro algebra at central charge $c_{p,q}=1-\frac{6(p-q)^2}{pq}$. For $p,q\geq 2$, $\mathcal{LM}(p,q)$ is a logarithmic extension of the well-known Virasoro minimal model at central charge $c_{p,q}$ in rational conformal field theory. The simplest case $p=2$, $q=3$, at central charge $0$, is related to percolation at a critical point \cite{MR1, RP1}, and $\mathcal{LM}(p,q)$ for $p,q\geq 2$ in general has been studied extensively in the physics literature \cite{EF, PRZ, RP, RP3, MR2, Ri, MRR}. See also Section \ref{sec:conclusion} for discussion and references on $W$-extended logarithmic minimal models, and see for example \cite{GJRSV, Ni} for work on whether the chiral LCFT $\mathcal{LM}(p,q)$ can be extended to a full (bulk) LCFT.

In this paper, we study the chiral LCFT $\mathcal{LM}(p,q)$ for $p,q\geq 2$ using the universal Virasoro vertex operator algebra $V_{c_{p,q}}$ \cite{FZ} (see \cite{MY-cp1-vir} for similar work on $\mathcal{LM}(p,1)$). Vertex operator algebras, which are essentially the same as chiral algebras in the physics literature, are a powerful algebraic tool for studying two-dimensional chiral CFT. In particular, the fusion product of two representations of a chiral algebra in CFT can be defined mathematically using vertex algebraic intertwining operators \cite{FHL, HLZ2, HLZ3}; see also \cite{KaRi} for further references and detailed analysis on relations between physicists' and mathematicians' approaches to fusion. In rational CFT, after Moore and Seiberg's work on polynomial equations and the Verlinde formula \cite{MS1, MS2} led to Turaev's definition of modular tensor category \cite{Tu}, Huang proved in \cite{Hu-rigid} that the representations of a rational vertex operator algebra naturally form a (semisimple) modular tensor category. More recently in \cite{HLZ1}-\cite{HLZ8}, Huang, Lepowsky, and Zhang showed that logarithmic CFTs can also be studied using non-semisimple braided tensor categories of modules for non-rational vertex operator algebras.

The vertex operator algebra $V_{c_{p,q}}$ at central charge $c_{p,q}$ for relatively prime $p,q\geq 2$ is non-rational, but it has a rational simple quotient $L_{c_{p,q}}$ \cite{Wa}. Thus while the central charge $c_{p,q}$ minimal model in rational CFT corresponds to the modular tensor category of $L_{c_{p,q}}$-modules, it should be possible to study the logarithmic extension $\mathcal{LM}(p,q)$ using some braided tensor category of $V_{c_{p,q}}$-modules. In fact, it was recently shown in \cite{CJORY} that the universal Virasoro vertex operator algebra $V_c$ at any central charge $c\in\mathbb{C}$ has a representation category $\mathcal{O}_c$ that admits the braided tensor category structure of \cite{HLZ1}-\cite{HLZ8}. Specifically, $\mathcal{O}_c$ is the category of $C_1$-cofinite grading-restricted generalized $V_c$-modules, or equivalently the category of finite-length $\mathfrak{Vir}$-modules at central charge $c$ whose composition factors are simple quotients of reducible Verma modules. Writing $c=13-6t-6t^{-1}$ for some $t\in\mathbb{C}^\times$, the braided tensor category $\mathcal{O}_c$ is by now well understood for $t\notin\mathbb{Q}$ \cite{Fr-MZhu, CJORY}, $t\in\mathbb{Z}_{\geq 1}$ \cite{Mil, McR-cpt-orb, CMY-singlet, MY-cp1-vir}, and $t=-1$ \cite{OH, MY-c25-vir}.

In this paper, we begin the detailed study of the tensor structure of $\mathcal{O}_{c_{p,q}}$ for $p,q\geq 2$, which can be identified with the properties of fusion in the chiral logarithmic minimal model $\mathcal{LM}(p,q)$. In the mathematical physics literature, calculations and conjectures for fusion products in $\mathcal{LM}(p,q)$ (equivalently $\mathcal{O}_{c_{p,q}}$) can be found in, for example, \cite{EF, RP1, RP, Ri, RS, GV, GJRSV, MRR, Ni}; our work especially proves some of the fusion product conjectures in \cite{MRR}. We focus mainly on the Kac modules $\mathcal{K}_{r,s}$, $r,s\in\mathbb{Z}_{\geq 1}$, for the Virasoro algebra at central charge $c_{p,q}$ defined in \cite{MRR}. These are (usually non-simple) $\mathfrak{Vir}$-modules labeled by entries of the Kac table of conformal weights $h_{r,s}$ (the weights such that the Virasoro Verma modules of those lowest conformal weights are reducible), and each $\mathcal{K}_{r,s}$ has the same character as the Verma module of lowest conformal weight $h_{r,s}$ modulo a singular vector of conformal weight $h_{r,s}+rs$. However, $\mathcal{K}_{r,s}$ is not generally a Verma module quotient; rather, it is the submodule of a Feigin-Fuchs module \cite{FF-modules} of lowest conformal weight $h_{r,s}$ generated by all vectors of conformal weight strictly less than $h_{r,s}+rs$. It turns out that every $\mathcal{K}_{r,s}$ is a homomorphic image of some repeated fusion tensor product of $\mathcal{K}_{1,2}$ and $\mathcal{K}_{2,1}$, as suggested in part (2) of the following theorem, which summarizes our main results:
\begin{Theorem}\label{main_theorem}
For $r,s\in\mathbb{Z}_{\geq 1}$, let $\mathcal{K}_{r,s}$ be the Virasoro Kac module defined in \cite{MRR}.
\begin{enumerate}
\item (Theorem \ref{thm:Krs_rigid}) If $1\leq r\leq p$ and $1\leq s\leq q$, then $\mathcal{K}_{r,s}$ is rigid and self-dual in the tensor category $\mathcal{O}_{c_{p,q}}$.

\item (Theorem \ref{thm:main_exact_sequences}, Proposition \ref{prop:log_modules}, Remark \ref{rem:log_modules}) Let $\boxtimes$ denote the fusion tensor product on $\mathcal{O}_{c_{p,q}}$. For all $r,s\in\mathbb{Z}_{\geq 1}$, there is an exact sequence
\begin{equation*}
0\longrightarrow\mathcal{K}_{r,s-1}\longrightarrow\mathcal{K}_{1,2}\boxtimes\mathcal{K}_{r,s}\longrightarrow\mathcal{K}_{r,s+1}\longrightarrow 0
\end{equation*}
which splits if and only if $q\nmid s$, and $\mathcal{K}_{1,2}\boxtimes\mathcal{K}_{r,s}$ is a logarithmic $\mathfrak{Vir}$-module if $q\mid s$. Similarly, there is an exact sequence
\begin{equation*}
0\longrightarrow\mathcal{K}_{r-1,s}\longrightarrow\mathcal{K}_{2,1}\boxtimes\mathcal{K}_{r,s}\longrightarrow\mathcal{K}_{r+1,s}\longrightarrow 0
\end{equation*}
which splits if and only if $p\nmid r$, and $\mathcal{K}_{2,1}\boxtimes\mathcal{K}_{r,s}$ is logarithmic if $p\mid r$.

\item (Theorem \ref{thm:Kr1_times_K1s})   For all $r,s\in\mathbb{Z}_{\geq 1}$, $\mathcal{K}_{r,1}\boxtimes\mathcal{K}_{1,s}\cong\mathcal{K}_{r,s}$.

\end{enumerate}
\end{Theorem}

The rigidity mentioned in part (1) is a strong duality property (see for example \cite[Section 2.1]{BK}) which is very useful in the study of tensor categories; for example, exact sequences remain exact after tensoring with a rigid object. In fact, in the tensor category $\mathcal{O}_{c_{p,q}}$, rigidity of $\mathcal{K}_{1,2}$ and $\mathcal{K}_{2,1}$ in particular is crucial for proving parts (2) and (3) of Theorem \ref{main_theorem}. However, we show in Proposition \ref{prop:some_Krs_not_rigid} that not all Kac modules in $\mathcal{O}_{c_{p,q}}$ are rigid; essentially, the non-simplicity of the vertex operator algebra $V_{c_{p,q}}$ obstructs rigidity for the whole category $\mathcal{O}_{c_{p,q}}$. Nevertheless, the contragredient modules of \cite{FHL} give $\mathcal{O}_{c_{p,q}}$ a weaker duality structure: by \cite[Theorem 2.12]{ALSW}, $\mathcal{O}_{c_{p,q}}$ is a ribbon Grothendieck-Verdier category in the sense of \cite{BD}. It would be interesting to explore the implications for $\mathcal{O}_{c_{p,q}}$ of this weaker duality structure.

The exact sequences of Theorem \ref{main_theorem}(2) confirm that the fusion predictions at the level of Grothendieck rings in \cite[Equation 4.16]{MRR} are accurate, while Theorem \ref{main_theorem}(3) proves \cite[Conjecture 5]{MRR}. Although these results do not give a complete list of fusion products involving Kac modules, they are the basis for the straightforward derivation of many more fusion product formulas. For example, in Theorem \ref{thm:K12_times_simple}, we use parts (1) and (2) of Theorem \ref{main_theorem} to determine all fusion products of $\mathcal{K}_{1,2}$ with simple objects in $\mathcal{O}_{c_{p,q}}$. In future work, we plan to use Theorem \ref{main_theorem} to construct interesting indecomposable logarithmic modules conjectured in for example \cite{EF, RP}, which we expect will be projective in some suitable subcategory of $\mathcal{O}_{c_{p,q}}$.

We now discuss some elements of the proof of Theorem \ref{main_theorem}. For part (1), the main work is to prove that $\mathcal{K}_{1,2}$ and $\mathcal{K}_{2,1}$ are rigid (and self-dual). To prove this, we use explicit solutions of Belavin-Polyakov-Zamolodchikov (BPZ) differential equations \cite{BPZ}, similar to the method used in \cite{TW, CMY-singlet, MY-cp1-vir}.

To obtain upper bounds on fusion products in $\mathcal{O}_{c_{p,q}}$, we first constrain the lowest conformal weight spaces in fusion products using the method pioneered in \cite{FZ} involving (bi)modules for the Zhu algebra of $V_{c_{p,q}}$. This method is also closely related to the Nahm-Gaberdiel-Kausch algorithm \cite{Na, GK} used by physicists to compute fusion products. We also constrain the degrees of possible singular vectors in fusion products using a result of Miyamoto \cite{Miy} on fusion tensor products of $C_1$-cofinite modules for vertex operator algebras.

To obtain lower bounds on fusion products, we construct intertwining operators (as defined in \cite{FHL}). In Theorem \ref{thm:gen_Verma_intw_ops}, we construct intertwining operators involving $\mathcal{K}_{1,2}$ and two Virasoro Verma modules as the solutions of what are essentially BPZ equations. This method is analogous to what was used in \cite{MY-sl2-intw-op, McR-gen-intw-op} to construct intertwining operators among generalized Verma modules for affine Lie algebras. Theorem \ref{thm:gen_Verma_intw_ops} actually applies to the Virasoro algebra at any central charge and may be of independent interest.  In any case, intertwining operators involving Verma modules then descend to intertwining operators of the form $\mathcal{Y}: \mathcal{K}_{1,2}\otimes\mathcal{K}_{r,s}\rightarrow\mathcal{K}_{r,s\pm 1}\lbrace z\rbrace$ in certain situations where the Kac modules are actually quotients of Verma modules (which is true if either $r< p$ or $s< q$). In Theorem \ref{thm:Kac_module_surj}, we obtain intertwining operators using properties of Feigin-Fuchs modules. Namely, Feigin-Fuchs modules are the $\mathfrak{Vir}$-module restrictions of irreducible Fock modules for a Heisenberg vertex operator algebra with a modified conformal vector. Thus known intertwining operators among Heisenberg Fock modules restrict to $\mathfrak{Vir}$-module intertwining operators involving their Kac submodules, and we show in particular that they restrict to surjective intertwining operators of the form $\mathcal{Y}: \mathcal{K}_{1,2}\otimes\mathcal{K}_{r,s}\rightarrow\mathcal{K}_{r,s+1}\lbrace z\rbrace$ for any $r,s\in\mathbb{Z}_{\geq 1}$. This implies that there is always a surjection $\mathcal{K}_{1,2}\boxtimes\mathcal{K}_{r,s}\rightarrow\mathcal{K}_{r,s+1}$ as indicated in Theorem \ref{main_theorem}(2).

All of our intertwining operator constructions have analogues after exchanging $\mathcal{K}_{1,2}$ with $\mathcal{K}_{2,1}$ because $c_{p,q}=c_{q,p}$. Finally, to complete the proofs of parts (2) and (3) of Theorem \ref{main_theorem}, we heavily use the rigidity of $\mathcal{K}_{1,2}$ and $\mathcal{K}_{2,1}$. For example, rigidity of $\mathcal{K}_{1,2}$ implies that the surjection $\mathcal{K}_{1,2}\boxtimes\mathcal{K}_{r,s-1}\rightarrow\mathcal{K}_{r,s}$ given by Heisenberg intertwining operators dualizes to a non-zero map $\mathcal{K}_{r,s-1}\rightarrow\mathcal{K}_{1,2}\boxtimes{K}_{r,s}$ which turns out to be an injection. We also need exactness of the tensoring functors $\mathcal{K}_{1,2}\boxtimes\bullet$ and $\mathcal{K}_{2,1}\boxtimes\bullet$, which follows from rigidity of $\mathcal{K}_{1,2}$ and $\mathcal{K}_{2,1}$.

We now discuss the remaining contents of this paper. In Section \ref{sec:Vir}, we review the Virasoro algebra at central charge $c_{p,q}$, in particular the structure of Verma modules; the universal Virasoro vertex operator algebra $V_{c_{p,q}}$, its simple quotient $L_{c_{p,q}}$, and their modules; and the definitions and structures of Feigin-Fuchs modules at central charge $c_{p,q}$ and their Kac submodules defined in \cite{MRR}. In Section \ref{sec:intw_op_and_tens_cats}, we review vertex algebraic intertwining operators and some elements of the vertex algebraic tensor category theory of \cite{HLZ1}-\cite{HLZ8}, especially as they pertain to $V_{c_{p,q}}$ and its modules. In particular, we discuss some properties of $C_1$-cofinite modules for the Virasoro algebra and the relationships between the Zhu algebra of $V_{c_{p,q}}$, its (bi)modules, and vertex algebraic intertwining operators among modules for the Virasoro algebra. In Section \ref{sec:intw_op_constructions}, we give two constructions of intertwining operators: the first involves $\mathcal{K}_{1,2}$ (or $\mathcal{K}_{2,1}$) and Verma modules and works at any central charge, while the second involves restricting intertwining operators among modules for the Heisenberg vertex operator algebra. 

In Section \ref{sec:tens_prods_and_rigidity}, we prove some preliminary results on tensoring $\mathcal{K}_{1,2}$ and $\mathcal{K}_{2,1}$ with Verma module quotients, and then we prove that $\mathcal{K}_{1,2}$ and $\mathcal{K}_{2,1}$ are rigid. We then complete the proof of Theorem \ref{main_theorem}(1) and also show that $\mathcal{O}_{c_{p,q}}$ is not a rigid tensor category. In Section \ref{sec:main_results}, we complete the proofs of parts (2) and (3) of Theorem \ref{main_theorem}, and we also determine how $\mathcal{K}_{1,2}$ and $\mathcal{K}_{2,1}$ tensor with all simple objects of $\mathcal{O}_{c_{p,q}}$. Finally, in Section \ref{sec:conclusion}, we discuss potential future applications and extensions of our results and methods, including constructing logarithmic modules which are projective in some suitable subcategory of $\mathcal{O}_{c_{p,q}}$, studying the representation theories of the singlet and triplet $W$-algebra extensions of $V_{c_{p,q}}$, and studying the representation theories of the $N=1,2$ superconformal algebras at interesting rational central charges.

\vspace{3mm}

\noindent\textbf{Acknowledgments.} The results of this paper form part of the Ph.D. thesis of the second author at Tsinghua University. We thank Sylvain Ribault, David Ridout, and the referees for comments and references.

\section{The Virasoro algebra and its representations}\label{sec:Vir}

In this section, we recall basic classes of modules for the Virasoro Lie algebra, including Verma modules and Feigin-Fuchs modules, in particular at central charge $c_{p,q}$.

\subsection{The Virasoro algebra and Verma modules}\label{subsec:Vir_and_Verma}

Let $\mathfrak{Vir}$ be the Virasoro Lie algebra with basis $\{ L_{n}\mid n \in \mathbb{Z} \} \cup \{ C \}$ and Lie brackets 
\begin{align*}
& [L_m, L_n]=(m-n)L_{m+n}+ \frac{1}{12} (m^3 - m) \delta_{m, -n}C,\qquad [L_m, C]=0,
\end{align*}
where $m, n \in \mathbb{Z}$. To define Verma $\mathfrak{Vir}$-modules, we need the two subalgebras
\begin{equation*}
\mathfrak{Vir}_{\geq 0} = \bigoplus_{n = 0}^\infty \mathbb{C}L_n \oplus \mathbb{C}C,\qquad\mathfrak{Vir}_{-} = \bigoplus\limits_{n = 1}^\infty \mathbb{C}L_{-n}.
\end{equation*}
A $\mathfrak{Vir}$-module has \textit{central charge} $c\in\mathbb{C}$ if the central element $C$ acts by the scalar $c$, and a vector in a $\mathfrak{Vir}$-module has \textit{conformal weight} $h\in\mathbb{C}$ if it is a generalized $L_0$-eigenvector of generalized eigenvalue $h$.

For any $c, h\in\mathbb{C}$, the Verma module $\mathcal{V}_h$ at central charge $c$ is defined as follows: First take $\mathbb{C} v_h$ to be the one-dimensional $\mathfrak{Vir}_{\geq 0}$-module on which $C$ acts by $c$, $L_0$ acts by $h$, and for $n>0$, $L_n$ acts by $0$. Then the Verma module is the induced module
\begin{equation*}
\mathcal{V}_h=U(\mathfrak{Vir})\otimes_{U(\mathfrak{Vir}_{\geq 0})} \mathbb{C} v_h.
\end{equation*}
By the Poincar\'{e}-Birkhoff-Witt Theorem, it has a basis consisting of the vectors $L_{-n_1}\cdots L_{-n_k} v_h$ for $1\leq n_1\leq\cdots\leq n_k$. Each Verma module $\mathcal{V}_h$ has a unique simple quotient $\mathcal{L}_h$. Both $\mathcal{V}_h$ and $\mathcal{L}_h$ are highest-weight $\mathfrak{Vir}$-modules, which means that each is generated by a single highest-weight vector (that is, an $L_0$-eigenvector annihilated by $L_n$ for $n>0$).

The composition series structure of $\mathcal{V}_h$ was determined by Feigin and Fuchs \cite{FF} (see also \cite{As}, the expositions in \cite{KR} and \cite[Chapters 5 and 6]{IK}, and the summary in \cite[Section 2.2]{CJORY}). In particular, every non-zero submodule of $\mathcal{V}_h$ is either a Verma module or a (non-direct) sum of two Verma submodules. To describe the structure in more detail, it is convenient to write $c=13-6t-6t^{-1}$ for some $t\in\mathbb{C}^\times$. The Verma module $\mathcal{V}_h$ is then reducible if $h=h_{r,s}$ for some $r,s\in\mathbb{Z}_{\geq 1}$, where
\begin{equation}\label{eqn:h_rs_def}
h_{r,s} =\frac{r^2-1}{4} t-\frac{rs-1}{2}+\frac{s^2-1}{4} t^{-1}.
\end{equation}
For convenience, we write $\mathcal{V}_{r,s}=\mathcal{V}_{h_{r,s}}$ and $\mathcal{L}_{r,s}=\mathcal{L}_{h_{r,s}}$ for $r,s\in\mathbb{Z}_{\geq 1}$. We also use $h_{r,s}$ for non-positive $r,s\in\mathbb{Z}$ to denote the same number \eqref{eqn:h_rs_def}; note the symmetry $h_{r,s}=h_{-r,-s}$.

In this paper, we mainly consider the central charges $c_{p,q}$, where $t=\frac{q}{p}$ for $p,q\in\mathbb{Z}_{\geq 2}$ and $\mathrm{gcd}(p,q)=1$; note that $c_{p,q}=c_{q,p}$. At central charge $c_{p,q}$, we have the additional conformal weight symmetry $h_{r,s}=h_{r+p,s+q}$ for $r,s\in\mathbb{Z}$. The conformal weight symmetries imply that any $h_{r,s}$ for $r,s\in\mathbb{Z}_{\geq 1}$ equals at least one such that $s\in\mathbb{Z}_{\geq 1}$ and $1\leq r\leq p$, and the only cases of $h_{r,s}=h_{r',s'}$ with $s,s'\in\mathbb{Z}_{\geq 1}$, $1\leq r,r'\leq p$, and $(r,s)\neq (r',s')$ are $h_{r,s} = h_{p-r,q-s}$ for $1\leq r\leq p-1$ and $1\leq s\leq q-1$, and $h_{p,s}=h_{p,2q-s}$ for $1\leq s\leq q-1$. The Verma modules $\mathcal{V}_{r,s}$ and their simple quotients $\mathcal{L}_{r,s}$ are thus parametrized by conformal weights $h_{r,s}$ such that $1\leq r\leq p$ and $ps\geq qr$. Because $c_{p,q}=c_{q,p}$, they are also parametrized by conformal weights $h_{r,s}$ such that $1\leq s\leq q$ and $qr\geq ps$.

We now recall the embedding diagrams for Verma modules at central charge $c_{p,q}$ derived in \cite{FF}; see also \cite[Section 5.3]{IK} and \cite[Theorem 2.2.4(1)]{CJORY}:
\begin{itemize}
\item\textit{Bulk case}: For $1\leq r\leq p-1$, $1\leq s\leq q-1$, and $n\geq 0$, we have embedding diagrams
\begin{equation*}
\xymatrixrowsep{.5pc}
\xymatrix{
 & \mathcal{V}_{p-r,(n+1)q+s} \ar[ld] & \mathcal{V}_{p-r,(n+3)q-s} \ar[l] \ar[ldd] & \mathcal{V}_{p-r,(n+3)q+s} \ar[l] \ar[ldd] & \cdots \ar[l]\ar[ldd]\\
\mathcal{V}_{r,nq+s} & & & &\\
 & \mathcal{V}_{r,(n+2)q-s} \ar[lu] & \mathcal{V}_{r,(n+2)q+s} \ar[l] \ar[luu] & \mathcal{V}_{r,(n+4)q-s} \ar[luu],\ar[l] & \cdots \ar[l] \ar[luu]\\
}
\end{equation*}

\item\textit{Boundary cases}: For $1\leq s\leq q-1$ and $n\geq 1$, we have embedding diagrams
\begin{equation*}
\xymatrix{
\mathcal{V}_{p,nq+s} & \mathcal{V}_{p,(n+2)q-s} \ar[l] & \mathcal{V}_{p,(n+2)q+s} \ar[l] & \mathcal{V}_{p,(n+4)q-s} \ar[l] & \cdots \ar[l]\\
}
\end{equation*}
For $1\leq r\leq p-1$ and $n\geq 1$, we have embedding diagrams
\begin{equation*}
\xymatrix{
\mathcal{V}_{r,nq} & \mathcal{V}_{p-r,(n+1)q} \ar[l] & \mathcal{V}_{r,(n+2)q} \ar[l] & \mathcal{V}_{p-r,(n+3)q} \ar[l] & \cdots \ar[l]\\
}
\end{equation*}

\item\textit{Corner case}: For $n\geq 1$, we have embedding diagrams
\begin{equation*}
\xymatrix{
\mathcal{V}_{p,nq} & \mathcal{V}_{p,(n+2)q} \ar[l] & \mathcal{V}_{p,(n+4)q} \ar[l] & \mathcal{V}_{p,(n+6)q} \ar[l] & \cdots \ar[l]\\
}
\end{equation*}
\end{itemize}
Because all submodules of Verma modules are generated by their singular vectors, these embedding diagrams depict all submodules of $\mathcal{V}_{r,s}$. Note that at central charge $c_{p,q}$, the Verma module $\mathcal{V}_h$ is irreducible if $h\neq h_{r,s}$ for all $r,s\in\mathbb{Z}_{\geq 1}$.

\subsection{Virasoro vertex operator algebras}

Let $c\in\mathbb{C}$ be a central charge, and let $\mathcal{V}_0$ be the Virasoro Verma module of central charge $c$ generated by a highest-weight vector $\mathbf{1}$ of conformal weight $0$. Then $L_{-1}\mathbf{1}$ is a singular vector in $\mathcal{V}_0$, and it is well-known \cite{FZ} (see also \cite[Section 6.1]{LL}) that $V_c =\mathcal{V}_0/\langle L_{-1}\mathbf{1}\rangle$ has the structure of a vertex operator algebra in the sense of \cite{FLM, FHL, LL}. In particular, there is a linear map
\begin{align}\label{eqn:alg_vertex_op}
Y: V_c & \rightarrow(\mathrm{End}\,V_c)[[z,z^{-1}]]\nonumber\\
 v & \mapsto Y(v,z) =\sum_{n\in\mathbb{Z}} v_n\,z^{-n-1}
\end{align}
such that $Y(\mathbf{1},x)=\mathrm{Id}_{V_c}$ (among other properties), and there is a conformal vector $\omega = L_{-2}\mathbf{1}$ such that
\begin{equation*}
Y(\omega,z)=\sum_{n\in\mathbb{Z}} L_n\,z^{-n-2},
\end{equation*}
that is, the operator $\omega_{n+1}\in\mathrm{End}\,V_c$ is the Virasoro operator $L_n$. We call $V_c$ the \textit{universal Virasoro vertex operator algebra} at central charge $c$.

A $\mathfrak{Vir}$-module $W$ at central charge $c$ is a \textit{grading-restricted generalized $V_c$-module} (or \textit{$V_c$-module} for short) if it satisfies the following properties:
\begin{enumerate}
\item Lower truncation: For any $w\in W$, $L_n w=0$ for $n\in\mathbb{Z}$ sufficiently positive.

\item Conformal weight grading: $W=\bigoplus_{h\in\mathbb{C}} W_{[h]}$ where $W_{[h]}$ is the generalized $L_0$-eigenspace of eigenvalue $h$.

\item Grading restrictions: $\dim W_{[h]}<\infty$ for all $h\in\mathbb{C}$, and for any fixed $h\in\mathbb{C}$, $W_{[h-n]}=0$ for $n\in\mathbb{Z}$ sufficiently positive.
\end{enumerate}
Any $V_c$-module $W$ admits a linear map
\begin{align}\label{eqn:mod_vertex_op}
Y_W: V_c & \rightarrow (\mathrm{End}\,W)[[z,z^{-1}]]\\
 v & \mapsto Y_W(v,z)=\sum_{n\in\mathbb{Z}} v_n\,z^{-n-1}
\end{align}
uniquely determined by the conditions $Y_W(\mathbf{1},z)=\mathrm{Id}_W$ and
\begin{equation*}
Y_W(\omega,z)=\sum_{n\in\mathbb{Z}} L_n\,z^{-n-2},
\end{equation*}
as well as by the axioms for modules for a general vertex operator algebra \cite{FLM, FHL, LL}.

\begin{Remark}\label{rem:CFT_VOA_relns}
In conformal field theory, a vertex operator algebra $V$ corresponds to the vacuum module of a chiral conformal field theory, and the maps $Y(v,z)$ and $Y_W(v,z)$ are conformal fields associated to states $v\in V$. In \eqref{eqn:alg_vertex_op} and \eqref{eqn:mod_vertex_op}, we have used the vertex algebraic convention that $v_n$ denotes the coefficient of $z^{-n-1}$ in $Y(v,z)$ or $Y_W(v,z)$ for arbitrary $v\in V$, rather than the physics convention that $v_n$ denotes the coefficient of $z^{-n-m}$ for $L_0$-eigenvectors $v\in V$ with eigenvalue $m$. The vertex algebraic convention was introduced in \cite{FLM} to simplify the appearance of the various general algebraic identities satisfied in a vertex operator algebra.
\end{Remark}

For any vertex operator algebra $V$ and $V$-module $W$, the \textit{contragredient} of $W$ \cite{FHL} is a $V$-module structure on the graded dual vector space $W'=\bigoplus_{h \in \mathbb{C}} W_{[h]}^*$, with $Y_{W'}$ defined by
\begin{equation}\label{eqn:contra}
 \langle Y_{W'}(v, z)w', w \rangle = \langle w', Y_{W}(e^{z L_1}(-z^{-2})^{L_0}v, z^{-1})w \rangle 
\end{equation} 
for all $v\in V$, $w'\in W'$, and $w\in W$. Here $\langle  \; , \; \rangle$ denotes the bilinear pairing between a graded vector
space and its graded dual.

A vertex operator algebra is simple if it is simple as a module for itself. For a $V_c$-module $W$, $V_c$-submodules are the same as $\mathfrak{Vir}$-submodules, so $V_c$ is simple if and only if it is simple as a $\mathfrak{Vir}$-module. At central charge $c_{p, q} = 1 - \frac{6(p-q)^2}{pq}$ with $p,q\geq 2$ and $\mathrm{gcd}(p,q)=1$, the Verma module embedding diagrams from the previous subsection show that $V_{c_{p,q}}=\mathcal{V}_{1,1}/\langle L_{-1}\mathbf{1}\rangle$ contains $\mathcal{L}_{1,2q-1} =\mathcal{L}_{2p-1,1}$ as its unique non-zero proper submodule and thus is not simple. The quotient $L_{c_{p,q}} =V_{c_{p,q}}/\mathcal{L}_{1,2q-1}$ is the \textit{simple Virasoro vertex operator algebra} at central charge $c_{p,q}$. It was conjectured in \cite{FZ} and proved in \cite{Wa} that $L_{c_{p, q}}$ is rational, which means in particular that every $L_{c_{p,q}}$-module is semisimple and that $L_{c_{p,q}}$ has finitely many simple modules up to isomorphism. In fact, the simple $L_{c_{p,q}}$-modules are precisely the $\mathcal{L}_{r,s}$ such that $r < p$ and $s< q$. By \cite{Hu-rigid}, the category of $L_{c_{p,q}}$-modules is a modular tensor category.

\subsection{Feigin-Fuchs modules and Kac modules}\label{subsec:FF_and_K_modules}

We now recall the Feigin-Fuchs modules \cite{FF-modules} for the Virasoro algebra at central charge $c_{p,q}$, which yield free-field realizations of the Virasoro algebra. For further details, see \cite[Chapters 4 and 8]{IK}; here, we mainly follow the exposition of \cite[Appendix E]{MRR}.

The Heisenberg Lie algebra $\mathfrak{H}$ has basis $\lbrace a_n\,\vert\,n \in \mathbb{Z}\rbrace\cup\lbrace K\rbrace$, with $K$ central, and
\begin{equation*}
[a_m, a_n] = m \delta_{m,-n}K
\end{equation*}
for $m,n\in\mathbb{Z}$.  We have $\mathfrak{H}=\mathfrak{H}_{\geq 0}\oplus\mathfrak{H}_-$ where 
\begin{equation*}
\mathfrak{H}_{\geq 0} =\bigoplus_{n=0}^\infty \mathbb{C} a_n\oplus\mathbb{C}K,\qquad\mathfrak{H}_-=\bigoplus_{n=1}^\infty\mathbb{C} a_{-n}.
\end{equation*}
For $\lambda\in\mathbb{C}$, the Fock $\mathfrak{H}$-module $\mathcal{F}_\lambda$ is defined as follows: First let $\mathbb{C}v_\lambda$ be the one-dimensional $\mathfrak{H}_{\geq 0}$-module on which $K$ acts by $1$, $a_0$ acts by $\lambda$, and for $n>0$, $a_n$ acts by $0$. Then the Fock module is the induced module
\begin{equation*}
\mathcal{F}_\lambda =U(\mathfrak{H})\otimes_{U(\mathfrak{H}_{\geq 0})} \mathbb{C} v_\lambda.
\end{equation*}
It is linearly spanned by vectors $a_{-n_1}\cdots a_{-n_k} v_\lambda$ for $1\leq n_1\leq\cdots\leq n_k$.

The Fock module $\mathcal{F}_0$ has the structure of a vertex operator algebra \cite{FLM} (see also \cite[Section 6.3]{LL}) with the linear map $Y: \mathcal{F}_0\rightarrow(\mathrm{End}\,\mathcal{F}_0)[[z,z^{-1}]]$ characterized by
\begin{equation*}
Y(a_{-1} v_0,z) =\sum_{n\in\mathbb{Z}} a_n\,z^{-n-1}.
\end{equation*}
Each $\mathcal{F}_\lambda$ is a simple $\mathcal{F}_0$-module. The vertex operator algebra $\mathcal{F}_0$ admits a one-parameter family of conformal vectors
\begin{equation*}
\omega_Q=\frac{1}{2}\left( a_{-1}^2 +Qa_{-2}\right)v_0
\end{equation*}
for $Q\in\mathbb{C}$, which yield $\mathfrak{Vir}$-actions of central charge $1-3Q^2$ on each $\mathcal{F}_\lambda$:
\begin{equation*}
L_{n} = \frac{1}{2} \sum_{j \in \mathbb{Z}} a_j a_{n-j} - \frac{1}{2}(n+1)Qa_n\quad\text{for}\quad n\neq 0, \qquad L_{0} = \frac{1}{2}a_0^2  + \sum_{j=1}^\infty a_{-j} a_{j} - \frac{1}{2}Qa_0.
\end{equation*} 
With respect to the conformal vector $\omega_Q$, the contragredient module of $\mathcal{F}_\lambda$ is $\mathcal{F}_\lambda ' \cong\mathcal{F}_{Q-\lambda}$.

To get central charge $c_{p,q}$, we take $Q=\sqrt{\frac{2}{pq}}(q-p)$. For each $\lambda$, the generator $v_\lambda\in\mathcal{F}_\lambda$ is an $L_0$-eigenvector of conformal weight $\frac{1}{2}\lambda(\lambda-Q)$. For selected values of $\lambda$, these conformal weights agree with $h_{r,s}$ as defined in \eqref{eqn:h_rs_def}. In particular, for $r,s\in\mathbb{Z}$, define
\begin{equation}\label{lambda_rs_def}
\lambda_{r,s} = (1-r)\sqrt{\frac{q}{2p}}-(1-s)\sqrt{\frac{p}{2q}}.
\end{equation}
For convenience, we write $\mathcal{F}_{r,s}=\mathcal{F}_{\lambda_{r,s}}$, with generating vector $v_{r,s}=v_{\lambda_{r,s}}$, for $r,s\in\mathbb{Z}$. Then $v_{r,s}$ has conformal weight $h_{r,s}$, and as a vector space graded by conformal weights, $\mathcal{F}_{r,s}$ is isomorphic to the Verma module $\mathcal{V}_{r,s}$. However, $\mathcal{F}_{r,s}$ is not isomorphic to $\mathcal{V}_{r,s}$ as a $\mathfrak{Vir}$-module; for example, $v_{r,s}$ does not generate $\mathcal{F}_{r,s}$ as a $\mathfrak{Vir}$-module. When we consider the Fock module $\mathcal{F}_{r,s}$ as a $\mathfrak{Vir}$-module (as opposed to an $\mathfrak{H}$-module), we call it a Feigin-Fuchs module.

The Heisenberg weights $\lambda_{r,s}$ satisfy the same symmetry $\lambda_{r,s}=\lambda_{r+p,s+q}$ as the Virasoro conformal weights $h_{r,s}$ . However, $\lambda_{r,s}\neq\lambda_{-r,-s}$, but rather $\lambda_{-r,-s}=Q-\lambda_{r,s}$, so that $\mathcal{F}_{-r,-s}\cong\mathcal{F}_{r,s}'$. Thus each $\lambda_{r,s}$ is equal to a unique such weight with either $1\leq r\leq p$ or $1\leq s\leq q$. For $r,s\in\mathbb{Z}_{\geq 1}$ such that $1\leq r\leq p$ and $ps\geq qr$, except for $(r,s)=(p,q)$, there are two distinct Fock modules with lowest conformal weight $h_{r,s}$:
\begin{itemize}

\item\textit{Bulk case}: For $1\leq r\leq p-1$, $1\leq s\leq q-1$, and $n\geq 0$, $\mathcal{F}_{r,nq+s}$ and its contragredient $\mathcal{F}_{(n+1)p-r,q-s}$ have lowest conformal weight $h_{r,nq+s}$.

\item\textit{Boundary cases}: For $1\leq s\leq q-1$ and $n\geq 1$, $\mathcal{F}_{p,nq+s}$ and its contragredient $\mathcal{F}_{np,q-s}$ have lowest conformal weight $h_{p,nq+s}$, and for $1\leq r\leq p-1$ and $n\geq 1 $, $\mathcal{F}_{r,nq}$ and its contragredient $\mathcal{F}_{(n+1)p-r,q}$ have lowest conformal weight $h_{r,nq}$.

\item\textit{Corner case}: For $n\geq 1$, $\mathcal{F}_{p,nq}$ and its contragredient $\mathcal{F}_{np,q}$ have  lowest conformal weight $h_{p,nq}$.

\end{itemize}

The Feigin-Fuchs module $\mathcal{F}_{r,s}$ has the same composition factors as the Verma module $\mathcal{V}_{r,s}$ but different $\mathfrak{Vir}$-module structure. The following diagrams show the composition factors of each $\mathcal{F}_{r,s}$. We draw an arrow between two simple subquotients if together they determine an indecomposable length-$2$ subquotient, with the arrow pointing to the submodule of the length-$2$ subquotient. In other words, an arrow from one simple subquotient to another indicates that Virasoro operators can be used to move vectors corresponding to the first subquotient to vectors corresponding to the second:
\begin{itemize}

\item\textit{Bulk case}: For $1\leq r\leq p-1$, $1\leq s\leq q-1$, and $n\geq 0$, $\mathcal{F}_{r,nq+s}$ has the structure:
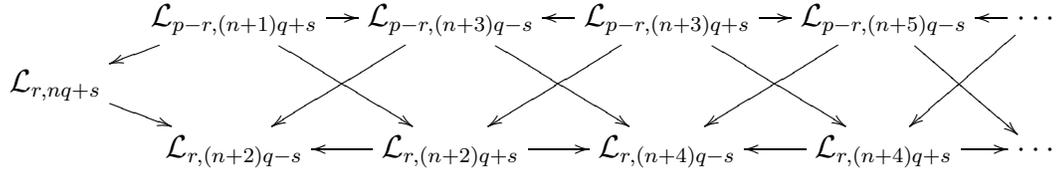
\begin{equation*}
\xymatrixcolsep{1pc}
\xymatrixrowsep{.5pc}
\xymatrix{
 & \mathcal{L}_{p-r,(n+1)q+s} \ar[ld] \ar[r] \ar[rdd] & \mathcal{L}_{p-r,(n+3)q-s} \ar[ldd] \ar[rdd]  & \mathcal{L}_{p-r,(n+3)q+s} \ar[l] \ar[r] \ar[ldd] \ar[rdd] &  \mathcal{L}_{p-r,(n+5)q-s} \ar[ldd] \ar[rdd] & \cdots \ar[l]  \ar[ldd]  \\
\mathcal{L}_{r,nq+s} \ar[rd] & & & &  \\
 & \mathcal{L}_{r,(n+2)q-s}  & \mathcal{L}_{r,(n+2)q+s} \ar[l] \ar[r] & \mathcal{L}_{r,(n+4)q-s} & \mathcal{L}_{r,(n+4)q+s} \ar[l] \ar[r] & \cdots \\
}
\end{equation*}
Since $\mathcal{F}_{(n+1)p-r,q-s}\cong\mathcal{F}_{r,s}'$ as both $\mathfrak{Vir}$-module and $\mathcal{F}_0$-module, its structure as a $\mathfrak{Vir}$-module is the same as the above except that all arrows are reversed. 

\item \textit{Boundary cases}: For $1\leq s\leq q-1$ and $n\geq 1$, $\mathcal{F}_{p,nq+s}$ has the structure:
\begin{equation*}
\xymatrixcolsep{1.5pc}
\xymatrix{
\mathcal{L}_{p, nq+s} \ar[r] & \mathcal{L}_{p,(n+2)q-s} & \ar[l] \mathcal{L}_{p,(n+2)q+s} \ar[r] & \mathcal{L}_{p,(n+4)q-s} & \mathcal{L}_{p,(n+4)q+s} \ar[l] \ar[r] & \cdots \\
}
\end{equation*}
Its contragredient $\mathcal{F}_{np,q-s}$ has the same composition factors, with arrows reversed in its structure. For $1\leq r\leq p-1$ and $n\geq 1$, $\mathcal{F}_{r,nq}$ has the structure
\begin{equation*}
\xymatrix{
\mathcal{L}_{r,nq} & \mathcal{L}_{p-r,(n+1)q} \ar[l] \ar[r] & \mathcal{L}_{r,(n+2)q}  & \mathcal{L}_{p-r,(n+3)q} \ar[l] \ar[r]  & \mathcal{L}_{r,(n+4)q} & \cdots \ar[l] \\
}
\end{equation*}
Its contragredient $\mathcal{F}_{(n+1)p-r,q}$ has the same structure with arrows reversed.

\item\textit{Corner case}: For $n\geq 1$, $\mathcal{F}_{p,nq}$ has the structure
\begin{equation*}
\xymatrix{
\mathcal{L}_{p,nq} & \mathcal{L}_{p,(n+2)q} & \mathcal{L}_{p,(n+4)q}  & \mathcal{L}_{p,(n+6)q} & \mathcal{L}_{p,(n+8)q} & \cdots \\
}
\end{equation*}
In particular, $\mathcal{F}_{p,nq}$ is semisimple as a $\mathfrak{Vir}$-module and thus is isomorphic to its contragredient $\mathcal{F}_{np,q}$ as a $\mathfrak{Vir}$-module (though not as an $\mathcal{F}_0$-module if $n>1$).

\end{itemize}

We now define Virasoro Kac modules following \cite[Section 3.2.1]{MRR}. For $r,s\in\mathbb{Z}_{\geq 0}$, the Kac module $\mathcal{K}_{r,s}$ is the submodule of $\mathcal{F}_{r,s}$ generated by all vectors of conformal weight strictly less than $h_{r,s} + rs$ (so $\mathcal{K}_{r,0}=\mathcal{K}_{0,s}=0$ for $r,s\in\mathbb{Z}_{\geq 0}$). Note that although $\mathcal{F}_{r,s}\cong\mathcal{F}_{r+p,s+q}$ for $r,s\in\mathbb{Z}_{\geq 1}$, the Kac modules $\mathcal{K}_{r,s}$ and $\mathcal{K}_{r+p,s+q}$ are different since the latter has more generators. The composition series structure of $\mathcal{K}_{r,s}$ follows from that of $\mathcal{F}_{r,s}$:
\begin{itemize}

\item \textit{Bulk case}: For $1\leq r\leq p-1$, $1\leq s\leq q-1$, and $m,n\geq 0$, the structures of $\mathcal{K}_{mp+r,nq+s}$ for $m\leq n$ and $n\leq m$, respectively, are:
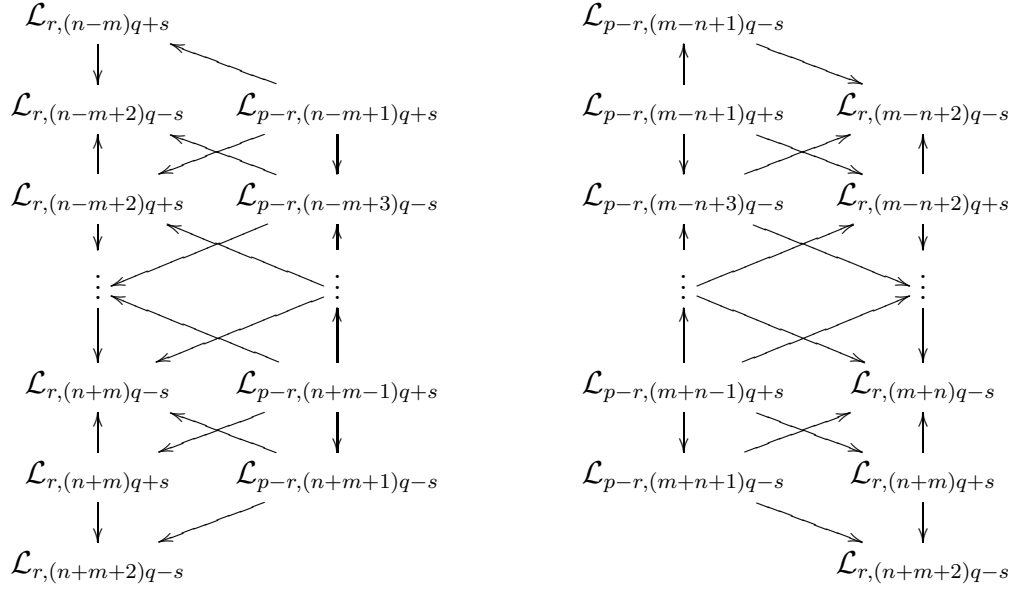
\begin{equation*}
\xymatrixcolsep{1pc}
\xymatrixrowsep{1.25pc}
\xymatrix{
 \mathcal{L}_{r,(n-m)q+s}  \ar[d] & \\
\mathcal{L}_{r,(n-m+2)q-s}  & \mathcal{L}_{p-r,(n-m+1)q+s} \ar[lu] \ar[ld] \ar[d]\\
\mathcal{L}_{r,(n-m+2)q+s} \ar[u]\ar[d] & \mathcal{L}_{p-r,(n-m+3)q-s} \ar[lu]\ar[ld]\\
\vdots \ar[d] & \vdots \ar[u]\ar[ld] \ar[lu]\\
\mathcal{L}_{r,(n+m)q-s}  & \mathcal{L}_{p-r,(n+m-1)q+s} \ar[u]\ar[d] \ar[lu]\ar[ld]\\
\mathcal{L}_{r,(n+m)q+s} \ar[u]\ar[d] & \mathcal{L}_{p-r,(n+m+1)q-s} \ar[lu]\ar[ld]\\
\mathcal{L}_{r,(n+m+2)q-s}  & \\
} 
\qquad\qquad\xymatrix{
\mathcal{L}_{p-r,(m-n+1)q-s} \ar[rd] & \\
\mathcal{L}_{p-r,(m-n+1)q+s}\ar[u]\ar[d]\ar[rd] & \mathcal{L}_{r,(m-n+2)q-s} \\
\mathcal{L}_{p-r, (m-n+3)q-s} \ar[ru]\ar[rd] & \mathcal{L}_{r,(m-n+2)q+s}\ar[u]\ar[d]\\
\vdots \ar[u]\ar[ru]\ar[rd] & \vdots \ar[d]\\
\mathcal{L}_{p-r, (m+n-1)q+s} \ar[u]\ar[d]\ar[ru]\ar[rd] & \mathcal{L}_{r,(m+n)q-s}\\
\mathcal{L}_{p-r,(m+n+1)q-s} \ar[ru]\ar[rd] & \mathcal{L}_{r,(n+m)q+s} \ar[u]\ar[d]\\
 & \mathcal{L}_{r,(n+m+2)q-s} \\
}
\end{equation*}
Note that these two diagrams agree when $m=n$ because $h_{r,s}=h_{p-r,q-s}$, and that $\mathcal{K}_{mp+r,nq+s}$ is an object of $\mathcal{O}_{c_{p,q}}$ with finite length $4\min(m,n)+2$.

\item\textit{Boundary cases}: For $1\leq s\leq q-1$, $m\geq 1$, and $n\geq 0$, the structures of $\mathcal{K}_{mp,nq+s}$ for $m\leq n$ and $n<m$, respectively, are:
\begin{align*}
\xymatrixcolsep{1pc}
\xymatrix{
\mathcal{L}_{p,(n-m+1)q+s} \ar[r] & \mathcal{L}_{p,(n-m+3)q-s} & \cdots \ar[l] & \mathcal{L}_{p,(n+m-1)q+s} \ar[l]\ar[r] &\mathcal{L}_{p,(n+m+1)q-s} \\
} &\\
\xymatrixcolsep{1pc}
\xymatrix{
\mathcal{L}_{p,(m-n+1)q-s} & \ar[l] \mathcal{L}_{p,(m-n+1)q+s} \ar[r] & \mathcal{L}_{p,(m-n+3)q-s} & \cdots \ar[l] & \mathcal{L}_{p,(m+n-1)q+s} \ar[l]\ar[r] &\mathcal{L}_{p,(m+n+1)q-s}
} & 
\end{align*}
Note that $\mathcal{K}_{mp,nq+s}$ has length $2m$ if $m\leq n$ and $2n+1$ if $n<m$. For $1\leq r\leq p-1$, $m\geq 0$, and $n\geq 1$, the structures of $\mathcal{K}_{mp+r,nq}$ for $m<n$ and $n\leq m$, respectively, are:
\begin{align*}
\xymatrixcolsep{1pc}
\xymatrix{
\mathcal{L}_{r,(n-m)q} & \ar[l] \mathcal{L}_{p-r,(n-m+1)q} \ar[r] & \mathcal{L}_{r,(n-m+2)q} & \cdots \ar[l] & \mathcal{L}_{p-r,(m+n-1)q} \ar[l]\ar[r] &\mathcal{L}_{r,(m+n)q}
} & \nonumber\\
\xymatrixcolsep{1pc}
\xymatrix{
\mathcal{L}_{p-r,(m-n+1)q} \ar[r] & \mathcal{L}_{r,(m-n+2)q} & \cdots \ar[l] & \mathcal{L}_{p-r,(m+n-1)q} \ar[l]\ar[r] &\mathcal{L}_{r,(n+m)q} \\
} &
\end{align*}
So $\mathcal{K}_{mp+r,nq}$ has length $2m+1$ if $m<n$ and length $2n$ if $n\leq m$.

\item\textit{Corner case}: For $m,n\geq 1$, $\mathcal{K}_{mp,nq}\cong\mathcal{K}_{np,mq}$ is semisimple with structure 
\begin{equation*}
\xymatrixcolsep{1pc}
\xymatrix{
\mathcal{L}_{p,(\vert m-n\vert +1)q} & \mathcal{L}_{p,(\vert m-n\vert +3)q} & \cdots & \mathcal{L}_{p,(m+n-3)q} & \mathcal{L}_{p,(m+n-1)q}
}
\end{equation*}
So $\mathcal{K}_{mp,nq}$ has length $\min(m,n)$.

\end{itemize}

\begin{Remark}\label{rem:Krs_Verma_quot}
The above composition series structures of the Kac modules $\mathcal{K}_{r,s}$ imply that $\mathcal{K}_{r,s}$ is singly-generated by a highest-weight vector of minimal conformal weight in $\mathcal{F}_{r,s}$ if and only if $r\leq p$ or $s\leq q$. Thus when $r\leq p$ or $s\leq q$, $\mathcal{K}_{r,s}$ is the quotient of the Verma module $\mathcal{V}_{r,s}$ by a singular vector of conformal weight $h_{r,s} +rs$.
\end{Remark}

\begin{Remark}\label{rem:K_small_rs_2}
The Verma module quotients $\mathcal{K}_{r,s}$ where $r\leq p$ or $s\leq q$ comprise every length-$2$ quotient of the bulk case Verma modules $\mathcal{V}_{r,s}$ (where $1\leq r\leq p-1$, $ps\geq qr$, and $q\nmid s$), every simple and length-$2$ quotient of the boundary case Verma modules $\mathcal{V}_{p,s}$ (where $s>q$ and $q\nmid s$) and $\mathcal{V}_{r,nq}$ (where $1\leq r\leq p-1$ and $n\geq 1$), and every simple quotient of the corner case Verma modules $\mathcal{V}_{p,nq}$ (where $n\geq 1$).
\end{Remark}

\section{Intertwining operators and Virasoro tensor categories}\label{sec:intw_op_and_tens_cats}

In this section, we recall properties of vertex algebraic intertwining operators among modules for the Virasoro vertex operator algebra $V_{c_{p,q}}$ and recall basic properties of the tensor category of $V_{c_{p,q}}$-modules obtained in \cite{CJORY}.

\subsection{Intertwining operators and tensor product modules}

If $W_1$, $W_2$, $W_3$ are modules for a vertex operator algebra $V$, then an intertwining operator of type $\binom{W_3}{W_1\,W_2}$ is a linear map
\begin{align*}
\mathcal{Y}: W_1 & \rightarrow \mathrm{Hom}(W_2,W_3)\lbrace z\rbrace[\log z]\\
w_1 & \mapsto \mathcal{Y}(w_1,z)=\sum_{h\in\mathbb{C}}\sum_{k\in\mathbb{Z}_{\geq 0}} (w_1)_{h;k}\,z^{-h-1} (\log z)^k
\end{align*}
which satisfies properties similar to those satisfied by $Y_W$ for a $V$-module $W$ (see for example \cite[Definition 3.10]{HLZ2} for the full definition). Intertwining operators of type $\binom{W_3}{W_1\,W_2}$ correspond to $V$-module homomorphisms from the fusion tensor product of $W_1$ and $W_2$ into $W_3$ and thus are essential for understanding tensor categories of $V$-modules.

When $V$ is a (universal or simple) Virasoro vertex operator algebra, the defining properties of an intertwining operator $\mathcal{Y}$ of type $\binom{W_3}{W_1\,W_2}$ become the following:
\begin{enumerate}

\item Lower truncation: For all $w_1\in W_1$, $w_2\in W_2$, and $h\in\mathbb{C}$, $(w_1)_{h+n;k} w_2 = 0$ for $n\in\mathbb{Z}$ sufficiently positive, independently of $k$.

\item The commutator formula: For all $w_1\in W_1$ and $n\in\mathbb{Z}$,
\begin{equation}\label{eqn:Vir_comm_form}
 L_n\mathcal{Y}(w_1,z) =\mathcal{Y}(w_1,z)L_n+\sum_{i\geq 0}\binom{n+1}{i} z^{n+1-i}\mathcal{Y}(L_{i-1} w_1,z).
\end{equation}

\item The iterate formula: For all $w_1\in W_1$ and $n\in\mathbb{Z}$,
\begin{align}\label{eqn:Vir_it_form}
\mathcal{Y}(L_n w_1,z) =\sum_{i\geq 0} (-1)^i\binom{n+1}{i}\left( L_{n-i}\,z^i\mathcal{Y}(w_1,z)+(-1)^{n} z^{n+1-i}\mathcal{Y}(w_1,z)L_{i-1}\right).
\end{align}

\item The $L_{-1}$-derivative property: For $w_1 \in W_1$, $\mathcal{Y}(L_{-1}w_1, z) = \frac{d}{dz}\mathcal{Y}(w_1, z)$.

\end{enumerate}
If $\mathcal{Y}$ is a linear map satisfying \eqref{eqn:Vir_comm_form} and \eqref{eqn:Vir_it_form}, then it also satisfies the Jacobi identity from \cite[Definition 3.10]{HLZ2} and thus is indeed an intertwining operator in the sense of that definition. We say that an intertwining operator $\mathcal{Y}$ of type $\binom{W_3}{W_1\,W_2}$ is \textit{surjective} if $W_3$ is spanned by the vectors $(w_1)_{h;k} w_2$ for $w_1\in W_1$, $w_2\in W_2$, $h\in\mathbb{C}$, and $k\in\mathbb{Z}_{\geq 0}$.

\begin{Remark}\label{rem:intw_op_CFT_relns}
Intertwining operators for a vertex operator algebra correspond to fields in a chiral conformal field theory (for example, compare the iterate formula \eqref{eqn:Vir_it_form} for $n<0$ with the definition of descendant field in \cite[Equation 6.148]{Fran}). In particular, if $\mathcal{Y}$ is an intertwining operator of type $\binom{W_3}{W_1\,W_2}$ and $w_1\in W_1$ satisfies $L_0 w_1=hw_1$ and $L_n w_1=0$ for $n>0$, then $\mathcal{Y}(w_1,z)$ is a primary field.
\end{Remark}


The fusion tensor product of modules for a vertex operator algebra $V$ corresponds to fusion in conformal field theory; see in particular \cite{KaRi} for a detailed comparison of the vertex algebraic approach to fusion developed in \cite{HLZ1}-\cite{HLZ8} with the Nahm-Gaberdiel-Kausch algorithm \cite{Na, GK} for computing fusion tensor products in non-rational conformal field theory. In \cite{HLZ3}, the fusion tensor product of $V$-modules is defined in terms of intertwining maps, which are intertwining operators with the formal variable $z$ specialized to a non-zero complex number using a choice of branch of logarithm. However, we can also define the fusion tensor product in terms of intertwining operators:
\begin{Definition}
  Let $\mathcal{C}$ be a category of grading-restricted generalized $V$-modules containing $W_1$ and $W_2$. A \textit{tensor product} of $W_1$ and $W_2$ in $\mathcal{C}$ is a pair $(W_1 \boxtimes W_2, \mathcal{Y}_{\boxtimes})$, where $W_1 \boxtimes W_2$ is a module in $\mathcal{C}$ and $\mathcal{Y}_{\boxtimes}$ is an intertwining operator  of type $\binom{W_1\boxtimes W_2}{W_1\,W_2}$, which satisfies the following universal property: For any module $W_3$ in $\mathcal{C}$ and intertwining operator $\mathcal{Y}$ of type $\binom{W_3}{W_1\,W_2}$,  there is a unique $V$-module homomorphism $f: W_1 \boxtimes W_2 \rightarrow W_3$ such that $\mathcal{Y} = f \circ \mathcal{Y}_{\boxtimes}.$
  \end{Definition}

If a tensor product $(W_1 \boxtimes W_2, \mathcal{Y}_{\boxtimes})$ exists, then it is unique up to unique isomorphism, and the  intertwining operator $\mathcal{Y}_{\boxtimes}$ is surjective. We sometimes call the tensor product module $W_1\boxtimes W_2$ the fusion tensor product to emphasize that it is not a $V$-module structure on the tensor product vector space $W_1\otimes W_2$. If a category $\mathcal{C}$ of $V$-modules satisfies certain conditions including existence of tensor products (see \cite{HLZ8} for details), then tensor products give $\mathcal{C}$ the structure of a braided tensor category with unit object $V$. See \cite{HLZ8} or the exposition in \cite[Section 3.3]{CKM-exts} for descriptions of the left and right unit isomorphisms, associativity isomorphisms, and braiding isomorphisms in $\mathcal{C}$. Later, we will need the explicit formula for the braiding isomorphisms: If $W_1$ and $W_2$ are modules in a vertex algebraic braided tensor category $\mathcal{C}$, then the braiding isomorphism $\mathcal{R}_{W_1,W_2}: W_1\boxtimes W_2\rightarrow W_2\boxtimes W_1$ is characterized by
\begin{equation}\label{eqn:braiding}
\mathcal{R}_{W_1,W_2}(\mathcal{Y}_\boxtimes(w_1,z)w_2) = e^{zL_{-1}}\mathcal{Y}_\boxtimes(w_2,e^{\pi i} z)w_1
\end{equation}
for all $w_1\in W_1$, $w_2\in W_2$.

\subsection{\texorpdfstring{$C_1$}{C1}-cofinite modules for the Virasoro algebra}

Given a module $W$ for a vertex operator algebra $V$, define the subspace
\begin{equation*}
C_1(W) = \mathrm{span} \lbrace v_{-1} w \; | \; v \in V,\, w \in W,\, \mathrm{wt}\,v > 0\rbrace,
\end{equation*}
where $\mathrm{wt}\,v$ denotes the conformal weight of a vector $v\in V$. The module $W$ is \textit{$C_1$-cofinite} if $\dim W/C_1(W)<\infty$. This vertex algebraic notion of $C_1$-cofiniteness corresponds to the conformal field theoretic notion of quasi-rationality defined by Nahm in \cite{Na}. A $V$-module $W$ is $C_1$-cofinite if and only if one can write $W=T+C_1(W)$ for some finite-dimensional subspace $T\subseteq W$, called a \textit{special subspace} in \cite{Na}.

\begin{Remark}
Recalling Remark \ref{rem:CFT_VOA_relns}, the space $C_1(W)$ would be defined in physics notation as the span of vectors $v_{-m} w$ for $w\in W$ and $v\in V$ homogeneous of conformal weight $m>0$.
\end{Remark}

In \cite{Na}, Nahm used a conformal field theoretic definition of fusion product (see also \cite{GK}) to show that the fusion product of two quasi-rational modules is quasi-rational. In \cite{Miy}, Miyamoto proved this same result in a vertex algebraic context, showing that the category of $C_1$-cofinite modules is closed under tensor products. Moreover, in \cite{Hu-diff-eqns}, Huang proved that compositions of intertwining operators involving $C_1$-cofinite modules for a vertex operator algebra satisfy systems of differential equations with regular singular points, a key step towards proving the associativity of intertwining operators involving $C_1$-cofinite modules. Thus it is natural to expect that the category of $C_1$-cofinite modules for a vertex operator algebra should satisfy the conditions of \cite{HLZ8} for existence of braided tensor category structure. These conditions have now been verified for the $C_1$-cofinite module categories of many vertex operator algebras; in particular, the main result of \cite{CJORY} is the following:
\begin{Theorem}
Let $V_c$ be the universal Virasoro vertex operator algebra at any central charge $c\in\mathbb{C}$, and let $\mathcal{O}_c$ be the category of $C_1$-cofinite $V_c$-modules. Then $\mathcal{O}_c$ is equal to the category of finite-length $V_c$-modules whose composition factors come from $\lbrace\mathcal{L}_{r,s}\;\vert\; r,s\in\mathbb{Z}_{\geq 1}\rbrace$. Moreover, $\mathcal{O}_c$ admits the braided tensor category structure of \cite{HLZ8}.
\end{Theorem}

\begin{Remark}
For $c=c_{p,q}$ where $p,q\geq 2$ and $\mathrm{gcd}(p,q)=1$, the braided tensor category $\mathcal{O}_{c_{p,q}}$ does not include the Verma modules $\mathcal{V}_{r,s}$ or the Feigin-Fuchs modules $\mathcal{F}_{r,s}$, since these modules do not have finite length. But $\mathcal{O}_{c_{p,q}}$ does include all proper quotients of the Verma modules $\mathcal{V}_{r,s}$ and all Kac modules $\mathcal{K}_{r,s}$.
\end{Remark}

For a $C_1$-cofinite module $W$, we call $\dim(W/C_1(W))$ the \textit{cofinite dimension} of $W$ and denote it by $\dim_{C_1}(W)$. That is, $\dim_{C_1}(W)$ is the dimension of a minimal special subspace for $W$. The next three propositions on cofinite dimension will help constrain the size of certain tensor product modules in $\mathcal{O}_{c_{p,q}}$. The first was derived in \cite{Na} and is also a quantitative version of the Key Theorem of \cite{Miy}; the bound on cofinite dimension can be obtained as the number of independent solutions to a finite system of differential equations (see \cite[Equation 2.3]{Miy}):
\begin{Proposition} \label{prop:Miyamoto}
If $W_1$ and $W_2$ are $C_1$-cofinite modules for a vertex operator algebra and $\mathcal{Y}$ is a surjective intertwining operator of type $\binom{W_3}{W_1\,W_2}$, then $W_3$ is also $C_1$-cofinite. In particular, 
$$\dim_{C_1}(W_3) \leq \dim_{C_1} (W_1)\cdot\dim_{C_1} (W_2).$$
\end{Proposition}

\begin{Proposition}\label{prop:C1_quotient}
If $W_1$ is a $C_1$-cofinite module for a vertex operator algebra and $W_2\subseteq W_1$ is a submodule, then $\dim_{C_1} (W_1/W_2)\leq \dim_{C_1}(W_1).$
\end{Proposition}
\begin{proof}
If $T$ is a finite-dimensional special subspace for $W_1$, then $(T+W_2)/W_2\cong T/(W_2\cap T)$ is a special subspace for $W_1/W_2$.
\end{proof}

\begin{Proposition}\label{prop:C1-cofinite_dimension}
Let $\mathcal{V}$ be the Virasoro Verma module of lowest conformal weight $h$ at central charge $c\in\mathbb{C}$, and let $\tilde{v}\in\mathcal{V}$ be a singular vector of conformal weight $h+N$ for some $N\in\mathbb{Z}_{\geq 0}$. Then the cofinite dimension of $\mathcal{V}/\langle\tilde{v}\rangle$ is equal to $N$.
\end{Proposition}
\begin{proof}
It is straightforward to show that $C_1(\mathcal{V})$ is spanned by all $L_{-n}w$ for any $n > 1$ and any $w \in \mathcal{V}$ (see \cite[Equation 2.27]{CJORY}), so that $\mathcal{V}=C_1(\mathcal{V})+\bigoplus_{n=0}^\infty L_{-1}^n v$ where $v$ is a generator of $\mathcal{V}$ of minimal conformal weight $h$. The singular vector $\tilde{v}\in\mathcal{V}$ may be expressed as a linear combination of Poincar\'{e}-Birkhoff-Witt ordered monomials in the $L_{-n}$ with $n>0$, acting on $v$. A crucial fact (see for example \cite[Theorem 3.1]{As}, where the result is attributed to Fuchs) is that the coefficient of $L^N_{-1} v$ is never $0$ (irrespective of the chosen order), and thus $\mathcal{V}/\langle\tilde{v}\rangle =C_1(\mathcal{V}/\langle \tilde{v}\rangle)+\bigoplus_{n=0}^{N-1} L_{-1}^N v$. As a result, $\dim_{C_1}(\mathcal{V}/\langle\tilde{v}\rangle)\leq N$.

To show equality, suppose some non-trivial linear combination of $L_{-1}^n v$, $0\leq n\leq N-1$, were contained in $C_1(\mathcal{V}/\langle\tilde{v}\rangle)$. Then since $C_1(\mathcal{V}/\langle\tilde{v}\rangle)$ is a graded vector space, we would get $L_{-1}^n v\in C_1(\mathcal{V}/\langle\tilde{v}\rangle)$ for some $n\leq N-1$. Using \cite[Equation 2.27]{CJORY} again, this would imply a relation of the form
\begin{equation*}
L_{-1}^n v +\sum_{m=2}^{n} L_{-m} w_m \equiv 0\quad(\mathrm{mod}\,\langle\tilde{v}\rangle),
\end{equation*}
where $w_m\in\mathcal{V}$ has conformal weight $h+n-m$. But this relation cannot vanish in $\mathcal{V}$ since $\mathcal{V}$ has a PBW basis, contradicting that $h+N$ is the minimal conformal weight of $\langle\tilde{v}\rangle$.
\end{proof}

\subsection{Zhu algebra theory}\label{subsec:Zhu_theory}

In the following sections, we will compute some fusion tensor products of Kac modules in $\mathcal{O}_{c_{p,q}}$ for $p,q\geq 2$ and $\mathrm{gcd}(p,q)=1$. One of the main tools we will use to constrain the sizes of these tensor product modules is the Zhu algebra, first constructed in \cite{Zh} to study the irreducible representations of a vertex operator algebra.

We recall the definition of the Zhu algebra for any vertex operator algebra $V$: First define a bilinear operation $*$ on $V$ by
\begin{equation*}
u * v = \mathrm{Res}_z\,z^{-1}Y((1+z)^{L_0} u, z)v
\end{equation*}
for $u, v\in V$, and then define the subspace $O(V)\subseteq V$ to be the linear span of the vectors
\begin{equation*}
\mathrm{Res}_z\,z^{-2}Y((1+z)^{L_0} u, z)v
\end{equation*}
for $u,v\in V$. Zhu showed that $*$ is well defined and associative on the quotient $A(V)=V/O(V)$, and thus $A(V)$ is an associative algebra with unit $\mathbf{1}+O(V)$. Similarly, in \cite{FZ}, Frenkel and Zhu constructed an $A(V)$-bimodule for any $V$-module $W$: We set $A(W)=W/O(W)$ where $O(W)$ is the span of vectors
\begin{equation*}
\mathrm{Res}_z\,z^{-2}Y((1+z)^{L_0} v, z)w
\end{equation*} 
for $v\in V$, $w\in W$. Then $A(W)$ is an $A(V)$-bimodule with left and right actions
\begin{align*}
(v+O(V))*(w+O(W)) & =\mathrm{Res}_z\,z^{-1}Y((1+z)^{L_0} v, z)w+O(W)\nonumber\\
 (w+O(W))*(v+O(V)) & =\mathrm{Res}_z\,z^{-1}Y((1+z)^{L_0-1} v, z)w+O(W)
\end{align*}
for $v\in V$, $w\in W$.

Zhu also proved in \cite{Zh} that there is a one-to-one correspondence between simple (grading-restricted) $V$-modules and simple finite-dimensional $A(V)$-modules. In particular, the lowest conformal weight space of a simple $V$-module is an $A(V)$-module with action
\begin{equation*}
(v+O(V))\cdot w = o(v)w,
\end{equation*}
where $o(v)=\mathrm{Res}_z\,Y(z^{L_0-1} v,z)$ is the component of $Y(v,z)$ that preserves conformal weights (recalling Remark \ref{rem:CFT_VOA_relns}, $o(v)$ is the zero-mode of $Y(v,z)$ in physics terminology). More generally, if $W$ is not necessarily simple, then the same formula gives an $A(V)$-action on the direct sum of the minimal conformal weight spaces of $W$. In particular,
\begin{equation}\label{eqn:W_conf_wt_decomp}
W = \bigoplus_{i \in I}\bigoplus_{n=0}^{\infty} W_{[h_i+n]},
\end{equation}
where $I$ is the set of cosets  $i\in\mathbb{C}/\mathbb{Z}$ such that $W_{[h]}\neq 0$ for some $h\in i$, and $h_i$ is the conformal weight of $W$ in $i$ with minimum real part. Then $W$ has the $\mathbb{Z}_{\geq 0}$-grading $W=\bigoplus_{n=0}^\infty W(n)$ where $W(n)=\bigoplus_{i\in I} W_{[h_i+n]}$, and this grading satisfies
\begin{equation*}
v_m \cdot W(n) \subseteq W(\mathrm{wt}\,v+n-m-1) 
\end{equation*}
for $v\in V$, $m\in\mathbb{Z}$, and $n\in\mathbb{Z}_{\geq 0}$. Thus for any $v\in V$, the operator $o(v)$ preserves the subspaces $W(n)$, and $o(v)$ defines an action of $A(V)$ on $W(0)$.

\begin{Remark}
For indecomposable $V$-modules $W$, the set $I\subseteq\mathbb{C}/\mathbb{Z}$ in \eqref{eqn:W_conf_wt_decomp} contains only one coset. More generally in this paper, we will only consider $V$-modules such that $I$ is finite.
\end{Remark}

In \cite{FZ, Li-intw-ops}, Frenkel-Zhu and Li gave precise relationships between $V$-module intertwining operators and $A(V)$-modules and bimodules. Let $W_1$, $W_2$, and $W_3$ be $V$-modules such that the conformal weights of $W_3$ are contained in the union of finitely many cosets of $\mathbb{C}/\mathbb{Z}$, that is, the set $I$ in \eqref{eqn:W_conf_wt_decomp} is finite. In this case, if $\mathcal{Y}$ is an intertwining operator of type $\binom{W_3}{W_1\,W_2}$, we can substitute $z\mapsto 1$ in $\mathcal{Y}$ (that is, $z^{-h-1}\mapsto 1$ for any $h\in\mathbb{C}$ and $\log z\mapsto 0$) to get a ``$P(1)$-intertwining map'' in the terminology of \cite{HLZ3}:
\begin{align*}
\mathcal{Y}(\cdot, 1)\cdot : W_1\otimes W_2 & \rightarrow \prod_{n=0}^\infty W_3(n)\\
w_1\otimes w_2 & \mapsto \mathcal{Y}(w_1,1)w_2.
\end{align*}
Letting $\pi_0:\prod_{n=0}^\infty W_3(n) \rightarrow W_3(0)$ denote the projection, we then get an $A(V)$-module map first constructed in \cite{FZ}:
\begin{align}\label{eqn:AV_(bi)mod_map}
\pi(\mathcal{Y}): A(W_1)\otimes_{A(V)} W_2(0) & \rightarrow W_3(0)\\
(w_1+O(W_1))\otimes u_2 & \mapsto \pi_0\left(\mathcal{Y}(w_1,1)u_2\right)\nonumber
\end{align}
for $w_1\in W_1$, $u_2\in W_2(0)$.

If we further assume that $W_2$ is generated by $W_2(0)$ as a $V$-module, and that $\mathcal{Y}$ is a surjective intertwining operator, then the $A(V)$-module homomorphism $\pi(\mathcal{Y})$ is surjective:
\begin{Proposition}\label{prop:surj_intw_op}
Let $\mathcal{Y}$ be a surjective intertwining operator of type $\binom{W_3}{W_1\,W_2}$ where $W_2$ is generated as a $V$-module by $W_2(0)$ and the conformal weights of $W_3$ are contained in the union of finitely many cosets in $\mathbb{C}/\mathbb{Z}$. Then the $A(V)$-module $W_3(0)$ is a homomorphic image of $A(W_1)\otimes_{A(V)} W_2(0)$. In particular this holds when $W_3=W_1\boxtimes W_2$ is a fusion tensor product of $W_1$ and $W_2$, and $\mathcal{Y}=\mathcal{Y}_\boxtimes$.
\end{Proposition}

This proposition is proved for example in \cite[Proposition 2.5]{MY-cp1-vir}, but it is actually a version of a result originally due to Nahm \cite{Na}. More specifically, Nahm's result in vertex algebraic language is that $(W_1\boxtimes W_2)(0)$ is spanned by vectors of the form $\pi_0(\mathcal{Y}_\boxtimes(t_1, 1)u_2)$ where $t_1$ comes from a special subspace $T_1$ for $W_1$ and $u_2\in W_2(0)$. Proposition \ref{prop:surj_intw_op} follows from this result because $T_1+O(W_1)$ generates $A(W_1)$ as an $A(V)$-bimodule \cite[Proposition 3.16]{Li-fin}.

Now suppose $V=V_{c_{p,q}}$ is the universal Virasoro vertex operator algebra at central charge $c_{p,q}=1-\frac{6(p-q)^2}{pq}$ for $p,q\geq 2$, $\mathrm{gcd}(p,q)=1$. Then there is an isomorphism $A(V_{c_{p,q}})\rightarrow\mathbb{C}[x]$ determined by $\omega+O(V_{c_{p,q}})\mapsto x$ \cite{FZ}. In the next sections, we will compute the fusion tensor products $\mathcal{K}_{1,2}\boxtimes\mathcal{K}_{r,s}$ and $\mathcal{K}_{2,1}\boxtimes\mathcal{K}_{r,s}$ of Kac modules, so to use the above proposition, we need the $\mathbb{C}[x]$-modules $A(\mathcal{K}_{1,2})\otimes_{\mathbb{C}[x]} \mathbb{C} v_{r,s}$ and $A(\mathcal{K}_{2,1})\otimes_{\mathbb{C}[x]} \mathbb{C} v_{r,s}$ for $r,s\in\mathbb{Z}_{\geq 1}$. These modules were (essentially) determined in \cite{Fr-MZhu}; see also \cite[Section 3.1]{MY-cp1-vir}:
\begin{align}\label{eqn:bimodules}
\begin{matrix}
A(\mathcal{K}_{1,2})\otimes_{\mathbb{C}[x]} \mathbb{C} v_{r,s} & \cong\mathbb{C}[x]/\langle(x-h_{r,s-1})(x-h_{r,s+1})\rangle\\
A(\mathcal{K}_{2,1})\otimes_{\mathbb{C}[x]} \mathbb{C} v_{r,s} &\cong\mathbb{C}[x]/\langle(x-h_{r-1,s})(x-h_{r+1,s})\rangle
\end{matrix}
\end{align}
Note that $h_{r,s+1}-h_{r,s-1}=-r+\frac{ps}{q}$ for $r,s\in\mathbb{Z}_{\geq 1}$, where we may assume $1\leq r\leq p$ and $ps\geq qr$. Thus $h_{r,s+1}-h_{r,s-1}\in\mathbb{Z}$ if and only if $q\mid s$, and these two conformal weights are equal if and only if $h_{r,s}=h_{p,q}$. Similar considerations hold for the conformal weights $h_{r\pm 1,s}$. 
Combining Proposition \ref{prop:surj_intw_op} with this discussion, we can now prove:
%
%
\begin{Proposition}\label{prop:lowest_weight_spaces}
Suppose there is a surjective intertwining operator of type $\binom{\mathcal{W}}{\mathcal{K}_{1,2}\,\widetilde{\mathcal{V}}_{r,s}}$ where $\mathcal{W}$ is non-zero and $\widetilde{\mathcal{V}}_{r,s}$ is some non-zero quotient of $\mathcal{V}_{r,s}$ for some $r,s\in\mathbb{Z}_{\geq 1}$. Then:
\begin{enumerate}
\item If $q\nmid s$, then $\dim\mathcal{W}(0)\leq 2$ and $L_0$ acts semisimply on $\mathcal{W}(0)$ with eigenvalue(s) $h_{r,s\pm 1}$.

\item If $q\mid s$ but $h_{r,s}\neq h_{p,q}$, then $\dim\mathcal{W}(0)=1$ and $L_0$ acts on $\mathcal{W}(0)$ by either the scalar $h_{r,s-1}$ or the scalar $h_{r,s+1}$.

\item If $h_{r,s}=h_{p,q}$, then $\dim\mathcal{W}(0)\leq 2$ and $L_0$ acts on $\mathcal{W}(0)$ by an indecomposable Jordan block with (generalized) eigenvalue $h_{p,q-1}=h_{p,q+1}$.
\end{enumerate}
Similarly, if there is a surjective intertwining operator of type $\binom{\mathcal{W}}{\mathcal{K}_{2,1}\,\widetilde{\mathcal{V}}_{r,s}}$, then:
\begin{enumerate}
\item If $p\nmid r$, then $\dim\mathcal{W}(0)\leq 2$ and $L_0$ acts semisimply on $\mathcal{W}(0)$ with eigenvalue(s) $h_{r\pm 1,s}$.

\item If $p\mid r$ but $h_{r,s}\neq h_{p,q}$, then $\dim\mathcal{W}(0)=1$ and $L_0$ acts on $\mathcal{W}(0)$ by either the scalar $h_{r-1,s}$ or the scalar $h_{r+1,s}$.

\item If $h_{r,s}=h_{p,q}$, then $\dim\mathcal{W}(0)\leq 2$ and $L_0$ acts on $\mathcal{W}(0)$ by an indecomposable Jordan block with (generalized) eigenvalue $h_{p-1,q}=h_{p+1,q}$.
\end{enumerate}
\end{Proposition}
\begin{proof}
By $c_{p,q}=c_{q,p}$ symmetry, it is enough to consider the case of a surjective intertwining operator of type $\binom{\mathcal{W}}{\mathcal{K}_{1,2}\,\widetilde{\mathcal{V}}_{r,s}}$. In cases (1) and (2), where $h_{r,s+1}\neq h_{r,s-1}$, Proposition \ref{prop:surj_intw_op} and \eqref{eqn:bimodules} imply that $\mathcal{W}(0)$ is a quotient of 
\begin{equation*}
\mathbb{C}[x]/\langle(x-h_{r,s-1})(x-h_{r,s+1})\rangle\cong\mathbb{C}[x]/\langle x-h_{r,s-1}\rangle\oplus\mathbb{C}[x]/\langle x-h_{r,s+1}\rangle,
\end{equation*}
where $L_0$ acts as $x$. Thus in these cases, $\dim\mathcal{W}(0)\leq 2$ and $L_0$ is diagonalizable on $\mathcal{W}(0)$ with eigenvalues $h_{r,s-1}$ and/or $h_{r,s+1}$. In case (2), where $\vert h_{r,s+1}-h_{r,s-1}\vert\in\mathbb{Z}_{\geq 1}$, our $\mathbb{Z}_{\geq 0}$-grading convention for $\mathcal{W}$ forces $\dim\mathcal{W}(0)=1$. Indeed, assuming without loss of generality that $h_{r,s-1}<h_{r,s+1}$, then either $\mathcal{W}(0)$ contains an $L_0$-eigenvector of conformal weight $h_{r,s-1}$ or $\mathcal{W}(0)$ is spanned by a single $L_0$-eigenvector of conformal weight $h_{r,s+1}$. Clearly $\dim\mathcal{W}(0)=1$ in the second case, and $\dim\mathcal{W}(0)=1$ in the first case as well because then any $L_0$-eigenvector of conformal weight $h_{r,s+1}$ would be contained in $\mathcal{W}(h_{r,s+1}-h_{r,s-1})$. 

In case (3), $\mathcal{W}(0)$ is a quotient of $\mathbb{C}[x]/\langle (x-h_{p,q+1})^2\rangle$, where $L_0$ again acts by $x$. Thus if $\dim\mathcal{W}(0)=1$, then $\mathcal{W}(0)$ is spanned by an $L_0$-eigenvector of conformal weight $h_{p,q+1}$, while if $\dim\mathcal{W}(0)=2$, then $\mathcal{W}(0)\cong\mathbb{C}[x]/\langle (x-h_{p,q+1})^2\rangle$ is spanned by the generalized $L_0$-eigenvector $1$ and the $L_0$-eigenvector $x-h_{p,q+1}$. Note that $L_0$ cannot be diagonalizable if $\dim\mathcal{W}(0)=2$ because $\mathbb{C}[x]/\langle (x-h_{p,q+1})^2\rangle$ is not isomorphic to $\mathbb{C}[x]/\langle x-h_{p,q+1}\rangle\oplus \mathbb{C}[x]/\langle x-h_{p,q+1}\rangle$.
\end{proof}

Proposition \ref{prop:lowest_weight_spaces} applies in particular when $\mathcal{W}=\mathcal{K}_{1,2}\boxtimes\widetilde{\mathcal{V}}_{r,s}$ or $\mathcal{W}=\mathcal{K}_{2,1}\boxtimes\widetilde{\mathcal{V}}_{r,s}$. Thus the Zhu algebra provides strong constraints (essentially, upper bounds) on fusion tensor products in $\mathcal{O}_{c_{p,q}}$. On the other hand, the existence of a $\mathbb{C}[x]$-module homomorphism $f: A(W_1)\otimes_{\mathbb{C}[x]} W_2(0)\rightarrow W_3(0)$ for $V_{c_{p,q}}$-modules $W_1$, $W_2$, and $W_3$ does not necessarily imply the existence of an intertwining operator $\mathcal{Y}$ of type $\binom{W_3}{W_1\,W_2}$ such that $\pi(\mathcal{Y})=f$, that is, the Zhu algebra is less useful for providing lower bounds on fusion tensor products. There is, however, the following consequence of \cite[Theorem 2.11]{Li-intw-ops}:
\begin{Theorem}\label{thm:Li's-thm}
Suppose $\mathcal{W}_1$ is a $V_{c_{p,q}}$-module, $\mathcal{V}_2$ is a Virasoro Verma module at central charge $c_{p,q}$, and $\mathcal{V}_3'$ is the contragredient of a Virasoro Verma module at central charge $c_{p,q}$. Then $\mathcal{Y}\mapsto\pi(\mathcal{Y})$ is a linear isomorphism between the space of intertwining operators of type $\binom{\mathcal{V}_3'}{\mathcal{W}_1\,\mathcal{V}_2}$ and $\mathrm{Hom}_{\mathbb{C}[x]}(A(\mathcal{W}_1)\otimes_{\mathbb{C}[x]} \mathcal{V}_2(0),\mathcal{V}_3'(0))$.
\end{Theorem}

\section{Construction of intertwining operators}\label{sec:intw_op_constructions}

The results in Section \ref{subsec:Zhu_theory}, especially Proposition \ref{prop:lowest_weight_spaces} and Theorem \ref{thm:Li's-thm}, indicate that the Zhu algebra provides information about Virasoro intertwining operators of type $\binom{W_3}{W_1\,W_2}$ when $W_2$ is (a quotient of) a Verma module and/or $W_3$ is (a submodule of) the contragredient of a Verma module. But the Kac modules $\mathcal{K}_{r,s}$ at central charge $c_{p,q}$ that we want to consider do not necessarily satisfy these conditions (recall Remark \ref{rem:Krs_Verma_quot}). Thus in this section, we give some intertwining operator constructions that do not rely on the Zhu algebra.

First, in Section \ref{subsec:Verma_intw_ops}, we will construct intertwining operators of type $\binom{\mathcal{V}_2}{\mathcal{K}_{1,2}\,\mathcal{V}_1}$, where $\mathcal{V}_1$ and $\mathcal{V}_2$ are Verma modules. Although Verma modules are not objects of the tensor category $\mathcal{O}_{c_{p,q}}$, we will be able to show that in some cases, intertwining operators involving Verma modules descend to well-defined intertwining operators of type $\binom{\mathcal{W}_2}{\mathcal{K}_{1,2}\,\mathcal{W}_1}$ where $\mathcal{W}_1$ and $\mathcal{W}_2$ are finite-length Verma module quotients. Thus in these cases we still obtain information about fusion tensor products in $\mathcal{O}_{c_{p,q}}$. Next, in Section \ref{subsec:Fock_intw}, we consider intertwining operators among Feigin-Fuchs modules. Although Feigin-Fuchs modules are also not objects of $\mathcal{O}_{c_{p,q}}$, we can restrict an intertwining operator among Feigin-Fuchs modules of type $\binom{\mathcal{F}_{r+r'-1,s+s'-1}}{\mathcal{F}_{r,s}\,\,\mathcal{F}_{r',s'}}$, for $r,r',s,s'\in\mathbb{Z}_{\geq 1}$, to the Kac submodules $\mathcal{K}_{r,s}$ and $\mathcal{K}_{r',s'}$. After careful analysis, we will show that the image of such a restriction is in fact the Kac submodule $\mathcal{K}_{r+r'-1,s+s'-1}\subseteq\mathcal{F}_{r+r'-1,s+s'-1}$, and thus we will obtain a surjective map $\mathcal{K}_{r,s}\boxtimes\mathcal{K}_{r',s'}\rightarrow\mathcal{K}_{r+r'-1,s+s'-1}$ in $\mathcal{O}_{c_{p,q}}$.

\subsection{Intertwining operators involving Verma modules}\label{subsec:Verma_intw_ops}

Here we construct some intertwining operators involving Virasoro Verma modules at any central charge, following the ideas and methods of \cite{Li-hom-functor, MY-sl2-intw-op, McR-gen-intw-op}; see also \cite[Section 7.3.1]{Fran}. We fix a central charge $c=13-6t-6t^{-1}$ for some $t\in\mathbb{C}^\times$ and consider the Virasoro Verma module $\mathcal{V}_{1,2}$ generated by a highest-weight vector $v_{1,2}$ of conformal weight $h_{1,2}=-\frac{1}{2}+\frac{3}{4 t}$. The Verma module $\mathcal{V}_{1,2}$ also contains a singular vector
\begin{equation}\label{eqn:v12_singular}
\widetilde{v}_{1,2}=\left( L_{-1}^2-t^{-1} L_{-2}\right)v_{1,2}
\end{equation}
of conformal weight $h_{1,2}+2$, so we can define the Kac module $\mathcal{K}_{1,2}$ at any central charge as the quotient $\mathcal{K}_{1,2}=\mathcal{V}_{1,2}/\langle\widetilde{v}_{1,2}\rangle$. Our goal is to construct intertwining operators of type $\binom{\mathcal{V}_2}{\mathcal{K}_{1,2}\,\,\mathcal{V}_1}$, where $\mathcal{V}_1$ and $\mathcal{V}_2$ are any two Virasoro Verma modules at central charge $c$.

To construct intertwining operators, we will use Li's ``generalized nuclear democracy theorem'' as specialized to Virasoro vertex operator algebras in \cite[Proposition 4.16]{Li-hom-functor}:
\begin{Theorem}\label{thm:Vir_gen_nucl_demo}
For $h\in\mathbb{C}$, let $\mathcal{V}_h$ be the Virasoro Verma module of central charge $c$ generated by a highest-weight vector $v_h$ of conformal weight $h$, and let $W_1$, $W_2$ be modules for the universal Virasoro vertex operator algebra $V_c$. If 
\begin{equation*}
\Phi(z): W_1\longrightarrow W_2\lbrace z\rbrace
\end{equation*}
is a linear map such that for all $m\in\mathbb{Z}$,
\begin{equation}\label{eqn:gen_intw_op}
[L_m,\Phi(z)] = z^m\left((m+1)h+z\frac{d}{dz}\right)\Phi(z),
\end{equation}
then there is a unique intertwining operator $\mathcal{Y}$ of type $\binom{W_2}{\mathcal{V}_h\,W_1}$ such that $\mathcal{Y}(v_{1,2},z)=\Phi(z)$.
\end{Theorem}

The operator $\Phi(z)$ in the statement of the theorem is called a ``generalized intertwining
operator'' in \cite[Definition 4.1]{Li-hom-functor}, but it is a primary field in the sense of conformal field theory. Note that \cite[Proposition 4.16]{Li-hom-functor} is stated for modules of the simple vertex operator algebra $L_{c_{p,q}}$, but based on the discussion preceding its statement, it applies to modules for the universal Virasoro vertex operator algebra at any central charge as long as we replace the simple module $L(c,h)$ in \cite[Proposition 4.16]{Li-hom-functor} with the Verma module $\mathcal{V}_h$.

Now we can state and prove the main result of this subsection:
\begin{Theorem}\label{thm:gen_Verma_intw_ops}
Let $\mathcal{V}_1$ and $\mathcal{V}_2$ be Virasoro Verma modules of central charge $c=13-6t-6t^{-1}$ and lowest conformal weights $h_1$ and $h_2$, respectively. If
\begin{equation*}
h_2=h_1+\frac{1}{4t}-\frac{1}{2t}\sqrt{4th_1+(t-1)^2}
\end{equation*}
for some choice of square root such that $\frac{1}{t}\sqrt{4th_1+(t-1)^2}\not\in\mathbb{Z}_{\geq 1}$, then there is a unique up to scale surjective intertwining operator of type $\binom{\mathcal{V}_2}{\mathcal{K}_{1,2}\,\mathcal{V}_1}$.
\end{Theorem}
\begin{proof}
We address the uniqueness first. Thus let $\mathcal{Y}$ be a surjective intertwining operator of type $\binom{\mathcal{V}_2}{\mathcal{K}_{1,2}\,\mathcal{V}_1}$, where the lowest conformal weights of $\mathcal{V}_1$ and $\mathcal{V}_2$ satisfy the relation specified in the theorem statement. Let $v_{1,2}$ and $v_1$ be generators of minimal conformal weight in $\mathcal{K}_{1,2}$ and $\mathcal{V}_1$, respectively. Then it is straightforward from the commutator formula \eqref{eqn:Vir_comm_form}, the iterate formula \eqref{eqn:Vir_it_form}, and the $L_{-1}$-derivative property for intertwining operators that $\mathcal{Y}$ is completely determined by the series $\mathcal{Y}(v_{1,2},z)v_1\in \mathcal{V}_2\lbrace z\rbrace$; this also follows from the proof of \cite[Proposition 11.9]{DL}.

Thus we need to show that $\mathcal{Y}(v_{1,2},z)v_1$ is uniquely determined up to scale. In fact, because the singular vector $\widetilde{v}_{1,2}$ vanishes in $\mathcal{K}_{1,2}$, we have
\begin{equation}\label{eqn:Y_zero_on_gen}
\mathcal{Y}(\widetilde{v}_{1,2},z)v_1 = 0,
\end{equation}
and then the formula \eqref{eqn:v12_singular} for $\widetilde{v}_{1,2}$, the $L_{-1}$-derivative property, and the iterate and commutator formulas for intertwining operators imply that \eqref{eqn:Y_zero_on_gen} holds if and only if the second-order differential equation
\begin{equation}\label{eqn:diff_eqn}
\left(t\cdot z^2\frac{d^2}{dz^2}+z\frac{d}{dz}-h_{1}\right)\mathcal{Y}(v_{1,2},z)v_1=\sum_{i= 1}^\infty L_{-i}z^{i}\mathcal{Y}(v_{1,2},z)v_1,
\end{equation}
is satisfied. This equation can also be derived using the methods of \cite[Section 7.3.1]{Fran}; see in particular the differential equation \cite[Equation 7.47]{Fran} for correlation functions involving a primary field of conformal weight $h_{1,2}$.


As in \cite[Section 7.3.1]{Fran}, the differential equation \eqref{eqn:diff_eqn} constrains the possibilities for the conformal weights $h_1$ and $h_2$ of $\mathcal{V}_1$ and $\mathcal{V}_2$. Specifically, we write $\mathcal{Y}(v_{1,2}, z)v_1=\sum_{k=0}^\infty \phi_k\,z^{h+k}$ where $h=h_2-h_{1,2}-h_1$ and $\phi_k\in\mathcal{V}_2$ is a vector of conformal weight $h_2+k$ (see Equations 5.4.14, 5.4.29, and 5.4.30 in \cite{FHL}).  Inserting this ansatz into \eqref{eqn:diff_eqn} and comparing coefficients of powers of $z$ on both sides, we get the recursive relation
\begin{equation}\label{eqn:diff_eqn_recursion}
\left[ tk^2+(t(2h-1)+1)k+(th(h-1)+h-h_1)\right]\phi_k =\sum_{i=1}^k L_{-i}\phi_{k-i}
\end{equation}
for all $k\in\mathbb{Z}_{\geq 0}$. For $k=0$, the sum on the right side is empty, so the recursion has a non-zero solution only if the term in brackets on the left side vanishes at $k=0$. For a given $h_1$, this happens for at most two values of $h_2$, namely
\begin{equation*}
h_2 =h_1+\frac{1}{4t}-\frac{1}{2t}\sqrt{4th_1+(t-1)^2}
\end{equation*}
for both choices of square root. Given our assumption that $h_2$ indeed equals the above for some choice of square root, \eqref{eqn:diff_eqn_recursion} simplifies to
\begin{equation}\label{eqn:diff_eqn_recursion_simple}
k\left(tk- \sqrt{4th_1+(t-1)^2}\right)\phi_k=\sum_{i=1}^k L_{-i}\phi_{k-i}.
\end{equation}
Thus the condition in the theorem statement that $\frac{1}{t}\sqrt{4th_1+(t-1)^2}\notin\mathbb{Z}_{\geq 1}$ implies that the left side of \eqref{eqn:diff_eqn_recursion_simple} vanishes only for $k=0$, and thus that $\phi_k$ for $k>0$ is completely determined by a choice of $\phi_0$ in the one-dimensional space $(\mathcal{V}_2)_{[h_2]}$. This proves that $\mathcal{Y}(v_{1,2},z)v_1$, and thus also the full intertwining operator $\mathcal{Y}$, is uniquely determined up to scale.

Conversely, to show that a surjective intertwining operator $\mathcal{Y}$ of type $\binom{\mathcal{V}_2}{\mathcal{K}_{1,2}\,\mathcal{V}_1}$ actually exists, we will first use Theorem \ref{thm:Vir_gen_nucl_demo} to obtain a surjective intertwining operator of type $\binom{\mathcal{V}_2}{\mathcal{V}_{1,2}\,\mathcal{V}_1}$. Thus we need to construct a linear map $\Phi(z): \mathcal{V}_1\rightarrow\mathcal{V}_2\lbrace z\rbrace$ which satisfies \eqref{eqn:gen_intw_op} for $h=h_{1,2}$. If such a $\Phi(z)$ exists, the resulting intertwining operator $\mathcal{Y}$ guaranteed by Theorem \ref{thm:Vir_gen_nucl_demo} will satisfy $\mathcal{Y}(v_{1,2},z)v_1=\Phi(z)v_1$. So because we will want $\mathcal{Y}$ to descend to a well-defined intertwining operator of type $\binom{\mathcal{V}_2}{\mathcal{K}_{1,2}\,\mathcal{V}_1}$, we will need \eqref{eqn:Y_zero_on_gen} to hold, and thus we are motivated to define $\Phi(z)v_1$ as a formal series solution to the differential equation \eqref{eqn:diff_eqn}.

Specifically, we define $\Phi(z)v_1=\sum_{k=0}^\infty\phi_k\,z^{h+k}\in\mathcal{V}_2\lbrace z\rbrace$, where $h=h_2-h_{1,2}-h_1$ again, to be a non-zero series whose coefficients satisfy \eqref{eqn:diff_eqn_recursion_simple} for all $k\geq 0$. The assumptions on $h_1$ and $h_2$ in the statement of the theorem guarantee that such a series exists and is uniquely determined by a choice of non-zero $\phi_0\in\mathcal{V}_2$ of conformal weight $h_2$. This determines how $\Phi(z)$ acts on $v_1$, but we still need to define how $\Phi(z)$ acts on higher-degree vectors in $\mathcal{V}_1$. For this, we can use the conformal primary condition \eqref{eqn:gen_intw_op} and the fact that $\mathcal{V}_1$ is spanned by vectors of the form $L_{-n_1}\cdots L_{-n_k} v_1$ where $n_1,\ldots, n_k>0$. Specifically, assuming we have already defined $\Phi(z)w$ where $w$ is some such monomial in the $L_{-n}$'s applied to $v_1$, we recursively define
\begin{equation}\label{eqn:Phi(z)_recursion}
\Phi(z)L_{-m} w = L_{-m}\Phi(z)w-z^{-m}\left((-m+1)h_{1,2}+z\frac{d}{dz}\right)\Phi(z)w
\end{equation}
for any $m>0$.

At first, \eqref{eqn:Phi(z)_recursion} defines a map from the tensor algebra freely generated by $L_{-1}, L_{-2}, L_{-3},\ldots$ to $\mathcal{V}_2\lbrace z\rbrace$. To show that $\Phi(z)$ is actually well defined as a linear map on $\mathcal{V}_1=U(\mathfrak{Vir}_-)\cdot v_1$, we need to check the Virasoro algebra commutation relations, that is,
\begin{equation*}
\Phi(z) L_{-n_1} \cdots [L_{-n_{j}}, L_{-n_{j+1}}] \cdots L_{-n_k} v_1= (n_{j+1}-n_j) \Phi(z)L_{-n_1} \cdots L_{-n_j-n_{j+1}} \cdots L_{-n_k} v_1
\end{equation*}
for $n_1, \ldots n_j,n_{j+1},\ldots, n_k>0$. We check the $j=1$ case first: writing $L_{-n_1}=L_{-m}$, $L_{-n_2}=L_{-n}$, and $L_{-n_3}\cdots L_{-n_k} v_1=w$, we have
\begin{align*}
\Phi(z) & [L_{-m}, L_{-n}]w = \Phi(z)L_{-m} L_{-n}w-\Phi(z) L_{-n} L_{-m}w\nonumber\\
& =\left(L_{-m}+(m-1)h_{1, 2}z^{-m}-z^{-m+1}\frac{d}{dz}\right)\Phi(z) L_{-n}w \nonumber\\
&\qquad\quad -\left(L_{-n}+ (n-1)h_{1, 2}z^{-n}-z^{-n+1}\frac{d}{dz}\right)\Phi(z)L_{-m}w\nonumber\\
& =\left[L_{-m}+(m-1)h_{1, 2}z^{-m}-z^{-m+1}\frac{d}{dz}, L_{-n}+(n-1)h_{1, 2}z^{-n}-z^{-n+1}\frac{d}{dz}\right]\Phi(z)w\nonumber\\
& =[L_{-m},L_{-n}]\Phi(z)w-(m-1)h_{1,2}\left[z^{-m},z^{-n+1}\frac{d}{dz}\right]\Phi(z)w\nonumber\\
&\qquad\quad -(n-1)h_{1,2}\left[z^{-m+1}\frac{d}{dz}, z^{-n}\right]\Phi(z)w +\left[z^{-m+1}\frac{d}{dz}, z^{-n+1}\frac{d}{dz}\right]\Phi(z)w\nonumber\\
& =\left((n-m)L_{-m-n} -\left(m(m-1)-n(n-1)\right)h_{1,2}z^{-m-n}+(m-n)z^{-m-n+1}\frac{d}{dz}\right)\Phi(z)w\nonumber\\
& = (n-m)\left(L_{-m-n}+(m+n-1)h_{1,2}z^{-m-n}-z^{-m-n+1}\frac{d}{dz}\right)\Phi(z)w\nonumber\\
& = (n-m)\Phi(z) L_{-m-n}w
\end{align*}
as required. The $j>0$ case then follows easily by induction on $j$ and the definition \eqref{eqn:Phi(z)_recursion}.

We have now shown that  $\Phi(z): \mathcal{V}_1\rightarrow\mathcal{V}_2\lbrace z\rbrace$ is well defined, and by definition $\Phi(z)$ satisfies the conformal primary condition \eqref{eqn:gen_intw_op} for $m<0$, but it is left to exhibit that $\Phi(z)$ satisfies \eqref{eqn:gen_intw_op} for $m\geq 0$. Applying both sides of \eqref{eqn:gen_intw_op} to $v_1$, we see that we first need to check
\begin{equation}\label{eqn:gen_intw_op_m_non-neg}
L_m\Phi(z)v_1=\left\lbrace\begin{array}{lll}
\left(h_1+h_{1,2}+z\frac{d}{dz}\right)\Phi(z)v_1 & \text{if} & m=0\\
z^m\left((m+1)h_{1,2}+z\frac{d}{dz}\right)\Phi(z)v_1 & \text{if} & m>0\\
\end{array}\right. .
\end{equation}
Recall that $\Phi(z)v_1=\sum_{k=0}^\infty \phi_k\,z^{h_2-h_{1,2}-h_1+k}$ where $\phi_k\in\mathcal{V}_2$ has conformal weight $h_2+k$ and satisfies \eqref{eqn:diff_eqn_recursion_simple}. Thus comparing coefficients of powers of $z$ on both sides of \eqref{eqn:gen_intw_op_m_non-neg}, we need to check $L_m\phi_k = 0 $ for $k < m$ and
\begin{equation}\label{eqn:gen_intw_op_components}
L_m\phi_k=\left\lbrace\begin{array}{lll}
(h_2+k)\phi_k & \text{if} & m=0\\
(k+m(h_{1,2}-1)+h_2-h_1)\phi_{k-m} & \text{if} & m>0
\end{array}\right.
\end{equation}
for $k\geq m$. The $k<m$ and $m=0$ cases are clear because $\phi_k$ has conformal weight $h_2+k$ and because $h_2$ is the lowest conformal weight of $\mathcal{V}_2$. In particular, the $k=0$ case holds.

We now prove the $k\geq 1$ case by induction on $k$. It is enough to prove the second case of \eqref{eqn:gen_intw_op_components}, assuming $1\leq m\leq k$. The calculation is analogous to what was done for affine Lie algebras in  \cite[Theorem 3.6]{McR-gen-intw-op}. We will use the recursion formula \eqref{eqn:diff_eqn_recursion_simple}, which we rewrite as
\begin{equation*}
tk\left(k+2(h_2-h_1)-\frac{1}{2t}\right)\phi_k=\sum_{i=1}^k L_{-i}\phi_{k-i},
\end{equation*}
as well as the induction hypothesis:
\begin{align*}
tk & \left(k+2(h_2-h_1)-\frac{1}{2t}\right) L_m \phi_k =\sum_{i=1}^k L_m L_{-i}\phi_{k-i}\nonumber\\
& = \sum_{i=1}^k\left(L_{-i}L_m\phi_{k-i} +(m+i)L_{m-i}\phi_{k-i}\right)+\frac{m^3-m}{12}\left(13-6t-\frac{6}{t}\right)\phi_{k-m}\nonumber\\
& =\sum_{i=1}^{k-m} \left(k-i+m(h_{1,2}-1)+h_2-h_1\right)L_{-i}\phi_{k-m-i}+\sum_{i=1}^{m-1} (m+i)L_{m-i}\phi_{k-i}\nonumber\\
&\qquad +2mL_0\phi_{k-m}+\sum_{i=m+1}^k (m+i)L_{m-i}\phi_{k-i}+\frac{m^3-m}{12}\left(13-6t-\frac{6}{t}\right)\phi_{k-m}\nonumber\\
& =\left(k+m(h_{1,2}-1)+h_2-h_1\right)t(k-m)\left(k-m+2(h_2-h_1)-\frac{1}{2t}\right)\phi_{k-m}\nonumber\\
&\qquad-\sum_{i=1}^{k-m} i L_{-i}\phi_{k-m-i} +\sum_{i=1}^{m-1} (m+i)\left(k-i+(m-i)(h_{1,2}-1)+h_2-h_1\right)\phi_{k-m}\nonumber\\
&\qquad +2m(k-m+h_2)\phi_{k-m}+\sum_{i=1}^{k-m} (2m+i)L_{-i}\phi_{k-m-i}+\frac{m^3-m}{12}\left(13-6t-\frac{6}{t}\right)\phi_{k-m}\nonumber\\
& =\left(k+m(h_{1,2}-1)+h_2-h_1\right)t(k-m)\left(k-m+2(h_2-h_1)-\frac{1}{2t}\right)\phi_{k-m}\nonumber\\
&\qquad -\frac{1}{6} m(m-1)(2m-1)h_{1,2}\phi_{k-m}+\frac{1}{2}m(m-1)(k-m+h_2-h_1)\phi_{k-m}\nonumber\\
&\qquad +m(m-1)(k+m(h_{1,2}-1)+h_2-h_1)\phi_{k-m}+2m(k-m+h_2)\phi_{k-m}\nonumber\\
&\qquad +2mt(k-m)\left(k-m+2(h_2-h_1)-\frac{1}{2t}\right)\phi_{k-m}+\frac{m^3-m}{12}\left(13-6t-\frac{6}{t}\right)\phi_{k-m}
\end{align*}
At this point, a straightforward but tedious calculation using the definition of $h_{1,2}$ in terms of $t$, and the definition of $h_2$ in terms of $h_1$ and $t$, shows that the right side here is
\begin{equation*}
tk\left(k+2(h_2-h_1)-\frac{1}{2t}\right)\left(k+m(h_{1,2}-1)+h_2-h_1\right)\phi_{k-m}.
\end{equation*}
This finishes the proof of \eqref{eqn:gen_intw_op_components} and thus of \eqref{eqn:gen_intw_op_m_non-neg} for $m\geq 0$.

Now we know that the conformal primary condition \eqref{eqn:gen_intw_op} holds for $m\geq 0$ when both sides of the equation are applied to $v_1$. Next, we show that if \eqref{eqn:gen_intw_op} holds for $m\geq 0$ when both sides are applied to some $w\in\mathcal{V}_1$, then it also holds when both sides are applied to $L_{-n} w$ for any $n\in\mathbb{Z}_{\geq 1}$.  Indeed, using this assumption for $w$ together with the definition \eqref{eqn:Phi(z)_recursion} of $\Phi(z)$,
\begin{align*}
[L_{m}, \Phi(z)] & L_{-n}w  = L_{m}\left( L_{-n}+(n-1)h_{1, 2}z^{-n}-z^{-n+1}\frac{d}{dz}\right)\Phi(z)w -\Phi(z)L_m L_{-n} w\nonumber\\
& =\left(L_{-n}+(n-1)h_{1,2}z^{-n}-z^{-n+1}\frac{d}{dz}\right)L_m\Phi(z)w+[L_m, L_{-n}]\Phi(z)w\nonumber\\
&\qquad -\Phi(z)L_{-n}L_m w-\Phi(z)[L_m,L_{-n}]w\nonumber\\
& =\left( L_{-n}+(n-1)h_{1, 2}z^{-n}-z^{-n+1}\frac{d}{dz}\right)[L_m,\Phi(z)]w +(m+n)[L_{m-n},\Phi(z)]w\nonumber\\
& =\left( L_{-n}+(n-1)h_{1, 2}z^{-n}-z^{-n+1}\frac{d}{dz}\right)z^m\left((m+1)h_{1,2}+z\frac{d}{dz}\right)\Phi(z)w\nonumber\\
&\qquad +(m+n)z^{m-n}\left((m-n+1)h_{1,2}+z\frac{d}{dz}\right)\Phi(z)w\nonumber\\
& =z^m\left((m+1)h_{1,2}+z\frac{d}{dz}\right)\left(L_{-n}+(n-1)h_{1,2}z^{-n}-z^{-n+1}\frac{d}{dz}\right)\Phi(z)w\nonumber\\
&\qquad +\left[(n-1)h_{1,2}z^{-n}-z^{-n+1}\frac{d}{dz}, (m+1)h_{1,2}z^m+z^{m+1}\frac{d}{dz}\right]\Phi(z)w\nonumber\\
&\qquad +(m+n)(m-n+1)h_{1,2}z^{m-n}\Phi(z)w+(m+n)z^{m-n+1}\frac{d}{dz}\Phi(z)w\nonumber\\
& =z^m\left((m+1)h_{1,2}+z\frac{d}{dz}\right)\Phi(z)L_{-n}w,
\end{align*}
as required. This calculation proves \eqref{eqn:gen_intw_op} for all $m\geq 0$ (and thus for all $m\in\mathbb{Z}$) because $\mathcal{V}_1$ is spanned by vectors of the form $L_{-n_1}\cdots L_{-n_k} v_1$.

We now have a linear map $\Phi(z): \mathcal{V}_1\rightarrow\mathcal{V}_2\lbrace z\rbrace$ which satisfies \eqref{eqn:gen_intw_op} and is determined by a non-zero $\phi_0\in\mathcal{V}_2$ of conformal weight $h_2$. Thus by Theorem \ref{thm:Vir_gen_nucl_demo}, there is a non-zero intertwining operator $\mathcal{Y}$ of type $\binom{\mathcal{V}_2}{\mathcal{V}_{1,2}\,\mathcal{V}_1}$ such that $\Phi(z)=\mathcal{Y}(v_{1,2},z)$. This intertwining operator is surjective because $\phi_0$ generates $\mathcal{V}_2$ and is a coefficient of $\mathcal{Y}(v_{1,2},z)v_1$. Finally, to show that $\mathcal{Y}$ descends to a well-defined (and surjective) intertwining operator of type $\binom{\mathcal{V}_2}{\mathcal{K}_{1,2}\,\mathcal{V}_1}$, recall that $\mathcal{K}_{1,2}=\mathcal{V}_{1,2}/\langle\tilde{v}_{1,2}\rangle$, so we need to show $\mathcal{Y}\vert_{\langle\widetilde{v}_{1,2}\rangle\otimes\mathcal{V}_1}=0$. Indeed, $\mathcal{Y}(\widetilde{v}_{1,2},z)v_1=0$ since by construction $\mathcal{Y}(v_{1,2},z)v_1$ satisfies the differential equation \eqref{eqn:diff_eqn}. Then because $v_1$ generates $\mathcal{V}_1$,  the proof of \cite[Proposition 11.9]{DL} shows $\mathcal{Y}\vert_{\langle\widetilde{v}_{1,2}\rangle\otimes\mathcal{V}_1}=0$.
\end{proof}

We now take central charge $c=c_{p,q}$ where $p,q\geq 2$ and $\mathrm{gcd}(p,q)=1$, that is, $t=\frac{q}{p}$. Using the theorem we can construct intertwining operators involving the Kac modules $\mathcal{K}_{r,s}$ of Section \ref{subsec:FF_and_K_modules} for $r\leq p$ or $s\leq q$ (which are Verma module quotients; recall Remark \ref{rem:Krs_Verma_quot}). More generally, for all $r,s\in\mathbb{Z}_{\geq 0}$, define $\widetilde{\mathcal{K}}_{r,s}$ to be the Verma module quotient $\mathcal{V}_{r,s}/\langle \widetilde{v}_{r,s}\rangle$ where $\widetilde{v}_{r,s}$ is the unique (up to scale) singular vector of conformal weight $h_{r,s}+rs$. Thus $\mathcal{K}_{r,s}\cong\widetilde{\mathcal{K}}_{r,s}$ when $r\leq p$ or $s\leq q$, but not when $r>p$ and $s>q$.
\begin{Theorem}\label{thm:K_tilde_intw_ops}
Let $r,s\in\mathbb{Z}_{\geq 1}$. There is a surjective intertwining operator of type $\binom{\widetilde{\mathcal{K}}_{r,s+1}}{\mathcal{K}_{1,2}\,\widetilde{\mathcal{K}}_{r,s}}$ if $q\nmid s$, or if $q\mid s$ and $1\leq r\leq p-1$, and there is a surjective intertwining operator of type $\binom{\widetilde{\mathcal{K}}_{r,s-1}}{\mathcal{K}_{1,2}\,\widetilde{\mathcal{K}}_{r,s}}$ if $q\nmid s$, or if $q\mid s$ and $ps\leq qr$.
\end{Theorem}
\begin{proof}
Taking $t=\frac{q}{p}$ and $h_1=h_{r,s}$ in Theorem \ref{thm:gen_Verma_intw_ops}, yields a surjective intertwining operator of type $\binom{\mathcal{V}_{r,s+1}}{\mathcal{K}_{1,2}\,\mathcal{V}_{r,s}}$ if $h_{r,s-1}-h_{r,s+1}\notin\mathbb{Z}_{\geq 1}$, and of type $\binom{\mathcal{V}_{r,s-1}}{\mathcal{K}_{1,2}\,\mathcal{V}_{r,s}}$ if $h_{r,s+1}-h_{r,s-1}\notin\mathbb{Z}_{\geq 1}$. Since
\begin{equation}\label{eqn:conf_wt_diff}
h_{r,s+1}-h_{r,s-1}=-r+\frac{ps}{q},
\end{equation}
this means there is a surjective intertwining operator of type $\binom{\mathcal{V}_{r,s+1}}{\mathcal{K}_{1,2}\,\mathcal{V}_{r,s}}$ when $q\nmid s$, or when $q\mid s$ and $ps\geq qr$, and there is a surjective intertwining operator of type $\binom{\mathcal{V}_{r,s-1}}{\mathcal{K}_{1,2}\,\mathcal{V}_{r,s}}$ when $q\nmid s$, or when $q\mid s$ and $ps\leq qr$. Whenever such intertwining operators exist, we get surjective intertwining operators of type $\binom{\widetilde{\mathcal{K}}_{r,s\pm1}}{\mathcal{K}_{1,2}\,\mathcal{V}_{r,s}}$ by composing with the surjections $\mathcal{V}_{r,s\pm 1}\twoheadrightarrow\widetilde{\mathcal{K}}_{r,s\pm1}$. 

We need to determine when these intertwining operators are well defined on the quotient $\widetilde{\mathcal{K}}_{r,s}=\mathcal{V}_{r,s}/\langle\widetilde{v}_{r,s}\rangle$. Thus suppose $\mathcal{Y}$ is a surjective intertwining operator of type $\binom{\widetilde{\mathcal{K}}_{r,s\pm 1}}{\mathcal{K}_{1,2}\,\mathcal{V}_{r,s}}$. To show that $\mathcal{Y}$ is well defined on $\widetilde{\mathcal{K}}_{r,s}$, we need to show that $\mathcal{Y}\vert_{\mathcal{K}_{1,2}\otimes\langle\widetilde{v}_{r,s}\rangle}=0$. Since $\langle \widetilde{v}_{r,s}\rangle$ is a Verma module with lowest conformal weight $h_{r,s}+rs=h_{-r,s}$, Proposition \ref{prop:lowest_weight_spaces} shows that if $\mathcal{Y}\vert_{\mathcal{K}_{1,2}\otimes\langle\widetilde{v}_{r,s}\rangle}\neq 0$, then the lowest conformal weight of its image in $\widetilde{\mathcal{K}}_{r,s\pm 1}$ is $h_{-r,s\pm 1}$, that is, $\widetilde{\mathcal{K}}_{r,s\pm 1}$ contains a singular vector of conformal weight $h_{-r,s+1}$ or $h_{-r,s-1}$.

First consider the case that $\mathcal{Y}$ is of type $\binom{\widetilde{\mathcal{K}}_{r,s+1}}{\mathcal{K}_{1,2}\,\mathcal{V}_{r,s}}$. Then $\widetilde{\mathcal{K}}_{r,s+1}$ does not contain any singular vector of conformal weight $h_{-r,s+1}$ since by definition $\widetilde{\mathcal{K}}_{r,s+1}$ is the quotient of $\mathcal{V}_{r,s}$ by its unique up to scale singular vector of weight $h_{-r,s+1}$. When $q\nmid s$, then $\widetilde{\mathcal{K}}_{r,s+1}$ also contains no singular vector of weight $h_{-r,s-1}$ since by \eqref{eqn:conf_wt_diff}, $h_{-r,s-1}-h_{-r,s+1}\notin\mathbb{Z}$ in this case. When $s=nq$ for some $n\geq 1$ and $1\leq r\leq p-1$, then the Verma module embedding diagrams imply that $\widetilde{\mathcal{K}}_{r,nq+1}\cong\mathcal{K}_{r,nq+1}$ satisfies a short exact sequence
\begin{equation*}
0\longrightarrow \mathcal{L}_{r,(n+2)q-1}\longrightarrow\mathcal{K}_{r,nq+1}\longrightarrow\mathcal{L}_{r,nq+1}\longrightarrow 0,
\end{equation*}
so $\widetilde{\mathcal{K}}_{r,nq+1}$ does not contain a singular vector of weight $h_{-r,nq-1}=h_{p-r,(n+1)q-1}$ in this case. These cases prove that $\mathcal{Y}$ descends to a well-defined surjective intertwining operator of type $\binom{\widetilde{\mathcal{K}}_{r,s+1}}{\mathcal{K}_{1,2}\,\widetilde{\mathcal{K}}_{r,s}}$  when $q\nmid s$, or when $q\mid s$ and $1\leq r\leq p-1$. (When $q\mid s$ and $r\geq p$, $\widetilde{\mathcal{K}}_{r,s+1}$ does contain a singular vector of conformal weight $h_{-r,s-1}$, so we cannot rule out that $\mathcal{Y}\vert_{\mathcal{K}_{1,2}\otimes\langle\widetilde{v}_{r,s}\rangle}\neq 0$ here.)

Now take $\mathcal{Y}$ of type $\binom{\widetilde{\mathcal{K}}_{r,s-1}}{\mathcal{K}_{1,2}\,\mathcal{V}_{r,s}}$. To show that $\mathcal{Y}\vert_{\mathcal{K}_{1,2}\otimes\langle\widetilde{v}_{r,s}\rangle}=0$, thus yielding a surjective intertwining operator of type $\binom{\widetilde{\mathcal{K}}_{r,s-1}}{\mathcal{K}_{1,2}\,\widetilde{\mathcal{K}}_{r,s}}$, we just need to show that $\widetilde{\mathcal{K}}_{r,s-1}$ does not contain a singular vector of conformal weight $h_{-r,s+ 1}$ (since the singular vector of weight $h_{-r,s-1}$ already vanishes in $\widetilde{\mathcal{K}}_{r,s-1}$). In the case $q\nmid s$, this follows because $h_{-r,s-1}-h_{-r,s+1}\notin\mathbb{Z}$. In the case $s=nq$ for some $n\geq 1$ and $ps\leq qr$, this follows because the Verma module embedding diagrams show that the Verma submodule $\langle \widetilde{v}_{r,nq-1}\rangle\subseteq\mathcal{V}_{r,nq-1}$ contains the singular vector of conformal weight $h_{-r,nq+1}$. This proves the second assertion of the theorem.
\end{proof}

\subsection{Intertwining operators involving Feigin-Fuchs modules}\label{subsec:Fock_intw}

Recall from Section \ref{subsec:FF_and_K_modules} that the Kac modules $\mathcal{K}_{r,s}$ are submodules of the Feigin-Fuchs modules $\mathcal{F}_{r,s}$, and that the latter are simple modules for the Heisenberg vertex operator algebra $\mathcal{F}_0$. Moreover, intertwining operators involving the $\mathcal{F}_{r,s}$ considered as $\mathcal{F}_0$-modules are still intertwining operators when we consider the $\mathcal{F}_{r,s}$ as $\mathfrak{Vir}$-modules. Thus we can construct some intertwining operators involving Kac modules by restricting Heisenberg Fock module intertwining operators.

Intertwining operators among Heisenberg Fock modules are well known (see for example the constructions in \cite{FLM, DL, LL}); they correspond to vertex operators in the free boson conformal field theory \cite[Section 6.3]{Fran}. In particular, for any $\lambda,\mu\in\mathbb{C}$, there is a unique up to scale non-zero (and surjective) $\mathcal{F}_0$-module intertwining operator of type $\binom{\mathcal{F}_{\lambda+\mu}}{\mathcal{F}_\lambda\,\mathcal{F}_\mu}$, determined by
\begin{equation}\label{eqn:Fock_intw_op}
\mathcal{Y}(v_\lambda,z)v_\mu =z^{\lambda\mu}\exp\left(\lambda\sum_{n=1}^\infty\frac{a_{-n}}{n}z^n\right)v_{\lambda+\mu}.
\end{equation}
This intertwining operator satisfies a commutator formula analogous to \eqref{eqn:Vir_comm_form}:
\begin{equation}\label{eqn:Heis_comm_form}
a_n\mathcal{Y}(w,z) =\mathcal{Y}(w,z)a_n +\sum_{i\geq 0} \binom{n}{i} z^{n-i}\mathcal{Y}(a_i w,z)
\end{equation}
for all $w\in\mathcal{F}_\lambda$ and $n\in\mathbb{Z}$. 

In case $\lambda=\lambda_{1,2}$ and $\mu=\lambda_{r,s}$ for some $r,s\in\mathbb{Z}_{\geq 1}$, we have $\lambda+\mu=\lambda_{r,s+1}$. So by restriction, we get a $\mathfrak{Vir}$-module intertwining operator $\mathcal{Y}$ of type $\binom{\mathcal{F}_{r,s+1}}{\mathcal{K}_{1,2}\,\mathcal{F}_{r,s}}$.
\begin{Lemma}\label{lem:rest_Y_surj}
For $r,s\in\mathbb{Z}_{\geq 1}$, let $\mathcal{Y}$ be a non-zero $\mathcal{F}_0$-module intertwining operator of type $\binom{\mathcal{F}_{r,s+1}}{\mathcal{F}_{1,2}\,\mathcal{F}_{r,s}}$. Then the $\mathfrak{Vir}$-module intertwining operator $\mathcal{Y}\vert_{\mathcal{K}_{1,2}\otimes\mathcal{F}_{r,s}}$ of type $\binom{\mathcal{F}_{r,s+1}}{\mathcal{K}_{1,2}\,\mathcal{F}_{r,s}}$ is surjective.
\end{Lemma}
\begin{proof}
Since \eqref{eqn:Fock_intw_op} shows that $v_{r,s+1}$ is the coefficient of the lowest power of $z$ in $\mathcal{Y}(v_{1,2}, z)v_{r,s}$,  the module $\mathcal{F}_{r,s+1}$ is linearly spanned by coefficients of powers of $z$ in expressions of the form
\begin{equation*}
a_{-n_1}\cdots a_{-n_k}\mathcal{Y}(v_{1,2},z)v_{r,s}
\end{equation*}
for $n_1,\ldots, n_k >0$. Then the commutator formula \eqref{eqn:Heis_comm_form} together with $a_i v_{1,2}=\lambda_{1,2}\delta_{i,0} v_{1,2}$ for $i\geq 0$ implies that such coefficients of powers of $z$ are linear combinations of coefficients in expressions of the form
\begin{equation*}
\mathcal{Y}(v_{1,2},z)a_{-m_1}\cdots a_{-m_\ell} v_{r,s}.
\end{equation*}
Since $v_{1,2}\in\mathcal{K}_{1,2}\subseteq\mathcal{F}_{1,2}$, this shows that $\mathcal{F}_{r,s+1}$ is spanned by coefficients of powers of $z$ in expressions $\mathcal{Y}(v,z)w$ for $v\in\mathcal{K}_{1,2}$ and $w\in\mathcal{F}_{r,s}$. That is, $\mathcal{Y}\vert_{\mathcal{K}_{1,2}\otimes\mathcal{F}_{r,s}}$ is surjective.
\end{proof}

Now given a surjective $\mathfrak{Vir}$-module intertwining operator of type $\binom{\mathcal{F}_{r,s+1}}{\mathcal{K}_{1,2}\,\mathcal{F}_{r,s}}$ as in the previous lemma, we can further restrict to the Kac submodule $\mathcal{K}_{r,s}\subseteq\mathcal{F}_{r,s}$ to get a non-surjective $\mathfrak{Vir}$-module intertwining operator of type $\binom{\mathcal{F}_{r,s+1}}{\mathcal{K}_{1,2}\,\mathcal{K}_{r,s}}$. We will show that the image in $\mathcal{F}_{r,s+1}$ of this intertwining operator is the Kac module $\mathcal{K}_{r,s+1}$:
\begin{Theorem}\label{thm:Kac_module_surj}
For any $r,s\in\mathbb{Z}_{\geq 1}$, any non-zero $\mathcal{F}_0$-module intertwining operator of type $\binom{\mathcal{F}_{r,s+1}}{\mathcal{F}_{1,2}\,\mathcal{F}_{r,s}}$ restricts to a surjective $\mathfrak{Vir}$-module intertwining operator of type $\binom{\mathcal{K}_{r,s+1}}{\mathcal{K}_{1,2}\,\mathcal{K}_{r,s}}$. In particular, for any $r,s\in\mathbb{Z}_{\geq 1}$, there is a surjection $\mathcal{K}_{1,2}\boxtimes\mathcal{K}_{r,s}\rightarrow\mathcal{K}_{r,s+1}$ in $\mathcal{O}_{c_{p,q}}$.
\end{Theorem}
\begin{proof}
Recall that every Feigin-Fuchs module $\mathcal{F}_{r,s}$ is the same as one for which either $r\leq p$ or $s\leq q$. Thus fix $(r,s)$ with $r\leq p$ or $s\leq q$; from the structures in Section \ref{subsec:FF_and_K_modules}, we see that $\mathcal{F}_{r,s}$ has a filtration
\begin{equation*}
\mathcal{K}_{r,s}\subseteq\mathcal{K}_{p+r,q+s}\subseteq\cdots\subseteq\mathcal{K}_{mp+r,mq+s}\subseteq\cdots\subseteq\mathcal{F}_{r,s}.
\end{equation*}
To shorten the notation, set $\mathcal{K}_{r,s}^{(m)} =\mathcal{K}_{mp+r,mq+s}$; then for each $m\geq 0$, the lowest conformal weight of $\mathcal{F}_{r,s}/\mathcal{K}_{r,s}^{(m)}$ is 
\begin{equation*}
h_{r,s}+(mp+r)(mq+s)=h_{-mp-r,mq+s},
\end{equation*}
and $\mathcal{K}_{r,s}^{(m+1)}/\mathcal{K}_{r,s}^{(m)}$ is generated by its lowest conformal weight space, and thus is a quotient (of length $1$, $2$, or $4$) of $\mathcal{V}_{-mp-r,mq+s}$.

Now let $\mathcal{Y}$ be the intertwining operator of type $\binom{\mathcal{F}_{r,s+1}}{\mathcal{K}_{1,2}\,\mathcal{F}_{r,s}}$ obtained by restricting a non-zero Fock module intertwining operator to $\mathcal{K}_{1,2}$, and for $m\geq 0$, set
\begin{equation*}
\mathcal{F}_{r,s+1}^{(m)}=\mathrm{Im}\,\mathcal{Y}\vert_{\mathcal{K}_{1,2}\otimes\mathcal{K}_{r,s}^{(m)}}
\end{equation*}
(so we need to show that $\mathcal{F}^{(m)}_{r,s+1}=\mathcal{K}^{(m)}_{r,s+1}$). For each $m$, $\mathcal{Y}$ induces an intertwining operator 
\begin{equation*}
\mathcal{Y}^{(m)}: \mathcal{K}_{1,2}\otimes(\mathcal{F}_{r,s}/\mathcal{K}_{r,s}^{(m)})\longrightarrow(\mathcal{F}_{r,s+1}/\mathcal{F}_{r,s+1}^{(m)})\lbrace z\rbrace
\end{equation*}
It follows from Lemma \ref{lem:rest_Y_surj} that each $\mathcal{Y}^{(m)}$ is surjective.

We will first show that $\mathcal{K}_{r,s+1}^{(m)}\subseteq\mathcal{F}_{r,s+1}^{(m)}$. Since $\mathcal{K}_{r,s+1}^{(m)}$ is generated by all vectors in $\mathcal{F}_{r,s+1}$ of conformal weight strictly less than $h_{r,s+1}+(mp+r)(mq+s+1)=h_{-mp-r,mq+s+1}$, it is sufficient to show that the conformal weights of $\mathcal{F}_{r,s+1}/\mathcal{F}_{r,s+1}^{(m)}=\mathrm{Im}\,\mathcal{Y}^{(m)}$ are contained in $ h_{-mp-r,mq+s+1}+\mathbb{Z}_{\geq 0}$. To prove this condition on conformal weights, first note that
\begin{equation*}
\mathcal{F}_{r,s}/\mathcal{K}_{r,s}^{(m)} =\bigcup_{n=0}^\infty \mathcal{K}_{r,s}^{(m+n)}/\mathcal{K}_{r,s}^{(m)}.
\end{equation*}
Thus the conformal weights of $\mathcal{F}_{r,s+1}/\mathcal{F}^{(m)}_{r,s+1}$ are the union (over $n\geq 0$) of the conformal weights of $\mathrm{Im}\,\mathcal{Y}^{(m)}\vert_{\mathcal{K}_{1,2}\otimes(\mathcal{K}_{r,s}^{(m+n)}/\mathcal{K}_{r,s}^{(m)})}$. Now consider the following commutative diagram with exact rows:
\begin{equation*}
\xymatrixcolsep{1.4pc}
\xymatrix{
0 \ar[r] & \mathcal{K}_{1,2}\otimes(\mathcal{K}_{r,s}^{(m+n)}/\mathcal{K}_{r,s}^{(m)}) \ar[r] \ar[d]^{\mathcal{Y}^{(m)}\vert_{\mathcal{K}_{1,2}\otimes(\mathcal{K}_{r,s}^{(m+n)}/\mathcal{K}_{r,s}^{(m)})}} & \mathcal{K}_{1,2}\otimes(\mathcal{K}_{r,s}^{(m+n+1)}/\mathcal{K}_{r,s}^{(m)}) \ar[r] \ar[d]^{\mathcal{Y}^{(m)}\vert_{\mathcal{K}_{1,2}\otimes(\mathcal{K}_{r,s}^{(m+n+1)}/\mathcal{K}_{r,s}^{(m)})}} & \mathcal{K}_{1,2}\otimes(\mathcal{K}_{r,s}^{(m+n+1)}/\mathcal{K}_{r,s}^{(m+n)}) \ar[r] \ar[d] & 0\\
0\ar[r] & (\mathcal{F}_{r,s+1}^{(m+n)}/\mathcal{F}_{r,s+1}^{(m)})\lbrace z\rbrace \ar[r] & (\mathcal{F}_{r,s+1}^{(m+n+1)}/\mathcal{F}_{r,s+1}^{(m)})\lbrace z\rbrace \ar[r] & (\mathcal{F}_{r,s+1}^{(m+n+1)}/\mathcal{F}_{r,s+1}^{(m+n)})\lbrace z\rbrace \ar[r] & 0\\
}
\end{equation*}
From this diagram, induction on $n$ implies that the conformal weights of $\mathcal{F}_{r,s+1}/\mathcal{F}_{r,s+1}^{(m)}$ are the union (over $n\geq 0$) of the conformal weights of surjective intertwining operators
\begin{equation*}
\mathcal{K}_{1,2}\otimes(\mathcal{K}_{r,s}^{(m+n+1)}/\mathcal{K}_{r,s}^{(m+n)})\longrightarrow(\mathcal{F}_{r,s+1}^{(m+n+1)}/\mathcal{F}_{r,s+1}^{(m+n)})\lbrace z\rbrace.
\end{equation*}
Since $\mathcal{K}_{r,s}^{(m+n+1)}/\mathcal{K}_{r,s}^{(m+n)}$ is a quotient of the Verma module $\mathcal{V}_{-(m+n)p-r,(m+n)q+s}$, Proposition \ref{prop:lowest_weight_spaces} implies that the conformal weights of $\mathcal{F}_{r,s+1}/\mathcal{F}_{r,s+1}^{(m)}$ are contained in
\begin{align*}
\bigcup_{n=0}^\infty &\, (h_{-(m+n)p-r,(m+n)q+s-1}+\mathbb{Z}_{\geq 0})\cup ( h_{-(m+n)p-r,(m+n)q+s+1}+\mathbb{Z}_{\geq 0})\nonumber\\
&\qquad = (h_{-mp-r,mq+s-1}+\mathbb{Z}_{\geq 0})\cup ( h_{-mp-r,mq+s+1}+\mathbb{Z}_{\geq 0}),
\end{align*}
where the equality follows because for $n\geq 0$,
\begin{align*}
&h_{-(m+n)p-r,(m+n)q+s\pm 1}  - h_{-mp-r,mq+s\pm 1} \nonumber\\
&\qquad\qquad = h_{r,s\pm 1} +((m+n)p+r)((m+n)q+s\pm 1) -h_{r,s\pm 1}-(mp+r)(mq+s\pm 1)\nonumber\\
&\qquad\qquad = np(mq+s\pm 1)+(mp+r)nq +n^2pq\in\mathbb{Z}_{\geq 0}.
\end{align*}

Now since all conformal weights of $\mathcal{F}_{r,s+1}/\mathcal{F}_{r,s+1}^{(m)}$ are congruent to $h_{r,s+1}$ modulo $\mathbb{Z}$, the lowest possible conformal weight of $\mathcal{F}_{r,s+1}/\mathcal{F}_{r,s+1}^{(m)}$ is indeed $h_{-mp-r,mq+s+1}$, except maybe when
\begin{equation*}
h_{-mp-r,mq+s+1}-h_{-mp-r,mq+s-1}\in\mathbb{Z}_{\geq 0}.
\end{equation*}
This occurs only when $q\mid s$, so because either $r\leq p$ or $s\leq q$, we have four cases to consider:
\begin{enumerate}
\item $1\leq r\leq p-1$ and $s=nq$ for some $n\geq 1$.
\item $r=p$ and $s=nq$ for some $n\geq 1$.
\item $r=np+r'$ for some $n\geq 1$ and $1\leq r'\leq p-1$, and $s=q$.
\item $r=np$ for some $n\geq 2$ and $s=q$.
\end{enumerate}
In the first case, $h_{-mp-r,mq+s\pm 1}= h_{p-r,(n+2m+1)q\pm 1}$, and the structure of $\mathcal{F}_{r,nq+1}$ from Section \ref{subsec:FF_and_K_modules} is (implicitly assuming $m\geq 1$):
\begin{equation*}
\xymatrixcolsep{2pc}
\xymatrixrowsep{.5pc}
\xymatrix{
 & \mathcal{L}_{p-r,(n+1)q+1} \ar[ld] \ar[r] \ar[rdd] & \cdots \ar[ldd]\ar[rdd] \ar[r] & \mathcal{L}_{p-r,(n+2m+1)q-1} \ar[ldd] \ar[rdd]  & \mathcal{L}_{p-r,(n+2m+1)q+1} \ar[l] \ar[r] \ar[ldd] \ar[rdd]  & \cdots \ar[ldd]  \\
\mathcal{L}_{r,nq+1} \ar[rd] & & & & &  \\
 & \mathcal{L}_{r,(n+2)q-1}  & \cdots \ar[l] &\mathcal{L}_{r,(n+2m)q+1} \ar[l] \ar[r] & \mathcal{L}_{r,(n+2m+2)q-1} &  \cdots \ar[l] \\
}
\end{equation*}
In the second case, $h_{-mp-r,mq+s\pm 1}=h_{p,(n+2m+2)q\pm 1}$, and the structure of $\mathcal{F}_{p,nq+1}$ is:
\begin{equation*}
\xymatrixcolsep{2pc}
\xymatrix{
\mathcal{L}_{p, nq+1} \ar[r] & \mathcal{L}_{p,(n+2)q-1} & \ar[l] \cdots \ar[r] &  \mathcal{L}_{p,(n+2m+2)q-1} & \mathcal{L}_{p,(n+2m+2)q+1} \ar[l] \ar[r] & \cdots \\
}
\end{equation*}
In the third case, $h_{-mp-r,mq+s\pm 1}=h_{p-r',(n+2m+2)q\pm 1}$, and $\mathcal{F}_{np+r',q+1}\cong\mathcal{F}_{p-r',nq-1}'$ has structure:
\begin{equation*}
\xymatrixcolsep{2pc}
\xymatrixrowsep{.5pc}
\xymatrix{
 & \mathcal{L}_{r',(n+1)q-1}  & \cdots \ar[l] & \mathcal{L}_{r',(n+2m+1)q+1}  \ar[l]\ar[r] & \mathcal{L}_{r',(n+2m+3)q-1} & \cdots \ar[l] \\
\mathcal{L}_{p-r',nq-1} \ar[ru] & & & & &  \\
 & \mathcal{L}_{p-r',nq+1} \ar[lu] \ar[ruu]\ar[r]  & \cdots \ar[luu]\ar[ruu]\ar[r]  &\mathcal{L}_{p-r',(n+2m+2)q-1} \ar[ruu]\ar[luu] & \mathcal{L}_{p-r',(n+2m+2)q+1} \ar[l]\ar[r]\ar[ruu]\ar[luu]  &\cdots \ar[luu]  \\
}
\end{equation*}
In the fourth case, $h_{-mp-r,mq+s\pm 1}=h_{p,(n+2m+2)q\pm 1}$, and the structure of $\mathcal{F}_{np,q+1}\cong\mathcal{F}_{p,nq-1}'$ is:
\begin{equation*}
\xymatrixcolsep{2pc}
\xymatrix{
\mathcal{L}_{p, nq-1} & \mathcal{L}_{p,nq+1} \ar[l]\ar[r] & \cdots \ar[r] & \mathcal{L}_{p,(n+2m+2)q-1}  & \mathcal{L}_{p,(n+2m+2)q+1} \ar[r]\ar[l] &\cdots \\
}
\end{equation*}
In all four cases, we know that $\mathcal{F}_{r,s+1}^{(m)}$ must contain all vectors in $\mathcal{F}_{r,s+1}$ with conformal weight strictly less than $h_{-mp-r,mq+s-1}$, and moreover $\mathcal{F}_{r,s+1}^{(m)}$ contains the lowest conformal weight vector $v_{r,s+1}$. But it follows from the above structures that in all cases, these vectors are enough to generate all of $\mathcal{K}_{r,s+1}^{(m)}$. That is, the lowest conformal weight of $\mathcal{F}_{r,s+1}/\mathcal{F}_{r,s+1}^{(m)}$ is at least $h_{-mp-r,mq+s+1}$, and $\mathcal{K}_{r,s+1}^{(m)}\subseteq\mathcal{F}_{r,s+1}^{(m)}$ as desired.

Conversely, we need to show that $\mathcal{F}_{r,s+1}^{(m)}\subseteq\mathcal{K}_{r,s+1}^{(m)}$. To do so, we consider the composition
\begin{equation*}
\widetilde{\mathcal{Y}}^{(m)}: \mathcal{K}_{1,2}\otimes\mathcal{K}_{r,s}^{(m)} \xrightarrow{\mathcal{Y}^{(m)}} \mathcal{F}_{r,s+1}^{(m)}\lbrace z\rbrace\twoheadrightarrow(\mathcal{F}_{r,s+1}^{(m)}/\mathcal{K}_{r,s+1}^{(m)})\lbrace z\rbrace.
\end{equation*}
This intertwining operator is surjective, so it is sufficient to show that $\widetilde{\mathcal{Y}}^{(m)}=0$. First we show that the restriction of $\widetilde{\mathcal{Y}}^{(m)}$ to $\mathcal{K}_{r,s}=\mathcal{K}_{r,s}^{(0)}$ vanishes. In fact, since $\mathcal{K}_{r,s}$ is generated by $v_{r,s}$, Proposition \ref{prop:lowest_weight_spaces} shows that if $\widetilde{\mathcal{Y}}^{(m)}\vert_{\mathcal{K}_{1,2}\otimes\mathcal{K}_{r,s}}\neq 0$, then $\mathrm{Im}\,\widetilde{\mathcal{Y}}^{(m)}\vert_{\mathcal{K}_{1,2}\otimes\mathcal{K}_{r,s}}$ has lowest conformal weight(s) $h_{r,s\pm 1}$. However, the conformal weights (if any) of $\mathcal{F}_{r,s+1}^{(m)}/\mathcal{K}_{r,s+1}^{(m)}$ are contained in $h_{-mp-r,mq+s+1}+\mathbb{Z}_{\geq 0}$, and
\begin{align*}
h_{-mp-r,mq+s+1}-h_{r,s+1} & = (mp+r)(mq+s+1)>0,\\
h_{-mp-r,mq+s+1}-h_{r,s-1} & = (mp+r)(mq+s+1) +h_{r,s+1}-h_{r,s-1}\nonumber\\
& =(mp+r)(mq+s)+mp+r+\left(-r+\frac{ps}{q}\right)\nonumber\\
 & = (mp+r)(mq+s)+p\left(m+\frac{s}{q}\right)>0.
\end{align*}
Thus $h_{r,s\pm 1}$ are not conformal weights of $\mathcal{F}_{r,s+1}^{(m)}/\mathcal{K}_{r,s+1}^{(m)}$, so it follows that $\widetilde{\mathcal{Y}}^{(m)}\vert_{\mathcal{K}_{1,2}\otimes\mathcal{K}_{r,s}}= 0$.

Now assume by induction that $\widetilde{\mathcal{Y}}^{(m)}\vert_{\mathcal{K}_{1,2}\otimes\mathcal{K}_{r,s}^{(n)}} = 0$ for some $n\in\lbrace 0, 1,\ldots, m-1\rbrace$. Then $\widetilde{\mathcal{Y}}^{(m)}$ induces a well-defined intertwining operator
\begin{equation}\label{eqn:vanishing_intw_op}
\mathcal{K}_{1,2}\otimes(\mathcal{K}_{r,s}^{(n+1)}/\mathcal{K}_{r,s}^{(n)})\longrightarrow(\mathcal{F}_{r,s+1}^{(m)}/\mathcal{K}_{r,s+1}^{(m)})\lbrace z\rbrace.
\end{equation}
As before, if this intertwining operator is non-zero, then Proposition \ref{prop:lowest_weight_spaces} implies that the lowest conformal weight(s) of its image are $h_{-np-r,nq+s\pm 1}$. But similar to the above,
\begin{align*}
h_{-mp-r,mq+s+1}-h_{-np-r,nq+s+1} & =(mp+r)(mq+s)-(np+r)(nq+s)+(m-n),\\
h_{-mp-r,mq+s+1}-h_{-np-r,nq+s-1} & =(mp+r)(mq+s)-(np+r)(nq+s)+(m+n)p+r+\frac{ps}{q}
\end{align*}
are both positive. So $h_{-np-r,nq+s\pm 1}$ are not conformal weights of $\mathcal{F}_{r,s+1}^{(m)}/\mathcal{K}_{r,s+1}^{(m)}$, and thus the intertwining operator \eqref{eqn:vanishing_intw_op} is $0$. It follows that $\widetilde{\mathcal{Y}}^{(m)}\vert_{\mathcal{K}_{1,2}\otimes\mathcal{K}_{r,s}^{(n+1)}} = 0$, and then by induction on $n$, we conclude $\widetilde{\mathcal{Y}}^{(m)}=0$. That is, $\mathcal{F}_{r,s+1}^{(m)}/\mathcal{K}_{r,s+1}^{(m)}=0$, completing the proof that $\mathcal{F}_{r,s+1}^{(m)} = \mathcal{K}_{r,s+1}^{(m)}$. This conclusion applies to all $(r,s)$ such that $r\leq p$ or $s\leq q$, and to all $m\geq 0$. Thus for any $r,s\in\mathbb{Z}_{\geq 1}$, the image $\mathcal{Y}\vert_{\mathcal{K}_{1,2}\otimes\mathcal{K}_{r,s}}$ equals $\mathcal{K}_{r,s+1}$, proving the theorem.
\end{proof}

\begin{Remark}\label{rem:Kac_surjection}
Because $c_{p,q}=c_{q,p}$, we can exchange the roles of $r$ and $s$ in Theorem \ref{thm:Kac_module_surj} to conclude that for all $r,s\in\mathbb{Z}_{\geq 1}$, there is also a surjection $\mathcal{K}_{2,1}\boxtimes\mathcal{K}_{r,s}\rightarrow\mathcal{K}_{r+1,s}$ in $\mathcal{O}_{c_{p,q}}$.
\end{Remark}

We will need to generalize the preceding theorem and remark to Fock module intertwining operators that do not necessarily involve $\mathcal{F}_{1,2}$ or $\mathcal{F}_{2,1}$. For this it is convenient to introduce the ind-completion (or direct limit completion) of the braided tensor category $\mathcal{O}_{c_{p,q}}$. This is the category  $\mathrm{Ind}\,\mathcal{O}_{c_{p,q}}$ of all generalized $V_{c_{p,q}}$-modules which are the unions of submodules which are objects of $\mathcal{O}_{c_{p,q}}$. For example, although the Feigin-Fuchs modules $\mathcal{F}_{r,s}$ for $r\leq p$ or $s\leq q$ are not objects of $\mathcal{O}_{c_{p,q}}$, they are objects of $\mathrm{Ind}\,\mathcal{O}_{c_{p,q}}$ because $\mathcal{F}_{r,s}=\bigcup_{m=0}^\infty \mathcal{K}_{mp+r,mq+s}$. By the main theorem of \cite{CMY-completions}, $\mathrm{Ind}\,\mathcal{O}_{c_{p,q}}$ admits the vertex algebraic braided tensor category structure of \cite{HLZ1}-\cite{HLZ8}, in such a way that the inclusion $\mathcal{O}_{c_{p,q}}\rightarrow\mathrm{Ind}\,\mathcal{O}_{c_{p,q}}$ is a braided tensor functor.

For any Feigin-Fuchs modules $\mathcal{F}_\lambda$ and $\mathcal{F}_\mu$, where $\lambda,\mu\in\mathbb{C}$, fix a non-zero $\mathcal{F}_0$-module intertwining operator of type $\binom{\mathcal{F}_{\lambda+\mu}}{\mathcal{F}_\lambda\,\mathcal{F}_\mu}$, which we will write as $\mathcal{Y}_{\lambda,\mu}$.

\begin{Theorem}\label{thm:Fock_module_associativity}
For any $r,s,r',s'\in\mathbb{Z}_{\geq 1}$, the $\mathcal{F}_0$-module intertwining operator $\mathcal{Y}_{\lambda_{r,s},\lambda_{r',s'}}$ restricts to a surjective $\mathfrak{Vir}$-module intertwining operator of type $\binom{\mathcal{K}_{r+r'-1,s+s'-1}}{\mathcal{K}_{r,s}\,\,\,\mathcal{K}_{r',s'}}$. In particular, there is a surjection $F_{r,s;r',s'}: \mathcal{K}_{r,s}\boxtimes\mathcal{K}_{r',s'}\rightarrow\mathcal{K}_{r+r'-1,s+s'-1}$ in $\mathcal{O}_{c_{p,q}}$. Moreover, for any $r,s,r',s',r'',s''\in\mathbb{Z}_{\geq 1}$, there is a commutative diagram
\begin{equation*}
\xymatrixcolsep{4.5pc}
\xymatrixrowsep{.75pc}
\xymatrix{
\mathcal{K}_{r,s}\boxtimes(\mathcal{K}_{r',s'}\boxtimes\mathcal{K}_{r'',s''}) \ar[r]^{\mathrm{Id}\boxtimes F_{r',s';r'',s''}} \ar[dd]^{\cong} & \mathcal{K}_{r,s}\boxtimes\mathcal{K}_{r'+r''-1,s'+s''-1} \ar[rd]^{\qquad F_{r,s;r'+r''-1,s'+s''-1}} & \\
 & & \mathcal{K}_{r+r'+r''-2,s+s'+s''-2} \\
 (\mathcal{K}_{r,s}\boxtimes\mathcal{K}_{r',s'})\boxtimes\mathcal{K}_{r'',s''} \ar[r]^{F_{r,s;r',s'}\boxtimes\mathrm{Id}} & \mathcal{K}_{r+r'-1,s+s'-1}\boxtimes\mathcal{K}_{r'',s''} \ar[ru]_{\quad\qquad F_{r+r'-1,s+s'-1;r'',s''}} & \\
}
\end{equation*}
where the vertical map is some non-zero multiple of the associativity isomorphism in $\mathcal{O}_{c_{p,q}}$.
\end{Theorem}
\begin{proof}
First note for $r,s,r',s'\in\mathbb{Z}_{\geq 1}$ that $\lambda_{r,s}+\lambda_{r',s'}=\lambda_{r+r'-1,s+s'-1}$. For $r,s,r',s',r'',s''\in\mathbb{Z}_{\geq 1}$, let us write $\lambda=\lambda_{r,s}$, $\mu=\lambda_{r',s'}$, and $\nu=\lambda_{r'',s''}$. Then $\mathcal{F}_\lambda$, $\mathcal{F}_\mu$, and $\mathcal{F}_\nu$ are objects of $\mathrm{Ind}\,\mathcal{O}_{c_{p,q}}$, and we have, for example, a surjective map $F_{\lambda,\mu}: \mathcal{F}_\lambda\boxtimes\mathcal{F}_\mu\rightarrow\mathcal{F}_{\lambda+\mu}$ in $\mathrm{Ind}\,\mathcal{O}_{c_{p,q}}$ induced by the surjective intertwining operator $\mathcal{Y}_{\lambda,\mu}$ and the universal property of tensor products in $\mathrm{Ind}\,\mathcal{O}_{c_{p,q}}$.

We need to compare the two compositions
\begin{equation*}
\mathcal{F}_\lambda\boxtimes(\mathcal{F}_\mu\boxtimes\mathcal{F}_\nu)\xrightarrow{\mathcal{A}} (\mathcal{F}_\lambda\boxtimes\mathcal{F}_\mu)\boxtimes\mathcal{F}_\nu\xrightarrow{F_{\lambda,\mu}\boxtimes\mathrm{Id}} \mathcal{F}_{\lambda+\mu}\boxtimes\mathcal{F}_{\nu}\xrightarrow{F_{\lambda+\mu,\nu}} \mathcal{F}_{\lambda+\mu+\nu},
\end{equation*}
where $\mathcal{A}$ is the associativity isomorphism in $\mathrm{Ind}\,\mathcal{O}_{c_{p,q}}$, and
\begin{equation*}
\mathcal{F}_\lambda\boxtimes(\mathcal{F}_\mu\boxtimes\mathcal{F}_\nu)\xrightarrow{\mathrm{Id}\boxtimes F_{\mu,\nu}} \mathcal{F}_\lambda\boxtimes\mathcal{F}_{\mu+\nu}\xrightarrow{F_{\lambda,\mu+\nu}} \mathcal{F}_{\lambda+\mu+\nu}.
\end{equation*}
From the description of the associativity isomorphisms in \cite{HLZ8} (see also \cite{CKM-exts} or \cite{CMY-completions}), the first composition is given as follows: Choose $r_1,r_2\in\mathbb{R}_{>0}$ such that $r_1>r_2>r_1-r_2$; then for any $w_\lambda\in\mathcal{F}_\lambda$, $w_\mu\in\mathcal{F}_\mu$, $w_\nu\in\mathcal{F}_\nu$, and $w'\in\mathcal{F}_{\lambda+\mu+\nu}'$,
\begin{align}\label{eqn:Fock_iterate}
& \left\langle w',F_{\lambda+\mu,\nu}\left((F_{\lambda,\mu}\boxtimes\mathrm{Id})\left(\mathcal{A}\left(\mathcal{Y}_\boxtimes(w_\lambda,r_1)\mathcal{Y}_\boxtimes(w_\mu,r_2)w_\nu\right)\right)\right)\right\rangle\nonumber\\
&\qquad\qquad =\left\langle w', F_{\lambda+\mu,\nu}\left((F_{\lambda,\mu}\boxtimes\mathrm{Id})\left(\mathcal{Y}_\boxtimes(\mathcal{Y}_\boxtimes(w_\lambda, r_1-r_2)w_\mu,r_2)w_\nu\right)\right)\right\rangle\nonumber\\
&\qquad\qquad =\left\langle w', F_{\lambda+\mu,\nu}\left(\mathcal{Y}_\boxtimes(\mathcal{Y}_{\lambda,\mu}(w_\lambda,r_1-r_2)w_\mu,r_2)w_\nu\right)\right\rangle\nonumber\\
&\qquad\qquad = \langle w',\mathcal{Y}_{\lambda+\mu,\nu}(\mathcal{Y}_{\lambda,\mu}(w_\lambda,r_1-r_2)w_\mu,r_2)w_\nu\rangle,
\end{align}
where in particular all the above series are absolutely convergent. Similarly,
\begin{align}\label{eqn:Fock_product}
& \left\langle w',F_{\lambda,\mu+\nu}\left(\mathrm{Id}\boxtimes F_{\mu,\nu})\left(\mathcal{Y}_\boxtimes(w_\lambda,r_1)\mathcal{Y}_\boxtimes(w_\mu,r_2)w_\nu\right)\right)\right\rangle\nonumber\\
&\qquad\qquad =\langle w',\mathcal{Y}_{\lambda,\mu+\nu}(w_\lambda,r_1)\mathcal{Y}_{\mu,\nu}(w_\mu,r_2)w_\nu\rangle.
\end{align}
Now, associativity properties following from the generalized Jacobi identities proved in \cite{DL} imply that \eqref{eqn:Fock_iterate} and \eqref{eqn:Fock_product} differ by a non-zero scalar $\alpha$ which is independent of $w_\lambda$, $w_\mu$, $w_\nu$, and $w'$ (it only depends on the choices of normalization for $\mathcal{Y}_{\lambda+\mu,\nu}$, $\mathcal{Y}_{\lambda,\mu}$, $\mathcal{Y}_{\lambda,\mu+\nu}$, and $\mathcal{Y}_{\mu,\nu}$). Thus because $\mathcal{F}_\lambda\boxtimes(\mathcal{F}_\mu\boxtimes\mathcal{F}_\nu)$ is spanned by projections of the series $\mathcal{Y}_\boxtimes(w_\lambda,r_1)\mathcal{Y}_\boxtimes(w_\mu,r_2)w_\nu$ to the conformal weight spaces by \cite[Corollary 7.17]{HLZ5}, we get a commutative diagram
\begin{equation}\label{diag:Fock_assoc}
\begin{matrix}
\xymatrixrowsep{.75pc}
\xymatrixcolsep{3pc}
\xymatrix{
\mathcal{F}_{\lambda}\boxtimes(\mathcal{F}_{\mu}\boxtimes\mathcal{F}_{\nu}) \ar[r]^(.55){\mathrm{Id}\boxtimes F_{\mu,\nu}} \ar[dd]^{\alpha\mathcal{A}} & \mathcal{F}_{\lambda}\boxtimes\mathcal{F}_{\mu+\nu} \ar[rd]^{ F_{\lambda,\mu+\nu}} & \\
 & & \mathcal{F}_{\lambda+\mu+\nu} \\
 (\mathcal{F}_{\lambda}\boxtimes\mathcal{F}_{\mu})\boxtimes\mathcal{F}_{\nu} \ar[r]^(.55){F_{\lambda,\mu}\boxtimes\mathrm{Id}} & \mathcal{F}_{\lambda+\mu}\boxtimes\mathcal{F}_{\nu} \ar[ru]_{ F_{\lambda+\mu,\nu}} & \\
}
\end{matrix}
\end{equation}

We can now define $F_{r,s;r',s'}=F_{\lambda,\mu}\vert_{\mathcal{K}_{r,s}\boxtimes\mathcal{K}_{r',s'}}$ for $r,s,r',s'\in\mathbb{Z}_{\geq 1}$, that is, $F_{r,s;r',s'}=F_{\lambda,\mu}\circ(\iota_{r,s}\boxtimes\iota_{r',s'})$ where $\iota_{r,s}:\mathcal{K}_{r,s}\rightarrow\mathcal{F}_{\lambda}$ and $\iota_{r',s'}:\mathcal{K}_{r',s'}\rightarrow\mathcal{F}_\mu$ are the inclusions. Then the above commutative diagram for Feigin-Fuchs modules will imply the commutative diagram for Kac modules in the statement of the theorem provided we can show that the image of $F_{r,s;r',s'}$ is indeed $\mathcal{K}_{r+r'-1,s+s'-1}$ for any $r,s,r',s'\in\mathbb{Z}_{\geq 1}$.

To prove $\mathrm{Im}\,F_{r,s;r',s'}=\mathcal{K}_{r+r'-1,s+s'-1}$, we first prove the $r=1$ case by induction on $s$. The $s=1$ case is clear because $\mathcal{K}_{1,1}$ is the unit object of $\mathcal{O}_{c_{p,q}}$ (and because $\mathcal{Y}_{0,\mu}$ gives the action of the vertex operator algebra $\mathcal{F}_0$ on its module $\mathcal{F}_\mu$), and the $s=2$ case is Theorem \ref{thm:Kac_module_surj}. Now assume by induction that $\mathrm{Im}\,F_{1,s;r',s'}=\mathcal{K}_{r',s+s'-1}$ for some $s\geq 2$. Because
\begin{equation*}
F_{1,2;1,s}: \mathcal{K}_{1,2}\boxtimes\mathcal{K}_{1,s}\rightarrow\mathcal{K}_{1,s+1}
\end{equation*}
is a surjection by Theorem \ref{thm:Kac_module_surj}, we have
\begin{align}\label{eqn:using_associativity}
\mathrm{Im}\,F_{1,s+1;r',s'} & =\mathrm{Im}\,F_{1,s+1;r',s'}\circ(F_{1,2;1,s}\boxtimes\mathrm{Id})\circ(\alpha\mathcal{A})\vert_{\mathcal{K}_{1,2}\boxtimes(\mathcal{K}_{1,s}\boxtimes\mathcal{K}_{r',s'})}\nonumber\\
& =\mathrm{Im}\,F_{\lambda_{1,2},\lambda_{1,s}+\lambda_{r',s'}}\circ(\mathrm{Id}\boxtimes F_{\lambda_{1,s},\lambda_{r',s'}})\vert_{\mathcal{K}_{1,2}\boxtimes(\mathcal{K}_{1,s}\boxtimes\mathcal{K}_{r',s'})}\nonumber\\
& =\mathrm{Im}\,F_{\lambda_{1,2},\lambda_{1,s}+\lambda_{r',s'}}\vert_{\mathcal{K}_{1,2}\boxtimes\mathcal{K}_{r',s+s'-1}}\nonumber\\
& =\mathcal{K}_{r',s+s'}
\end{align}
using \eqref{diag:Fock_assoc} and the inductive hypothesis. This completes the inductive step. Similarly, we can prove $\mathrm{Im}\,F_{r,1;r',s'}=\mathcal{K}_{r+r'-1,s'}$ using Remark \ref{rem:Kac_surjection} and induction on $r$. Finally, the $r=1$ and $s=1$ special cases imply that for any $r,s\in\mathbb{Z}_{\geq 1}$, $F_{r,1;1,s}:\mathcal{K}_{r,1}\boxtimes\mathcal{K}_{1,s}\rightarrow\mathcal{K}_{r,s}$ is a surjection, and then a calculation similar to \eqref{eqn:using_associativity} shows that $\mathrm{Im}\,F_{r,s;r',s'}=\mathcal{K}_{r+r'-1,s+s'-1}$ in general. This proves the theorem.
\end{proof}

\section{Some tensor products and (non)-rigidity results}\label{sec:tens_prods_and_rigidity}

We continue to fix $p,q\geq 2$ such that $\mathrm{gcd}(p,q)=1$.  In this section, we first derive some preliminary results on tensor products in $\mathcal{O}_{c_{p,q}}$, mainly involving the Kac modules $\mathcal{K}_{1,2}$ and $\mathcal{K}_{2,1}$. We then prove that these two modules are rigid and self-dual in $\mathcal{O}_{c_{p,q}}$, and finally obtain some results on rigidity (or not) for other modules in $\mathcal{O}_{c_{p,q}}$.

\subsection{Tensoring with \texorpdfstring{$\mathcal{K}_{1, 2}$}{K12} and \texorpdfstring{$\mathcal{K}_{2, 1}$}{K21}} 

We can already see from Theorem \ref{thm:Fock_module_associativity} that all Virasoro Kac modules $\mathcal{K}_{r,s}$ in $\mathcal{O}_{c_{p,q}}$ can be obtained through repeatedly tensoring $\mathcal{K}_{1,2}$ and $\mathcal{K}_{2,1}$. We now prove further results on how $\mathcal{K}_{1,2}$ and $\mathcal{K}_{2,1}$ tensor with each other and with other Kac modules, though our first result here in general involves $\widetilde{\mathcal{K}}_{r,s}$ rather than $\mathcal{K}_{r,s}$. Recall from Section \ref{subsec:Verma_intw_ops} that $\widetilde{\mathcal{K}}_{r,s}$ is the quotient of the Verma module $\mathcal{V}_{r,s}$ by its (unique up to scale) singular vector of conformal weight $h_{r,s}+rs$. In particular, $\widetilde{\mathcal{K}}_{r,0}=\widetilde{\mathcal{K}}_{0,s}=0$ for $r,s\in\mathbb{Z}_{\geq 0}$.

\begin{Theorem}\label{thm:K12_times_Ktilde}
For $r,s\geq 1$ such that $q\nmid s$,
\begin{equation*}
 \mathcal{K}_{1, 2} \boxtimes \widetilde{\mathcal{K}}_{r, s}\cong
        \widetilde{\mathcal{K}}_{r, s-1} \oplus \widetilde{\mathcal{K}}_{r, s+1}  .    
\end{equation*}
Similarly, for $r,s\geq 1$ such that $p\nmid r$,
\begin{equation*}
 \mathcal{K}_{2, 1} \boxtimes \widetilde{\mathcal{K}}_{r, s}\cong
        \widetilde{\mathcal{K}}_{r-1, s} \oplus \widetilde{\mathcal{K}}_{r+1, s}  .    
\end{equation*}
\end{Theorem}
\begin{proof}
The second statement follows from the first because $c_{p,q}=c_{q,p}$. To prove the first statement, the conformal weights of $\mathcal{K}_{1,2}\boxtimes\widetilde{\mathcal{K}}_{r,s}$ are contained in $( h_{r,s-1}+\mathbb{Z}_{\geq 0})\cup( h_{r,s+1}+\mathbb{Z}_{\geq 0})$ by Proposition \ref{prop:lowest_weight_spaces}. Since we assume $q\nmid s$, we have $h_{r,s-1}-h_{r,s+1}= r-\frac{ps}{q}\notin\mathbb{Z}$, and therefore there is a direct sum decomposition
 \begin{equation*}
 \mathcal{K}_{1,2}\boxtimes\widetilde{\mathcal{K}}_{r,s} =\mathcal{W}_{r,s-1}\oplus\mathcal{W}_{r,s+1}
 \end{equation*}
 such that the conformal weights of $\mathcal{W}_{r,s\pm 1}$ are contained in $h_{r,s\pm 1}+\mathbb{Z}_{\geq 0}$. Moreover, if either of $\mathcal{W}_{r,s\pm 1}\neq 0$, then its lowest conformal weight is precisely $h_{r,s\pm 1}$ and its lowest conformal weight space is one dimensional. Thus there are $\mathfrak{Vir}$-homomorphisms $\Pi_{\pm}: \mathcal{V}_{r,s\pm 1}\rightarrow\mathcal{W}_{r,s\pm 1}$. Each of $\Pi_{\pm}$ is surjective. Indeed, since there is a surjective intertwining operator of type $\binom{\mathcal{W}_{r,s\pm 1}/\mathrm{Im}\,\Pi_{\pm}}{\mathcal{K}_{1,2}\,\,\,\widetilde{\mathcal{K}}_{r,s}}$, Proposition \ref{prop:lowest_weight_spaces} implies that if $\mathcal{W}_{r,s\pm 1}/\mathrm{Im}\,\Pi_{\pm}$ were non-zero, then its lowest conformal weight would be $h_{r,s\pm 1}$. But this is impossible since the one-dimensional conformal weight $h_{r,s\pm 1}$ subspace of $\mathcal{W}_{r,s\pm 1}$ is contained in $\mathrm{Im}\,\Pi_{\pm}$. It follows that
 \begin{equation*}
 \mathcal{K}_{1, 2} \boxtimes \widetilde{\mathcal{K}}_{r, s} \cong (\mathcal{V}_{r, s-1}/\mathcal{J}^{-}) \oplus (\mathcal{V}_{r, s+1}/\mathcal{J}^{+})
 \end{equation*}
 for some submodules $\mathcal{J}^\pm\subseteq\mathcal{V}_{r,s\pm 1}$.
 
 Now since $q\nmid s$, Theorem \ref{thm:K_tilde_intw_ops} shows there is a surjective intertwining operator of type $\binom{\widetilde{\mathcal{K}}_{r,s\pm 1}}{\mathcal{K}_{1,2}\,\widetilde{\mathcal{K}}_{r,s}}$ for either sign choice. Thus the universal property of the tensor product $\mathcal{K}_{1,2}\boxtimes\widetilde{\mathcal{K}}_{r,s}$ induces surjective $\mathfrak{Vir}$-homomorphisms 
 \begin{equation*}
 \mathcal{V}_{r,s\pm 1}/\mathcal{J}^{\pm}\longrightarrow \widetilde{\mathcal{K}}_{r,s\pm 1}
 \end{equation*}
for both sign choices. That is, $\mathcal{J}^\pm$ is contained in the submodule of $\mathcal{V}_{r,s\pm 1}$ generated by the singular vector of conformal weight $h_{r,s\pm 1}+r(s\pm 1)$.

 To prove the theorem, we need to show that the above surjections are isomorphisms, and to prove this, we will use cofinite dimensions. From Propositions \ref{prop:Miyamoto}, \ref{prop:C1_quotient}, and \ref{prop:C1-cofinite_dimension}, we get
 \begin{align*}
 2rs \geq\dim_{C_1}(\mathcal{K}_{1,2}\boxtimes\widetilde{\mathcal{K}}_{r,s}) &=\dim_{C_1}(\mathcal{V}_{r,s-1}/\mathcal{J}^-)+\dim_{C_1}(\mathcal{V}_{r,s+1}/\mathcal{J}^+)\nonumber\\
 &\geq \dim_{C_1}(\widetilde{\mathcal{K}}_{r,s-1})+\dim_{C_1}(\widetilde{\mathcal{K}}_{r,s+1}) = r(s-1)+r(s+1)=2rs,
\end{align*} 
where the second inequality follows because there are surjections $\mathcal{V}_{r,s\pm 1}/\mathcal{J}^{\pm}\rightarrow \widetilde{\mathcal{K}}_{r,s\pm 1}$. Thus all inequalities are equalities and we get $\dim_{C_1}(\mathcal{V}_{r,s\pm 1}/\mathcal{J}^\pm)=\dim_{C_1}(\widetilde{\mathcal{K}}_{r,s\pm 1})$ for both sign choices. Then since $\mathcal{V}_{r,s\pm 1}/\mathcal{J}^{\pm}$ is spanned by vectors $L_{-1}^n v_{r,s\pm 1}$ modulo $C_1(\mathcal{V}_{r,s\pm 1}/\mathcal{J}^\pm)$, and since $C_1(\mathcal{V}_{r,s\pm 1}/\mathcal{J}^\pm)$ is graded, we get $L_{-1}^n v_{r,s\pm 1}\in C_1(\mathcal{V}_{r,s\pm 1}/\mathcal{J}^\pm)$ for some $n\leq r(s\pm 1)$. Then as in the proof of Proposition \ref{prop:C1-cofinite_dimension}, the PBW basis of $\mathcal{V}_{r,s\pm 1}$ implies that $\mathcal{J}^\pm$ contains a non-zero vector of conformal weight $h_{r,s\pm 1}+n$ for some $n\leq r(s\pm 1)$. But since $\mathcal{V}_{r,s\pm 1}/\mathcal{J}^\pm$ surjects onto $\widetilde{\mathcal{K}}_{r,s\pm 1}$, the submodule $\mathcal{J}^\pm$ is contained in the submodule generated by the singular vector in $\mathcal{V}_{r,s\pm 1}$ of conformal weight $h_{r,s\pm 1}+r(s\pm 1)$. Consequently, $\mathcal{J}^\pm$ contains the singular vector of weight $h_{r,s\pm 1}+r(s\pm 1)$, and then $\mathcal{V}_{r,s\pm 1}/\mathcal{J}^\pm = \widetilde{\mathcal{K}}_{r,s\pm 1}$.
\end{proof}

Since $\widetilde{\mathcal{K}}_{r,s}=\mathcal{K}_{r,s}$ when $r\leq p$ or $s\leq q$, the above theorem yields:
\begin{Corollary}\label{cor:1st_Kac_module_fusion}
For $1\leq r\leq p$ and $s\in\mathbb{Z}_{\geq 1}$ such that $q\nmid s$, and for $r\in\mathbb{Z}_{\geq 1}$ and $1\leq s\leq q-1$,
\begin{equation*}
\mathcal{K}_{1,2}\boxtimes\mathcal{K}_{r,s}\cong\mathcal{K}_{r,s-1}\oplus\mathcal{K}_{r,s+1}.
\end{equation*}
Similarly, for $r\in\mathbb{Z}_{\geq 1}$ such that $p\nmid r$ and $1\leq s\leq q$, and for $1\leq r\leq p-1$ and $s\in\mathbb{Z}_{\geq 1}$,
 \begin{equation*}
\mathcal{K}_{2,1}\boxtimes\mathcal{K}_{r,s}\cong\mathcal{K}_{r-1,s}\oplus\mathcal{K}_{r+1,s}.
\end{equation*}
\end{Corollary}

\begin{Remark}\label{rem:1st_Kac_module_fusion}
Since $\mathcal{K}_{r,0}=\mathcal{K}_{0,s}=0$ for $r,s\in\mathbb{Z}_{\geq 0}$, Corollary \ref{cor:1st_Kac_module_fusion} yields
\begin{equation*}
\mathcal{K}_{1,2}\boxtimes\mathcal{K}_{r,1}\cong\mathcal{K}_{r,2},\qquad\quad\mathcal{K}_{2,1}\boxtimes\mathcal{K}_{1,s}\cong\mathcal{K}_{2,s}
\end{equation*} 
for $r,s\geq 1$.
\end{Remark}

 The next result is based on Corollary \ref{cor:1st_Kac_module_fusion} and Remark \ref{rem:1st_Kac_module_fusion}:
\begin{Proposition}\label{prop:Kr1_K1s_small_r_or_s}
Let $r,s\in\mathbb{Z}_{\geq 1}$. If either $r\leq p$ or $s\leq q$, then $\mathcal{K}_{r,1}\boxtimes\mathcal{K}_{1,s}\cong\mathcal{K}_{r,s}$.
\end{Proposition}
\begin{proof}
By $c_{p,q}=c_{q,p}$ symmetry, it is enough to consider the case $r\leq p$. We have $\mathcal{K}_{1,1}\boxtimes\mathcal{K}_{1,s}\cong\mathcal{K}_{1,s}$ because $\mathcal{K}_{1,1}$ is the unit object of $\mathcal{O}_{c_{p,q}}$, and the $r=2$ case of the proposition is Remark \ref{rem:1st_Kac_module_fusion}. Now assume by induction on $r$ that for some $r\in\lbrace 2,\ldots p-1\rbrace$, $\mathcal{K}_{\rho,1}\boxtimes\mathcal{K}_{1,s}\cong\mathcal{K}_{\rho,s}$ for all $\rho\leq r$. Then by Corollary \ref{cor:1st_Kac_module_fusion} and the induction hypothesis,
\begin{align*}
\mathcal{K}_{r-1,s} \oplus\mathcal{K}_{r+1,s}  \cong\mathcal{K}_{2,1}\boxtimes\mathcal{K}_{r,s} & \cong \mathcal{K}_{2,1}\boxtimes (\mathcal{K}_{r,1}\boxtimes\mathcal{K}_{1,s})\cong(\mathcal{K}_{2,1}\boxtimes\mathcal{K}_{r,1})\boxtimes\mathcal{K}_{1,s}\nonumber\\
&\cong(\mathcal{K}_{r-1,1}\oplus\mathcal{K}_{r+1,1})\boxtimes\mathcal{K}_{1,s}\cong\mathcal{K}_{r-1,s}\oplus(\mathcal{K}_{r+1,1}\boxtimes\mathcal{K}_{1,s}).
\end{align*}
Thus $\mathcal{K}_{r+1,1}\boxtimes\mathcal{K}_{1,s}\cong\mathcal{K}_{r+1,s}$ (both are isomorphic to the direct summand of $\mathcal{K}_{2,1}\boxtimes\mathcal{K}_{r,s}$ consisting of vectors whose weights are congruent to $h_{r+1,s}$ mod $\mathbb{Z}$). This completes the induction.
\end{proof}

This paper will culminate in generalizations of Corollary \ref{cor:1st_Kac_module_fusion} and Proposition \ref{prop:Kr1_K1s_small_r_or_s} to all $r,s\in\mathbb{Z}_{\geq 1}$, but we need to prove that $\mathcal{K}_{1,2}$ and $\mathcal{K}_{2,1}$ are rigid first.

\subsection{Rigidity of \texorpdfstring{$\mathcal{K}_{1,2}$}{K12} and \texorpdfstring{$\mathcal{K}_{2,1}$}{K21}}

Recall that an object in a tensor category is called rigid if it has both left and right duals, each of which comes with evaluation and coevaluation morphisms that satisfy the so-called ``zig-zag'' relations, and a tensor category is called rigid if every object is rigid (for more details on rigidity in tensor categories, see for example \cite[Chapter 2]{BK}). In vertex algebraic tensor categories such as $\mathcal{O}_{c_{p,q}}$, left duals and right duals are the same due to the braiding and natural twist automorphism $\theta=e^{2\pi i L_0}$. Thus, recalling that the unit object of $\mathcal{O}_{c_{p,q}}$ is the Kac module $\mathcal{K}_{1,1}$, which is isomorphic to the universal vertex operator algebra $V_{c_{p,q}}$ as a $\mathfrak{Vir}$-module, an object $X$ of $\mathcal{O}_{c_{p,q}}$ is rigid with dual $X^*$ if there are evaluation and coevaluation morphisms
\begin{equation*}
e_X: X^*\boxtimes X\longrightarrow\mathcal{K}_{1,1},\qquad\quad i_X:\mathcal{K}_{1,1}\longrightarrow X\boxtimes X^*,
\end{equation*}
such that both the following compositions are identities:
\begin{align*}
X \xrightarrow{\cong} \mathcal{K}_{1,1}\boxtimes X\xrightarrow{i_X\boxtimes\mathrm{Id}} (X\boxtimes X^*)\boxtimes X & \xrightarrow{\cong} X\boxtimes(X^*\boxtimes X)\xrightarrow{\mathrm{Id}\boxtimes e_X} X\boxtimes\mathcal{K}_{1,1}\xrightarrow{\cong} X\\
 X^* \xrightarrow{\cong} X^*\boxtimes\mathcal{K}_{1,1}\xrightarrow{\mathrm{Id}\boxtimes i_X} X^*\boxtimes (X\boxtimes X^*) & \xrightarrow{\cong} (X^*\boxtimes X)\boxtimes X^*\xrightarrow{e_X\boxtimes\mathrm{Id}} \mathcal{K}_{1,1}\boxtimes X^*\xrightarrow{\cong} X^*
\end{align*}
In this section, we will show that certain Kac modules are rigid, but that the tensor category $\mathcal{O}_{c_{p,q}}$ as a whole is not rigid.

\begin{Remark} \label{rem:K11_rigid}
The module $\mathcal{K}_{1, 1}$ is rigid since it is the unit object of $\mathcal{O}_{c_{p, q}}$. \end{Remark}

We will now show that $\mathcal{K}_{1,2}$ is rigid in $\mathcal{O}_{c_{p,q}}$; because $c_{p,q}=c_{q,p}$, this will also show that $\mathcal{K}_{2,1}$ is rigid. The proof is based on \cite[Section 4.1]{MY-cp1-vir}, which in turn is based on \cite{TW}, but some adjustments are needed because the vertex operator algebra $V_{c_{p,q}}$ is neither simple nor self-contragredient. In particular, we will find that the dual of $\mathcal{K}_{1,2}$ is not the contragredient module $\mathcal{K}_{1,2}'$, but that rather $\mathcal{K}_{1,2}$ is self-dual.

We first construct coevaluation and evaluation candidates $i:\mathcal{K}_{1,1}\rightarrow\mathcal{K}_{1,2}\boxtimes\mathcal{K}_{1,2}$ and $e: \mathcal{K}_{1,2}\boxtimes\mathcal{K}_{1,2}\rightarrow\mathcal{K}_{1,1}$. For $i$, the same construction as in \cite[Section 4.1]{MY-cp1-vir} works, so we just recall the derivation here. First, by Proposition \ref{prop:lowest_weight_spaces}, the lowest conformal weight(s) of $\mathcal{K}_{1,2}\boxtimes\mathcal{K}_{1,2}$ are $h_{1,1}=0$ and/or $h_{1,3}=2\big(\frac{p}{q}-1\big)$. Thus when we give $\mathcal{K}_{1,2}\boxtimes\mathcal{K}_{1,2}$ the $\mathbb{Z}_{\geq 0}$-grading of Section \ref{subsec:Zhu_theory}, the space of conformal weight $0$ is contained in $(\mathcal{K}_{1,2}\boxtimes\mathcal{K}_{1,2})(0)$; as usual, let $\pi_0$ denote the projection to the latter subspace. Using $\mathcal{Y}_\boxtimes$ to denote the tensor product intertwining operator of type $\binom{\mathcal{K}_{1,2}\boxtimes\mathcal{K}_{1,2}}{\mathcal{K}_{1,2}\,\mathcal{K}_{1,2}}$, Nahm's argument in \cite{Na} implies that $(\mathcal{K}_{1,2}\boxtimes\mathcal{K}_{1,2})(0)$ is spanned by vectors of the form $\pi_0(\mathcal{Y}_\boxtimes(t,1)u)$ where $u\in\mathcal{K}_{1,2}(0)$ and $t$ comes from a special subspace $T\subseteq\mathcal{K}_{1,2}$ such that $\mathcal{K}_{1,2}=T+C_1(\mathcal{K}_{1,2})$ (recall the discussion following Proposition \ref{prop:surj_intw_op}). Since we may take $T=\mathrm{span}(v_{1,2}, L_{-1} v_{1,2})$ by Proposition \ref{prop:C1-cofinite_dimension}, $(\mathcal{K}_{1,2}\boxtimes\mathcal{K}_{1,2})(0)$ is spanned by $\pi_0(\mathcal{Y}_\boxtimes(v_{1,2},1)v_{1,2})$ and $\pi_0(\mathcal{Y}_\boxtimes(L_{-1} v_{1,2},1)v_{1,2})$.

It is not difficult to calculate how $L_0$ acts on $\pi_0(\mathcal{Y}_\boxtimes(v_{1,2},1)v_{1,2})$ and $\pi_0(\mathcal{Y}_\boxtimes(L_{-1} v_{1,2},1)v_{1,2})$; such calculations are typical in computing fusion products with the Nahm-Gaberdiel-Kausch algorithm \cite{Na, GK}, and in the vertex algebraic setting, they can also be performed using the Zhu algebra (see for example \cite[Proposition 3.4]{MY-cp1-vir}). One can show in particular that $-\pi_0\left(\mathcal{Y}_\boxtimes(v_{1, 2},1) v_{1, 2}\right)) + \frac{2q}{p} \pi_0\left(\mathcal{Y}_\boxtimes(L_{-1}v_{1, 2}, 1) v_{1, 2}\right)$ is (if non-zero) an $L_0$-eigenvector of conformal weight $0$, so there is a $\mathfrak{Vir}$-homomorphism
\begin{align}\label{eqn:coevaluation}
\mathcal{V}_{1,1} & \rightarrow\mathcal{K}_{1,2}\boxtimes\mathcal{K}_{1,2}\nonumber\\
\mathbf{1} & \mapsto -\pi_0\left(\mathcal{Y}_\boxtimes(v_{1, 2},1) v_{1, 2}\right)) + \frac{2q}{p} \pi_0\left(\mathcal{Y}_\boxtimes(L_{-1}v_{1, 2}, 1) v_{1, 2}\right),
\end{align}
where we denote $v_{1,1}=\mathbf{1}$.
Since $\mathcal{K}_{1,1}\cong\mathcal{V}_{1,1}/\langle L_{-1}\mathbf{1}\rangle$, this map will descend to a well-defined $\mathfrak{Vir}$-module homomorphism $i: \mathcal{K}_{1,1}\rightarrow\mathcal{K}_{1,2}\boxtimes\mathcal{K}_{1,2}$ if $L_{-1}\mathbf{1}$ is in the kernel. Indeed, this is proved by the same calculation as in \cite[Section 4.1]{MY-cp1-vir}, which uses the commutator formula \eqref{eqn:Vir_comm_form}, the iterate formula \eqref{eqn:Vir_it_form}, and the fact that $\big(L^2_{-1}-\frac{p}{q}L_{-2}\big)v_{1, 2}=0$ in $\mathcal{K}_{1,2}$.

\begin{Remark}
It is not clear from \eqref{eqn:coevaluation} that our coevaluation candidate $i:\mathcal{K}_{1,1}\rightarrow\mathcal{K}_{1,2}\boxtimes\mathcal{K}_{1,2}$ is necessarily non-zero. This follows from Corollary \ref{cor:1st_Kac_module_fusion} in case $q>2$, but for $q=2$, we will not be able to conclude $i\neq 0$ until the rigidity proof is complete. In any case, it will not be necessary to know $i\neq 0$ in order to prove that $\mathcal{K}_{1,2}$ is rigid.
\end{Remark}

The evaluation candidate depends on whether $q>2$ or $q=2$. If $q>2$, then  Theorem \ref{thm:K_tilde_intw_ops} shows that there is a surjective intertwining operator $\mathcal{E}$ of type $\binom{\mathcal{K}_{1,1}}{\mathcal{K}_{1,2}\,\mathcal{K}_{1,2}}$, and we define $e:\mathcal{K}_{1,2}\boxtimes\mathcal{K}_{1,2}\rightarrow\mathcal{K}_{1,1}$ to be the unique  homomorphism such that $e\circ\mathcal{Y}_\boxtimes =\mathcal{E}$. Moreover,  the proof of Theorem \ref{thm:gen_Verma_intw_ops} shows that the coefficient of the lowest power of $z$ in $\mathcal{E}(v_{1,2},z)v_{1,2}$ is a non-zero vector of minimal conformal weight in $\mathcal{K}_{1,1}$. Thus by rescaling if necessary,
\begin{equation}\label{eqn:E_for_q>2}
\mathcal{E}(v_{1,2},z)v_{1,2} = z^{-2h_{1,2}}(\mathbf{1}+z E(z))
\end{equation}
where $E(z)\in\mathcal{K}_{1,1}[[z]]$. 

In case $q=2$, Theorem \ref{thm:K_tilde_intw_ops} shows that there is a surjective intertwining operator of type $\binom{\mathcal{K}_{1,3}}{\mathcal{K}_{1,2}\,\mathcal{K}_{1,2}}$; moreover, Section \ref{subsec:FF_and_K_modules} shows that the structure of $\mathcal{K}_{1,1}$ is given by the exact sequence
\begin{equation*}
0\longrightarrow\mathcal{L}_{1,3}\longrightarrow\mathcal{K}_{1,1}\longrightarrow\mathcal{L}_{1,1}\longrightarrow 0.
\end{equation*}
Thus we get a non-zero but non-surjective intertwining operator
\begin{equation*}
\mathcal{E}: \mathcal{K}_{1,2}\boxtimes\mathcal{K}_{1,2}\rightarrow\mathcal{K}_{1,3}\lbrace z\rbrace\twoheadrightarrow\mathcal{L}_{1,3}\lbrace z\rbrace\hookrightarrow\mathcal{K}_{1,1}\lbrace z\rbrace,
\end{equation*}
and we define $e$ so that $e\circ\mathcal{Y}_\boxtimes=\mathcal{E}$. In this case, the coefficient of the lowest power of $z$ in $\mathcal{E}(v_{1,2},z)v_{1,2}$ is a singular vector $\tilde{v}_{1,3}$ generating the maximal proper submodule of $\mathcal{K}_{1,1}$, and
\begin{equation}\label{eqn:E_for_q=2}
\mathcal{E}(v_{1,2},z)v_{1,2} = z^{h_{1,3}-2h_{1,2}}(\tilde{v}_{1,3}+z E(z)),
\end{equation}
for some $E(z)\in\mathcal{L}_{1,3}[[z]]\subseteq\mathcal{K}_{1,1}[[z]]$.

\begin{Theorem}\label{thm:rigidity}
The Kac modules $\mathcal{K}_{1, 2}$ and $\mathcal{K}_{2,1}$ are rigid and self-dual in $\mathcal{O}_{c_{p, q}}$.
\end{Theorem}
\begin{proof}
As mentioned before, it is enough to prove $\mathcal{K}_{1,2}$ is rigid. We need to show that the compositions $\mathfrak{L}$ and $\mathfrak{R}$ given respectively by
\begin{equation*}
\mathcal{K}_{1, 2} \xrightarrow{\cong} \mathcal{K}_{1, 1} \boxtimes \mathcal{K}_{1, 2} \xrightarrow{i \boxtimes\mathrm{Id}} (\mathcal{K}_{1, 2} \boxtimes \mathcal{K}_{1, 2}) \boxtimes \mathcal{K}_{1, 2} \xrightarrow{\cong} \mathcal{K}_{1, 2} \boxtimes (\mathcal{K}_{1, 2} \boxtimes \mathcal{K}_{1, 2} )\xrightarrow{\mathrm{Id} \boxtimes e} \mathcal{K}_{1, 2} \boxtimes \mathcal{K}_{1, 1}\xrightarrow{\cong}  \mathcal{K}_{1, 2}
\end{equation*}
\begin{equation*}
 \mathcal{K}_{1, 2} \xrightarrow{\cong} \mathcal{K}_{1, 2} \boxtimes \mathcal{K}_{1, 1} \xrightarrow{\mathrm{Id}\boxtimes i} \mathcal{K}_{1, 2} \boxtimes (\mathcal{K}_{1, 2} \boxtimes \mathcal{K}_{1, 2}) \xrightarrow{\cong} (\mathcal{K}_{1, 2} \boxtimes \mathcal{K}_{1, 2}) \boxtimes \mathcal{K}_{1, 2} \xrightarrow{e\boxtimes \mathrm{Id}} \mathcal{K}_{1, 1} \boxtimes \mathcal{K}_{1, 2}\xrightarrow{\cong}  \mathcal{K}_{1, 2}
 \end{equation*}
equal the same non-zero multiple of the identity (we can then rescale either $e$ or $i$ to get the identity). We will first prove that $\langle v_{1,2}',\mathfrak{R}(v_{1,2})\rangle\neq 0$, where $v'_{1, 2}$ is a vector of weight $h_{1,2}$ in the contragredient module $ \mathcal{K}'_{1, 2}$. This will show that $\mathfrak{R}$ is a non-zero multiple of the identity since $\dim\mathrm{End}\,\mathcal{K}_{1,2}=1$. Then we will prove $\mathfrak{R}=\mathfrak{L}$ using the braiding isomorphisms in $\mathcal{O}_{c_{p,q}}$.

 For the case $q\geq 3$, the arguments and calculations in the proof of \cite[Theorem 4.1]{MY-cp1-vir} (in the $p\geq 3$ case of that theorem) literally apply at central charge $c_{p,q}$ if we make the substitution $p\mapsto \frac{q}{p}$ everywhere. The result is that for $q\geq 3$,
\begin{equation}\label{eqn:R_coefficient}
\langle v_{1,2}',\mathfrak{R}(v_{1,2})\rangle =\frac{2-q/p}{\cos(p \pi/q)},
\end{equation}
showing in particular that $\mathfrak{R}\neq 0$. For the case $q=2$, the argument is still similar to the $p=2$ case of \cite[Theorem 4.1]{MY-cp1-vir}, but there are some differences so we give more details here.

To calculate $\mathfrak{R}$ when $q=2$ and $p\geq 3$ is odd, it follows exactly as in \cite[Section 4.1]{MY-cp1-vir} that $\langle v_{1,2}',\mathfrak{R}(v_{1,2})\rangle$ is the coefficient of $z^{-2h_{1,2}}$ in the series
\begin{equation}\label{eqn:rig_prod}
\left(\frac{4}{p} z \frac{d}{dz}-1\right) \langle v'_{1, 2}, [l \circ (e \boxtimes \mathrm{Id}) \circ \mathcal{A} \circ \mathcal{Y}_{\boxtimes}] (v_{1, 2}, 1) \mathcal{Y}_{\boxtimes}(v_{1, 2}, z)v_{1, 2}\rangle,
\end{equation}
where $\mathcal{Y}_\boxtimes$ represents tensor product intertwining operators, $\mathcal{A}:\mathcal{K}_{1,2}\boxtimes(\mathcal{K}_{1,2}\boxtimes\mathcal{K}_{1,2})\rightarrow(\mathcal{K}_{1,2}\boxtimes\mathcal{K}_{1,2})\boxtimes\mathcal{K}_{1,2}$ is the associativity isomorphism, and $l:\mathcal{K}_{1,1}\boxtimes\mathcal{K}_{1,2}\rightarrow\mathcal{K}_{1,2}$ is the left unit isomorphism. If we treat $z$ as a complex number in \eqref{eqn:rig_prod}, using the branch of logarithm $\log z = \ln |z| + i \arg z$
with $-\pi < \arg z < \pi$ for non-integer powers of $z$ in $\mathcal{Y}_\boxtimes$, then \eqref{eqn:rig_prod} defines an analytic function on the simply-connected region
\begin{equation*}
U_1 = \{ z \in \mathbb{C} \mid |z| < 1 \} \setminus (-1, 0].
\end{equation*}
From the definitions of $\mathcal{A}$, $e$, and $l$ (see \cite{HLZ8} or the exposition in \cite[Section 3.3]{CKM-exts} for the definitions of the associativity and unit isomorphisms), this function has an analytic continuation to the simply-connected region
\begin{equation*}
U_2 = \{ z \in \mathbb{C}\mid |z|> |1-z|>0 \} \setminus [1, \infty) = \{ z \in \mathbb{C} \mid \mathrm{Re}\,z > 1/2 \} \setminus [1, \infty)
\end{equation*}
given by
\begin{align}\label{eqn:rig-it}
\left(\frac{4}{p} z \frac{d}{dz}-1\right)  \langle v'_{1, 2}, Y_{\mathcal{K}_{1, 2}}(\mathcal{E}(v_{1, 2}, 1-z)v_{1, 2}, z)v_{1, 2} \rangle
\end{align}
This expression can be viewed as a double series in $1 - z$ and $z$, or as a series in $\frac{1-z}{z}$. Thus, to show that $ \langle v'_{1, 2}, \mathfrak{R}(v_{1, 2})\rangle \neq 0$, we need to find the explicit expansion of the above as a series in $z$ and $\log z$ on $U_1 \cap U_2$, and then extract the coefficient of $z^{-2h_{1, 2}} (\log z)^0$.

Compositions of intertwining operators as in \eqref{eqn:rig_prod} and \eqref{eqn:rig-it} are solutions to Belavin-Polyakov-Zamolodchikov (BPZ) equations \cite{BPZ, Fran, Hu-vir}, and such solutions can be expressed in terms of hypergeometric functions. Accordingly, we can expand \eqref{eqn:rig-it} as a series in powers of $z$ and $\log z$ using analytic properties of hypergeometric functions; see \cite{Gu} for an early use of this method at central charge $c=-2$. For general $c_{p,q}$, the BPZ differential equation is given in \cite[Equation 8.71]{Fran} (where we take $t=\frac{p}{q}$ and $h_0=h_1=h_2=h_3=h_{1,2}$); see also the discussion in \cite[Section 4.2]{TW}, and see \cite[Proposition 4.1.2]{CMY-singlet} for a recent vertex algebraic derivation of the BPZ equation at $c_{p,1}$ central charge. In any case, when $q=2$,
\begin{equation*}
 \langle v'_{1, 2}, [l \circ (e \boxtimes \mathrm{Id}) \circ \mathcal{A} \circ \mathcal{Y}_{\boxtimes}] (v_{1, 2}, 1) \mathcal{Y}_{\boxtimes}(v_{1, 2}, z)v_{1, 2}\rangle,\qquad\langle v'_{1, 2}, Y_{\mathcal{K}_{1, 2}}(\mathcal{E}(v_{1, 2}, 1-z)v_{1, 2}, z)v_{1, 2} \rangle
\end{equation*}
are series solutions of the second-order differential equation
\begin{equation}\label{eqn:rig-diff-eqn}
z(1-z) \phi''(z)+\frac{p}{2} (1-2z) \phi'(z) - \frac{p}{2}h_{1,2} z^{-1}(1-z)^{-1} \phi(z) =0
\end{equation}
on the regions $U_1$ and $U_2$, respectively. 
If we write 
\begin{equation*}
\phi(z) = z^{p/4}(1-z)^{p/4} f(z)
\end{equation*}
where $\phi(z)$ is a solution of \eqref{eqn:rig-diff-eqn}, then $f(z)$ solves the hypergeometric differential equation
\begin{equation}\label{eqn:hyp-diff-eqn}
z(1-z) f''(z)+p (1-2z) f'(x) + \frac{p}{2}\left(1-\frac{3p}{2}\right) f(z)=0,
\end{equation}
whose solutions are well known (see for example \cite[Section 15.10]{DLMF}). In particular, some solutions of this hypergeometric differential equation are logarithmic.

To identify which solution of \eqref{eqn:hyp-diff-eqn} corresponds to \eqref{eqn:rig-it}, we take $\phi(z)=\langle v'_{1, 2}, Y_{\mathcal{K}_{1, 2}}(\mathcal{E}(v_{1, 2}, 1-z)v_{1, 2}, z)v_{1, 2} \rangle$ and use the $L_0$-conjugation property for intertwining operators \cite[Proposition 3.36(b)]{HLZ2} and \eqref{eqn:E_for_q=2} to find
\begin{align*}
(1-z)^{2h_{1,2}}\phi(z) & =\left(\frac{1-z}{z}\right)^{2h_{1,2}}\left\langle z^{-L_0} v_{1,2}', Y_{\mathcal{K}_{1,2}}\left(\mathcal{E}(z^{L_0} v_{1,2},1-z)z^{L_0}v_{1,2},z\right)z^{L_0}v_{1,2}\right\rangle\nonumber\\
& =\left(\frac{1-z}{z}\right)^{2h_{1,2}}\left\langle v_{1,2}',Y_{\mathcal{K}_{1,2}}\left(\mathcal{E}\left(v_{1,2},\frac{1-z}{z}\right)v_{1,2},1\right)v_{1,2}\right\rangle\nonumber\\
& =\left(\frac{1-z}{z}\right)^{2 h_{1,2}}\bigg(\langle v_{1,2}',Y_{\mathcal{K}_{1,2}}(\tilde{v}_{1,3},1)v_{1,2}\rangle\left(\frac{1-z}{z}\right)^{h_{1,3}-2h_{1,2}} +\ldots \bigg)\nonumber\\
& \in\left(\frac{1-z}{z}\right)^{p-1}\left(\langle v_{1,2}',Y_{\mathcal{K}_{1,2}}(\tilde{v}_{1,3},1)v_{1,2}\rangle+\frac{1-z}{z}\mathbb{C}\left[\left[\frac{1-z}{z}\right]\right]\right),
\end{align*}
as a series in $\frac{1-z}{z}$. In particular, $(1-z)^{2h_{1,2}}\phi(z)$ is analytic at $z=1$; from \cite[Section 15.10]{DLMF}, the only solutions with this property are
\begin{equation}\label{eqn:phi_is_hyper_geom}
\phi(z) =a\cdot z^{p/4}(1-z)^{p/4} {}_2F_1\bigg(\frac{p}{2},\frac{3p}{2}-1;p;1-z\bigg)
\end{equation}
for $a\in\mathbb{C}$. Actually, \eqref{eqn:phi_is_hyper_geom} gives the expansion of $\phi(z)$ as a series in $1-z$, not $\frac{1-z}{z}$, and thus converges on the region 
\begin{equation*}
1-U_1=\lbrace z\in\mathbb{C}\mid \vert 1-z\vert < 1\rbrace\setminus [1,2),
\end{equation*}
which intersects non-trivially with both $U_1$ and $U_2$.

We claim that $a\neq 0$ in \eqref{eqn:phi_is_hyper_geom}, which is equivalent to $\langle v_{1,2}',Y_{\mathcal{K}_{1,2}}(\tilde{v}_{1,3},1)v_{1,2}\rangle\neq 0$. To prove this claim, suppose $w'\in\mathcal{K}_{1,2}'$ is such that $\langle w',Y_{\mathcal{K}_{1,2}}(\tilde{v}_{1,3},z)v_{1,2}\rangle = 0$. Then by \eqref{eqn:contra} and the commutator formula \eqref{eqn:Vir_comm_form} (recall also \eqref{eqn:gen_intw_op}), for any $n>0$,
\begin{align*}
\langle L_{-n}w', Y_{\mathcal{K}_{1,2}}(\tilde{v}_{1,3},z)v_{1,2}\rangle & = \langle w', L_n Y_{\mathcal{K}_{1,2}}(\tilde{v}_{1,3},z)v_{1,2}\rangle \nonumber\\
& = z^n\left((n+1)h_{1,3}+z\dfrac{d}{dz}\right)\langle w',Y_{\mathcal{K}_{1,2}}(\tilde{v}_{1,3},z)v_{1,2}\rangle =0
\end{align*}
as well. Now, $\mathcal{K}_{1,2}'$ is generated by $v_{1,2}'$ since it is simple and isomorphic to $\mathcal{K}_{1,2}$ when $q=2$. Thus $\langle v_{1,2}',Y_{\mathcal{K}_{1,2}}(\tilde{v}_{1,3},1)v_{1,2}\rangle =0$ would imply $Y_{\mathcal{K}_{1,2}}(\tilde{v}_{1,3},z)v_{1,2}=0$. Then \cite[Proposition 11.9]{DL} would imply that $Y_{\mathcal{K}_{1,2}}\vert_{\mathcal{L}_{1,3}\otimes\mathcal{K}_{1,2}}=0$, and thus that $Y_{\mathcal{K}_{1,2}}$ would descend to a well-defined intertwining operator of type $\binom{\mathcal{K}_{1,2}}{L_{c_{p,2}}\,\mathcal{K}_{1,2}}$. But this is impossible since $\mathcal{K}_{1,2}$ is not an $L_{c_{p,2}}$-module \cite{Wa}. Thus in \eqref{eqn:phi_is_hyper_geom}, we have $a\neq 0$, and by rescaling the evaluation intertwining operator $\mathcal{E}$ if necessary, we may assume that $a=1$.

To finish the proof that $\mathfrak{R}\neq 0$ in the $q=2$ case, we need to extract the coefficient of $z^{-2h_{1,2}}$ in the expansion of 
\begin{equation}\label{eqn:series_for_rigidity}
\left(\frac{4}{p} z\frac{d}{dz}-1\right)\phi(z)
\end{equation}
as a series in $z$ on $U_1$. Using \eqref{eqn:phi_is_hyper_geom} and \cite[Equations 15.8.10 and 15.8.12]{DLMF}, the expansion of $\phi(z)$ on $U_1$ is given by 
\begin{align*}
\phi(z) & =\frac{z^{p/4}(1-z)^{p/4}}{\Gamma(\frac{p}{2})\Gamma(\frac{3p}{2}-1)}\sum_{k=0}^{p-2}(-1)^k\frac{(1-\frac{p}{2})_k(\frac{p}{2})_k(p-k-2)!}{k!} z^{k-p+1}\nonumber\\
&\qquad\qquad -\frac{z^{p/4}(1-z)^{p/4}}{\Gamma(1-\frac{p}{2})\Gamma(\frac{p}{2})}\sum_{k=0}^\infty \frac{(\frac{p}{2})_k(\frac{3p}{2}-1)_k}{k! (k+p-1)!} z^k(\log z+C_{p,k})
\end{align*}
for certain $C_{p,k}\in\mathbb{C}$. Since $-2h_{1,2}=1-\frac{3p}{4}$ is the lowest power of $z$ in this expansion, only the lowest power of $z$ contributes to the coefficient of $z^{-2h_{1,2}}$ in \eqref{eqn:series_for_rigidity}. As a result,
\begin{equation*}
\langle v_{1,2}',\mathfrak{R}(v_{1,2})\rangle =\frac{(p-2)!}{\Gamma(\frac{p}{2})\Gamma(\frac{3p}{2}-1)}\left(\frac{4}{p}(-2h_{1,2})-1\right)=\frac{4(\frac{1}{p}-1)(p-2)!}{\Gamma(\frac{p}{2})\Gamma(\frac{3p}{2}-1)}\neq 0.
\end{equation*}
We can now rescale either the evaluation or coevaluation for $\mathcal{K}_{1,2}$ so that $\mathfrak{R}=\mathrm{Id}_{\mathcal{K}_{1,2}}$.

It remains to show that $\mathfrak{R}=\mathfrak{L}$. To do so, we apply the braiding isomorphisms $\mathcal{R}$ to the composition $\mathfrak{L}$, obtaining the following commutative diagram; the rectangles commute thanks to the naturality of the braiding and the hexagon axioms:
\begin{equation*}
\xymatrixrowsep{1.5pc}
\xymatrix{
\mathcal{K}_{1,2} \ar[d]^{\cong} \ar[rd]^{\cong} & &\\
\mathcal{K}_{1,1}\boxtimes\mathcal{K}_{1,2} \ar[d]^{i\boxtimes\mathrm{Id}} \ar[r]^{\mathcal{R}} & \mathcal{K}_{1,2}\boxtimes\mathcal{K}_{1,1} \ar[d]^{\mathrm{Id}\boxtimes i} &   \\
(\mathcal{K}_{1,2}\boxtimes\mathcal{K}_{1,2})\boxtimes\mathcal{K}_{1,2} \ar[d]^{\cong} 
\ar[r]^{\mathcal{R}} & \mathcal{K}_{1,2}\boxtimes(\mathcal{K}_{1,2}\boxtimes\mathcal{K}_{1,2}) \ar[r]^{\mathrm{Id}\boxtimes\mathcal{R}} & \mathcal{K}_{1,2}\boxtimes(\mathcal{K}_{1,2}\boxtimes\mathcal{K}_{1,2}) \ar[d]^{\cong} \\
\mathcal{K}_{1,2}\boxtimes(\mathcal{K}_{1,2}\boxtimes\mathcal{K}_{1,2}) \ar[d]^{\mathrm{Id}\boxtimes e} \ar[r]^{\mathcal{R}} & (\mathcal{K}_{1,2}\boxtimes\mathcal{K}_{1,2})\boxtimes\mathcal{K}_{1,2} \ar[d]^{e\boxtimes\mathrm{Id}} \ar[r]^{\mathcal{R}\boxtimes\mathrm{Id}} & (\mathcal{K}_{1,2}\boxtimes\mathcal{K}_{1,2})\boxtimes\mathcal{K}_{1,2}\\
\mathcal{K}_{1,2}\boxtimes\mathcal{K}_{1,1} \ar[r]^{\mathcal{R}} \ar[d]^{\cong} & \mathcal{K}_{1,1}\boxtimes\mathcal{K}_{1,2} \ar[ld]^\cong &\\
\mathcal{K}_{1,2} & &\\
}
\end{equation*}
Thus we will get $\mathfrak{L}=\mathfrak{R}$ as required if there is a non-zero scalar $c$ such that
\begin{equation}\label{eqn:twist_with_braiding}
\mathcal{R}\circ i = c\cdot i,\qquad e\circ\mathcal{R}=c\cdot e.
\end{equation}
To determine $c$, the definitions of $\mathcal{R}$ \eqref{eqn:braiding} and $i$ \eqref{eqn:coevaluation}, the $L_{-1}$-commutator formula, and the $L_0$-conjugation formula \cite[Proposition 3.36(b)]{HLZ2} imply
\begin{align*}
(\mathcal{R}\circ i)(\mathbf{1}) &=-(\mathcal{R}\circ\pi_0)(\mathcal{Y}_\boxtimes(v_{1,2},1)v_{1,2})+\frac{2q}{p}(\mathcal{R}\circ\pi_0)(\mathcal{Y}_{\boxtimes}(L_{-1} v_{1,2},1)v_{1,2})\nonumber\\
& =-\pi_0\left(e^{L_{-1}}\mathcal{Y}_\boxtimes(v_{1,2},e^{\pi i})v_{1,2}\right)+\frac{2q}{p}\pi_0\left(e^{L_{-1}}\mathcal{Y}_\boxtimes(v_{1,2},e^{\pi i})L_{-1}v_{1,2}\right)\nonumber\\
& =-\pi_0\left(\mathcal{Y}_\boxtimes(v_{1,2},e^{\pi i})v_{1,2}\right)-\frac{2q}{p}\pi_0\left(\mathcal{Y}_\boxtimes(L_{-1} v_{1,2},e^{\pi i})v_{1,2}\right)\nonumber\\
& =-\pi_0\left(e^{\pi i(L_0-2h_{1,2})}\mathcal{Y}_\boxtimes(v_{1,2},1)v_{1,2}\right)-\frac{2q}{p}\pi_0\left(e^{\pi i(L_0-2h_{1,2}-1)}\mathcal{Y}_{\boxtimes}(L_{-1} v_{1,2},1)v_{1,2}\right)\nonumber\\
& =e^{\pi i(L_0 - 2h_{1,2})}i(\mathbf{1}).
\end{align*}
Since the conformal weight of $i(\mathbf{1})$ is $0$ and since $\mathbf{1}$ generates $\mathcal{K}_{1,1}$ as a $\mathfrak{Vir}$-module, it follows that $\mathcal{R}\circ i=e^{-2\pi i h_{1,2}}\cdot i$, that is, $c=e^{-2\pi i h_{1,2}}$.

To show that the same value of $c$ works for the evaluation $e$, note that we may have rescaled the original evaluation candidate to ensure $\mathfrak{R}=\mathrm{Id}_{\mathcal{K}_{1,2}}$. Thus from \eqref{eqn:E_for_q>2} and \eqref{eqn:E_for_q=2}, we have
\begin{equation}\label{eqn:E_intw_op}
\mathcal{E}(v_{1,2},z)v_{1,2} =z^{-2h_{1,2}}\left(z^{L_0}v+z E(z)\right)
\end{equation}
where $E(z)\in\mathcal{K}_{1,1}[[z]]$ and $v\in\mathcal{K}_{1,1}$ is a non-zero singular vector of conformal weight $0$ (if $q\geq 3$) or $h_{1,3}$ (if $q=2$). We need to show that $\Omega(\mathcal{E})= e^{-2\pi i h_{1,2}}\cdot\mathcal{E}$ where $\Omega(\mathcal{E}):=e\circ\mathcal{R}\circ\mathcal{Y}_\boxtimes$ is characterized by
\begin{equation*}
\Omega(\mathcal{E})(w,z)w'=e^{z L_{-1}}\mathcal{E}(w',e^{\pi i}z)w
\end{equation*}
for $w,w'\in\mathcal{K}_{1,2}$. Now, because $\mathfrak{R}=\mathrm{Id}_{\mathcal{K}_{1,2}}$, the coevaluation induces an injective linear map
\begin{equation*}
\mathrm{Hom}(\mathcal{K}_{1,2}\boxtimes\mathcal{K}_{1,2},\mathcal{K}_{1,1})\hookrightarrow\mathrm{Hom}(\mathcal{K}_{1,2},\mathcal{K}_{1,2}\boxtimes\mathcal{K}_{1,1})\cong\mathrm{End}\,\mathcal{K}_{1,2}=\mathbb{C}\cdot\mathrm{Id}_{\mathcal{K}_{1,2}}.
\end{equation*}
Thus the universal property of tensor products shows that the space of intertwining operators of type $\binom{\mathcal{K}_{1,1}}{\mathcal{K}_{1,2}\,\mathcal{K}_{1,2}}$ is one-dimensional, and thus $\Omega(\mathcal{E})=\tilde{c}\cdot\mathcal{E}$ for some $\tilde{c}\in\mathbb{C}$. We can then determine $\tilde{c}$ by using \eqref{eqn:E_intw_op} to calculate
\begin{align*}
\Omega(\mathcal{E})(v_{1,2},z)v_{1,2} & = e^{z L_{-1}}\mathcal{E}(v_{1,2}, e^{\pi i} z)v_{1,2}\nonumber\\
& = (e^{\pi i} z)^{-2h_{1,2}}e^{zL_{-1}}\left((e^{\pi i} z)^{L_0} v -z E(-z)\right) \nonumber\\
&= e^{\pi i(\mathrm{wt}\,v-2h_{1,2})}\cdot z^{-2h_{1,2}}\left(z^{L_0} v+z\widetilde{E}(z)\right)
\end{align*}
for some $\widetilde{E}(z)\in\mathcal{K}_{1,1}[[z]]$. Since $\mathrm{wt}\,v=0$ if $q\geq 3$ and $\mathrm{wt}\,v=h_{1,3}=p-1\in 2\mathbb{Z}$ if $q=2$, we get $\tilde{c}=e^{-2\pi ih_{1,2}}$ as required. This completes the proof that $\mathfrak{L}=\mathfrak{R}$ and thus completes the proof of the theorem.
\end{proof}

If $X$ is any rigid self-dual object in a tensor category with unit object $\mathbf{1}$, equipped with evaluation $e_X: X\boxtimes X\rightarrow\mathbf{1}$ and coevaluation $i_X: \mathbf{1}\rightarrow X\boxtimes X$, the intrinsic dimension of $X$ is defined to be $d(X)=e_X\circ i_X\in\mathrm{End}\,\mathbf{1}$. As in \cite[Proposition 4.2]{MY-cp1-vir}, we can use the proof of Theorem \ref{thm:rigidity} to find the intrinsic dimensions of $\mathcal{K}_{1,2}$ and $\mathcal{K}_{2,1}$:
\begin{Proposition}\label{prop:intrinsic_dim}
The rigid self-dual objects $\mathcal{K}_{1,2}$ and $\mathcal{K}_{2,1}$ in $\mathcal{O}_{c_{p,q}}$ have intrinsic dimensions
\begin{align*}
d(\mathcal{K}_{1,2}) & =-(e^{\pi i p/q}+e^{-\pi ip/q})\cdot\mathrm{Id}_{\mathcal{K}_{1,1}} = -2\cos(p\pi/q)\cdot\mathrm{Id}_{\mathcal{K}_{1,1}},\nonumber\\
d(\mathcal{K}_{2,1}) & =-(e^{\pi i q/p}+e^{-\pi iq/p})\cdot\mathrm{Id}_{\mathcal{K}_{1,1}} = -2\cos(q\pi/p)\cdot\mathrm{Id}_{\mathcal{K}_{1,1}}.
\end{align*}
In particular, $d(\mathcal{K}_{1,2})=0$ if $q=2$ and $d(\mathcal{K}_{2,1})=0$ if $p=2$.
\end{Proposition}
\begin{proof}
Because $c_{p,q}=c_{q,p}$, it is enough to calculate $d(\mathcal{K}_{1,2})=e\circ i$, where $i$ is given by \eqref{eqn:coevaluation}. If $q\geq 3$, then \eqref{eqn:E_for_q>2} and \eqref{eqn:R_coefficient} show that $\mathcal{E}=e\circ\mathcal{Y}_\boxtimes$ must then be normalized to satisfy
\begin{equation*}
\mathcal{E}(v_{1,2},z)v_{1,2}=\frac{\cos(p\pi/q)}{2-q/p}\cdot z^{-2h_{1,2}}(\mathbf{1}+z E(z))
\end{equation*}
for some $E(z)\in\mathcal{K}_{1,1}[[z]]$. Thus
\begin{align*}
(e\circ i)(\mathbf{1}) & = -\pi_0(\mathcal{E}(v_{1,2},1)v_{1,2})+\frac{2q}{p}\pi_0(\mathcal{E}(L_{-1}v_{1,2},1)v_{1,2})\nonumber\\
& = -\frac{\cos(p\pi/q)}{2-q/p}\cdot\mathbf{1}+\frac{2q}{p}\pi_0(L_0\mathcal{E}(v_{1,2},1)v_{1,2}-\mathcal{E}(L_0 v_{1,2},1)v_{1,2}-\mathcal{E}(v_{1,2},1)L_0v_{1,2})\nonumber\\
& =\frac{\cos(p\pi/q)}{2-q/p}\left(-1-\frac{4q h_{1,2}}{p}\right)\cdot\mathbf{1}\nonumber\\
& =-2\cos(p\pi/q)\cdot\mathbf{1},
\end{align*}
where the second equality uses $n=0$ case of the commutator formula \eqref{eqn:Vir_comm_form}. Since $\mathbf{1}$ generates $\mathcal{K}_{1,1}$, the result follows when $q\geq 3$. When $q=2$, $e\circ i=0$ because $i(\mathbf{1})$ has conformal weight $0$ while the minimum conformal weight of $\mathrm{Im}\,e\subseteq\mathcal{K}_{1,1}$ is $h_{1,3}=p-1>0$.
\end{proof}

\subsection{Further (non)-rigidity for Kac modules}

We will now use the results of the previous two subsections to show that if $r\leq p$ and $s\leq q$, then the Kac module $\mathcal{K}_{r,s}$ is rigid and self-dual in $\mathcal{O}_{c_{p,q}}$. We will also show that at least some Kac modules $\mathcal{K}_{r,s}$ with $r>p$ or $s>q$ are not rigid. In particular, although $\mathcal{O}_{c_{p,q}}$ does contain a non-trivial class of rigid objects, it is not a rigid tensor category. We begin with a lemma:
\begin{Lemma}\label{lem:direct_summand_rigid}
Suppose $B\cong B_1\oplus B_2$ in some tensor category where both $B$ and $B_1$ are rigid and self-dual, and $\mathrm{Hom}(B_1,B_2)=0$. Then $B_2$ is also rigid and self-dual.
\end{Lemma}
\begin{proof}
From Corollary 3 in the Appendix of \cite{KL4}, an object $B$ in a tensor category equipped with a morphism $i_B:\mathbf{1}\rightarrow B\boxtimes B$ is rigid (and self-dual with coevaluation $i_B$) if and only if the linear map
\begin{equation*}
\psi_{A,B,C}: \mathrm{Hom}(B\boxtimes A,C)\longrightarrow\mathrm{Hom}(A,B\boxtimes C)
\end{equation*}
is an isomorphism for all objects $A$ and $C$, where for $f\in\mathrm{Hom}(B\boxtimes A, C)$, $\psi_{A,B,C}(f)$ is defined to be the following composition (where we suppress unit and associativity isomorphisms from the notation for brevity):
\begin{align*}
A\xrightarrow{i_B\boxtimes\mathrm{Id}} B\boxtimes B\boxtimes A\xrightarrow{\mathrm{Id}\boxtimes f} B\boxtimes C.
\end{align*}
In the setting of the lemma, we have such a rigid object $B$ with coevaluation $i_B:\mathbf{1}\rightarrow B\boxtimes B$ and isomorphism $\psi_{A,B,C}$ for all $A$ and $C$. For $k=1,2$, we also have projections $p_k: B\rightarrow B_k$ and inclusions $q_k: B_k\rightarrow B$ such that $p_k\circ q_{k'}=\delta_{k,k'}$ and $q_1\circ p_1+q_2\circ p_2=\mathrm{Id}$.

For $k=1,2$, we define $i_{B_k}=(p_k\boxtimes p_k)\circ i_B\in\mathrm{Hom}(\mathbf{1},B_k\boxtimes B_k)$; we will show that $\psi_{A,B_2,C}$ defined using $i_{B_2}$ is an isomorphism. To do so, consider the diagram with exact rows:
\begin{equation*}
\xymatrixcolsep{2pc}
\xymatrix{
0 \ar[r] & \mathrm{Hom}(B_1\boxtimes A, C) 
\ar[d]^{\psi_{A,B_1,C}} \ar[rr]^{f\mapsto f\circ(p_1\boxtimes\mathrm{Id})} && \mathrm{Hom}(B\boxtimes A,C) \ar[d]^{\psi_{A,B,C}} \ar[rr]^{g\mapsto g\circ(q_2\boxtimes\mathrm{Id})} && \mathrm{Hom}(B_2\boxtimes A,C) \ar[d]^{\psi_{A,B_2,C}} \ar[r] & 0\\
0 \ar[r] & \mathrm{Hom}(A, B_1\boxtimes C) \ar[rr]^{f\mapsto (q_1\boxtimes\mathrm{Id})\circ f} && \mathrm{Hom}(A, B\boxtimes C) \ar[rr]^{g\mapsto (p_2\boxtimes\mathrm{Id})\circ g} && \mathrm{Hom}(A, B_2\boxtimes C) \ar[r] & 0\\
}
\end{equation*}
To show that the left square commutes, take $f\in\mathrm{Hom}(B_1\boxtimes A, C)$ and calculate
\begin{align*}
(q_1\boxtimes\mathrm{Id})\circ\psi_{A,B_1,C}(f) & = (q_1\boxtimes\mathrm{Id})\circ(\mathrm{Id}\boxtimes f)\circ (p_1\boxtimes p_1\boxtimes\mathrm{Id})\circ(i_B\boxtimes\mathrm{Id})\nonumber\\
& =(\mathrm{Id}\boxtimes f)\circ(\mathrm{Id}\boxtimes p_1\boxtimes\mathrm{Id})\circ(i_B\boxtimes\mathrm{Id})\nonumber\\
&\qquad - (q_2\boxtimes\mathrm{Id})\circ(\mathrm{Id}\boxtimes f)\circ (p_2\boxtimes p_1\boxtimes\mathrm{Id})\circ(i_B\boxtimes\mathrm{Id})
\end{align*}
since $q_1\circ p_1=\mathrm{Id}-q_2\circ p_2$. The first term on the right is $\psi_{A,B,C}(f\circ(p_1\boxtimes\mathrm{Id}))$, while the second vanishes because
\begin{equation*}
(p_2\boxtimes p_1)\circ i_B\in\mathrm{Hom}(\mathbf{1},B_2\boxtimes B_1)\cong\mathrm{Hom}(B_1,B_2) = 0
\end{equation*}
by the self-duality of $B_1$. Thus the left square of the diagram commutes, and the right square commutes for similar reasons.

Now, $\psi_{A,B,C}$ is an isomorphism because $B$ is rigid. Although we do not know if $\psi_{A,B_1,C}$ is defined using the actual coevaluation for $B_1$, we do know from commutativity of the diagram that  $\psi_{A,B_1,C}$ is injective and $\psi_{A,B_2,C}$ is surjective. Then because $B\cong B_1\oplus B_2$, we can exchange $B_1$ and $B_2$ everywhere in the previous paragraph, so  in particular $\psi_{A,B_2,C}$ is also injective and thus an isomorphism. That is, $B_2$ is rigid and self-dual.
\end{proof}

Using Lemma \ref{lem:direct_summand_rigid}, we can now prove:
\begin{Theorem}\label{thm:Krs_rigid}
If $r\leq p$ and $s\leq q$, then $\mathcal{K}_{r,s}$ is rigid and self-dual in $\mathcal{O}_{c_{p,q}}$.
\end{Theorem}
\begin{proof}
By Remark \ref{rem:K11_rigid} and Theorem \ref{thm:rigidity}, $\mathcal{K}_{1,1}$ and $\mathcal{K}_{1,2}$ are rigid and self-dual. Now assume by induction on $s$ that $\mathcal{K}_{1,s-1}$ and $\mathcal{K}_{1,s}$ are rigid and self-dual for some $s\in\lbrace 2,3,\ldots, q-1\rbrace$. Then using Corollary \ref{cor:1st_Kac_module_fusion}, $\mathcal{K}_{1,2}\boxtimes\mathcal{K}_{1,s}\cong\mathcal{K}_{1,s-1}\oplus\mathcal{K}_{1,s+1}$ is rigid and self-dual (see for example \cite[Lemma A.3(b)]{KL4}), and $\mathrm{Hom}(\mathcal{K}_{1,s-1},\mathcal{K}_{1,s+1})=0$. Thus $\mathcal{K}_{1,s+1}$ is rigid and self-dual by Lemma \ref{lem:direct_summand_rigid}, and it follows by induction that $\mathcal{K}_{1,s}$ is rigid and self-dual for all $1\leq s\leq q$. Similarly, $\mathcal{K}_{r,1}$ is rigid and self-dual for all $1\leq r\leq p$, and then using Proposition \ref{prop:Kr1_K1s_small_r_or_s}, $\mathcal{K}_{r,s}\cong\mathcal{K}_{r,1}\boxtimes\mathcal{K}_{1,s}$ is rigid and self-dual.
\end{proof}

It turns out that rigidity for the whole category $\mathcal{O}_{c_{p,q}}$ is obstructed by the subcategory of $L_{c_{p,q}}$-modules, that is, the subcategory consisting of finite direct sums of the modules $\mathcal{L}_{r,s}$ with $r<p$ and $s<q$. To exhibit some non-rigid objects of $\mathcal{O}_{c_{p,q}}$, we first need a lemma:
\begin{Lemma}\label{lem:L_tensor_ideal}
Suppose that $\mathcal{Y}$ is a surjective $V_{c_{p,q}}$-module intertwining operator of type $\binom{W_3}{W_1\,W_2}$ where $W_1$ is an $L_{c_{p,q}}$-module and $W_2$, $W_3$ are objects of $\mathcal{O}_{c_{p,q}}$. Then $W_3$ is an $L_{c_{p,q}}$-module.
\end{Lemma}
\begin{proof}
By assumption, $W_3$ is spanned by coefficients of powers of $z$ and $\log z$ in $\mathcal{Y}(w_1,z)w_2=\sum_{h\in\mathbb{C}}\sum_{k\in\mathbb{N}} (w_1)_{h;k} w_2\,z^{-h-1}(\log z)^k$ as $w_1$ and $w_2$ run over $W_1$ and $W_2$, respectively. Thus to show that $W_3$ is an $L_{c_{p,q}}$-module we need to show that $v_n(w_1)_{h;k} w_2=0$ for all $w_1\in W_1$, $w_2\in W_2$, $h\in\mathbb{C}$, $k\in\mathbb{N}$, $n\in\mathbb{Z}$, and $v$ in the maximal proper ideal of $V_{c_{p,q}}$. In fact, the easy intertwining operator generalization of \cite[Proposition 4.5.7]{LL} shows that $v_n(w_1)_{h;k} w_2$ is a linear combination of vectors $(v_m w_1)_{j;k} w_2$ for $m\in\mathbb{Z}$ and $j\in\mathbb{C}$. Since $W_1$ is an $L_{c_{p,q}}$-module, each $v_m w_1=0$, and thus each $v_n(w_1)_{h;k}w_2=0$ as well.
\end{proof}

As a result, we have:
\begin{Proposition}
If $L$ is a non-zero $L_{c_{p,q}}$-module, then $L$ is not rigid in $\mathcal{O}_{c_{p,q}}$.
\end{Proposition}
\begin{proof}
If $L$ were rigid, we would have a non-zero evaluation $e_L: L^*\boxtimes L\rightarrow\mathcal{K}_{1,1}$ with $L^*\boxtimes L$ an $L_{c_{p,q}}$-module by Lemma \ref{lem:L_tensor_ideal}. But in fact there is no non-zero map from any $L_{c_{p,q}}$-module to $\mathcal{K}_{1,1}$ because the unique irreducible submodule $\mathcal{L}_{1,2q-1}\subseteq\mathcal{K}_{1,1}$ is not an $L_{c_{p,q}}$-module.
\end{proof}

We can also use Lemma \ref{lem:L_tensor_ideal} to prove that certain Kac modules are not rigid:
\begin{Proposition}\label{prop:some_Krs_not_rigid}
For $1\leq r\leq p-1$, the Kac module $\mathcal{K}_{r,q+1}$ is not rigid in $\mathcal{O}_{c_{p,q}}$, and for $1\leq s\leq q-1$, $\mathcal{K}_{p+1,s}$ is not rigid.
\end{Proposition}
\begin{proof}
By $c_{p,q}=c_{q,p}$ symmetry, it is enough to show that $\mathcal{K}_{r,q+1}$ is not rigid. We will show in Proposition \ref{prop:K12_times_simple_indecomp} (without using non-rigidity of $\mathcal{K}_{r,q+1}$) that there is an exact sequence
\begin{equation*}
0\longrightarrow\mathcal{K}_{r,q-1}\longrightarrow\mathcal{K}_{1,2}\boxtimes\mathcal{K}_{r,q}\longrightarrow\mathcal{K}_{r,q+1}\longrightarrow 0.
\end{equation*}
We can combine this with the exact sequence
\begin{equation*}
0\longrightarrow\mathcal{L}_{1,2q-1}\longrightarrow\mathcal{K}_{1,1}\longrightarrow\mathcal{L}_{1,1}\longrightarrow 0
\end{equation*}
to obtain the following commutative diagram with (right) exact rows and columns as indicated and arrows given labels for future reference (recall the Nine, or $3\times 3$, Lemma):
\begin{equation*}
\xymatrixrowsep{1.25pc}
\xymatrixcolsep{1.25pc}
\xymatrix{
& 0 \ar[d] & 0 \ar[d] & &\\
&\mathcal{K}_{r,q-1}\boxtimes\mathcal{L}_{1,2q-1} \ar[d]^{g_1}\ar[r]^(.45){f_1} & \mathcal{K}_{1,2}\boxtimes\mathcal{K}_{r,q}\boxtimes\mathcal{L}_{1,2q-1} \ar[d]^{g_3}\ar[r]^(.53){f_2} & \mathcal{K}_{r,q+1}\boxtimes\mathcal{L}_{1,2q-1} \ar[d]^{g_5}\ar[r] & 0\\
0 \ar[r] &\mathcal{K}_{r,q-1}\boxtimes\mathcal{K}_{1,1} \ar[d]^{g_2}\ar[r]^(.45){f_3} & \mathcal{K}_{1,2}\boxtimes\mathcal{K}_{r,q}\boxtimes\mathcal{K}_{1,1} \ar[d]^{g_4}\ar[r]^(.53){f_4} & \mathcal{K}_{r,q+1}\boxtimes\mathcal{K}_{1,1} \ar[d]^{g_6}\ar[r] & 0\\
&\mathcal{K}_{r,q-1}\boxtimes\mathcal{L}_{1,1} \ar[d]\ar[r]^(.45){f_5} & \mathcal{K}_{1,2}\boxtimes\mathcal{K}_{r,q}\boxtimes\mathcal{L}_{1,1} \ar[d]\ar[r]^(.53){f_6} & \mathcal{K}_{r,q+1}\boxtimes\mathcal{L}_{1,1} \ar[d]\ar[r] & 0\\
& 0 & 0 & 0 &\\
}
\end{equation*}
Now supposing for the sake of contradiction that $\mathcal{K}_{r,q+1}$ were rigid, then it would be flat (see for example Corollary 1 from the Appendix of \cite{KL4}), and thus the rightmost column in the diagram would be exact, not just right exact. By a diagram chase, the map $f_5: \mathcal{K}_{r,q-1}\boxtimes\mathcal{L}_{1,1}\rightarrow\mathcal{K}_{1,2}\boxtimes\mathcal{K}_{r,q}\boxtimes\mathcal{L}_{1,1}$ would then be injective. 

Indeed, if $w_0\in\mathrm{Ker}\,f_5$, choose $w_1\in\mathcal{K}_{r,q-1}\boxtimes\mathcal{K}_{1,1}$ such that $g_2(w_1)=w_0$, so that
\begin{equation*}
g_4(f_3(w_1)) = f_5(g_2(w_1))=f_5(w_0)=0.
\end{equation*}
Thus $f_3(w_1)\in\mathrm{Ker}\,g_4 =\mathrm{Im}\,g_3$, that is, there exists $w_2\in\mathcal{K}_{1,2}\boxtimes\mathcal{K}_{r,q}\boxtimes\mathcal{L}_{1,2q-1}$ such that $g_3(w_2)=f_3(w_1)$. Now,
\begin{equation*}
g_5(f_2(w_2))=f_4(g_3(w_2))=f_4(f_3(w_1))=0,
\end{equation*}
so $f_2(w_2)=0$ since we are assuming $g_5$ is injective. That is, $w_2\in\mathrm{Ker}\,f_2=\mathrm{Im}\,f_1$, and there exists $w_3\in\mathcal{K}_{r,q-1}\boxtimes\mathcal{L}_{1,2q-1}$ such that $f_1(w_3)=w_2$. We then get 
\begin{equation*}
f_3(g_1(w_3))=g_3(f_1(w_3))=g_3(w_2)=f_3(w_1),
\end{equation*}
so $g_1(w_3)=w_1$ because $f_3$ is injective. Finally, we get $w_0=g_2(w_1)=g_2(g_1(w_3))=0$, proving $\mathrm{Ker}\,f_5=0$ as desired.

Now, Lemma \ref{lem:L_tensor_ideal} says that $\mathcal{K}_{r,q}\boxtimes\mathcal{L}_{1,1}$ is an $L_{c_{p,q}}$-module, and there is also a surjection
\begin{equation*}
\mathcal{K}_{r,q}\xrightarrow{\cong}\mathcal{K}_{r,q}\boxtimes\mathcal{K}_{1,1}\twoheadrightarrow\mathcal{K}_{r,q}\boxtimes\mathcal{L}_{1,1}.
\end{equation*}
Since $\mathcal{K}_{r,q}$ is simple and not an $L_{c_{p,q}}$-module, it follows that 
\begin{equation*}
\mathcal{K}_{1,2}\boxtimes\mathcal{K}_{r,q}\boxtimes\mathcal{L}_{1,1} =\mathcal{K}_{1,2}\boxtimes 0 = 0.
\end{equation*}
On the other hand, $\mathcal{L}_{r,q-1}$ is an $L_{c_{p,q}}$-module, so there is a non-zero intertwining operator $Y_{\mathcal{L}_{r,q-1}}$ of type $\binom{\mathcal{L}_{r,q-1}}{L_{c_{p,q}}\,\mathcal{L}_{r,q-1}}$, which leads to a surjection
\begin{equation*}
\mathcal{K}_{r,q-1}\boxtimes\mathcal{L}_{1,1}\twoheadrightarrow\mathcal{L}_{r,q-1}\boxtimes\mathcal{L}_{1,1}\cong\mathcal{L}_{1,1}\boxtimes\mathcal{L}_{r,q-1}\twoheadrightarrow\mathcal{L}_{r,q-1}\neq 0.
\end{equation*}
Thus there is no injection $\mathcal{K}_{r,q-1}\boxtimes\mathcal{L}_{1,1}\rightarrow\mathcal{K}_{1,2}\boxtimes\mathcal{K}_{r,q}\boxtimes\mathcal{L}_{1,1}$, contradicting the assumption that $\mathcal{K}_{r,q+1}$ is rigid.
\end{proof}

We leave the question of whether the remaining Kac modules are rigid (or not) in $\mathcal{O}_{c_{p,q}}$ open for now, since it is not important for the remaining results in this paper. 


\section{Main results on fusion tensor products of Kac modules}\label{sec:main_results}

In this section, we prove the main theorems on fusion tensor products of the modules $\mathcal{K}_{r,s}$ for $r,s\in\mathbb{Z}_{\geq 1}$. We mainly focus on how $\mathcal{K}_{1,2}$ and $\mathcal{K}_{2,1}$ tensor with general $\mathcal{K}_{r,s}$, since it follows from Theorem \ref{thm:Fock_module_associativity} that $\mathcal{K}_{r,s}$ for any $r,s\in\mathbb{Z}_{\geq 1}$ can be realized as a homomorphic image of $\mathcal{K}_{2,1}^{\boxtimes(r-1)}\boxtimes\mathcal{K}_{1,2}^{\boxtimes(s-1)}$.

\subsection{Tensor products of \texorpdfstring{$\mathcal{K}_{1,2}$}{K12} and \texorpdfstring{$\mathcal{K}_{2,1}$}{K21} with Kac modules}

Here we determine how $\mathcal{K}_{1,2}$ and $\mathcal{K}_{2,1}$ tensor with all Kac modules $\mathcal{K}_{r,s}$, $r,s\in\mathbb{Z}_{\geq 1}$. Recall that Corollary \ref{cor:1st_Kac_module_fusion} and Remark \ref{rem:1st_Kac_module_fusion} give the first results in this direction; we will generalize these results here.

Because $\mathcal{K}_{1,2}$ and $\mathcal{K}_{2,1}$ are rigid, for any $\mathfrak{Vir}$-modules $W_1$ and $W_2$ in $\mathcal{O}_{c_{p,q}}$, there are natural isomorphisms
\begin{align*}
\mathrm{Hom}(\mathcal{K}_{1,2}\boxtimes W_1,W_2) &\cong\mathrm{Hom}(W_1,\mathcal{K}_{1,2}\boxtimes W_2)\nonumber\\
\mathrm{Hom}(\mathcal{K}_{2,1}\boxtimes W_1,W_2) &\cong\mathrm{Hom}(W_1,\mathcal{K}_{2,1}\boxtimes W_2)
\end{align*}
(see for example \cite[Lemma 2.1.6]{BK}). In particular, recall from Theorem \ref{thm:Kac_module_surj} that for any $r,s\in\mathbb{Z}_{\geq 1}$, there is a surjection $F_{r,s}^+: \mathcal{K}_{1,2}\boxtimes\mathcal{K}_{r,s}\rightarrow\mathcal{K}_{r,s+1}$ (that is, $F_{r,s}^+=F_{1,2;r,s}$ in the notation of Theorem \ref{thm:Fock_module_associativity}). Thus for $s\geq 2$, rigidity of $\mathcal{K}_{1,2}$ yields a non-zero map
\begin{equation*}
F_{r,s}^-: \mathcal{K}_{r,s-1}\longrightarrow\mathcal{K}_{1,2}\boxtimes\mathcal{K}_{r,s}
\end{equation*}
given specifically by the following composition (where for brevity we suppress unit and associativity isomorphisms from the notation):
\begin{align*}
\mathcal{K}_{r,s-1}\xrightarrow{i_{\mathcal{K}_{1,2}}\boxtimes\mathrm{Id}} \mathcal{K}_{1,2}\boxtimes\mathcal{K}_{1,2}\boxtimes\mathcal{K}_{r,s-1}\xrightarrow{\mathrm{Id}\boxtimes F_{r,s-1}^+} \mathcal{K}_{1,2}\boxtimes\mathcal{K}_{r,s}.
\end{align*}
Of course, setting $\mathcal{K}_{r,0}=0$, we can also set $F_{r,1}^-=0:\mathcal{K}_{r,0}\rightarrow\mathcal{K}_{1,2}\boxtimes\mathcal{K}_{r,1}$.
\begin{Proposition}\label{prop:partial_exact_for_gen_rs}
For all $r,s\in\mathbb{Z}_{\geq 1}$, $F_{r,s}^-$ is injective and $\mathrm{Im}\,F_{r,s}^-\subseteq\mathrm{Ker}\,F_{r,s}^+$.
\end{Proposition}
\begin{proof}
The case $s=1$ is trivial. For $s\geq 2$, let $k:\mathrm{Ker}\,F_{r,s}^-\rightarrow\mathcal{K}_{r,s-1}$ be the inclusion map. Then the following composition vanishes (where we again suppress unit and associativity isomorphisms from the notation):
\begin{align*}
\mathcal{K}_{1,2}\boxtimes\mathrm{Ker}\,F_{r,s}^-\xrightarrow{\mathrm{Id}\boxtimes k}\mathcal{K}_{1,2}\boxtimes\mathcal{K}_{r,s-1}\xrightarrow{\mathrm{Id}\boxtimes F_{r,s}^-} \mathcal{K}_{1,2}\boxtimes\mathcal{K}_{1,2}\boxtimes\mathcal{K}_{r,s}
 \xrightarrow{e_{\mathcal{K}_{1,2}}\boxtimes\mathrm{Id}}\mathcal{K}_{r,s}.
\end{align*}
By the definition of $F_{r,s}^-$, this composition is
\begin{align*}
\mathcal{K}_{1,2}  \boxtimes\mathrm{Ker}\,F_{r,s}^-& \xrightarrow{\mathrm{Id}\boxtimes k}\mathcal{K}_{1,2}\boxtimes\mathcal{K}_{r,s-1}\xrightarrow{\mathrm{Id}\boxtimes i_{\mathcal{K}_{1,2}}\boxtimes\mathrm{Id}}\mathcal{K}_{1,2}\boxtimes\mathcal{K}_{1,2}\boxtimes\mathcal{K}_{1,2}\boxtimes\mathcal{K}_{r,s-1}\nonumber\\
& \xrightarrow{\mathrm{Id}\boxtimes\mathrm{Id}\boxtimes F_{r,s-1}^+} \mathcal{K}_{1,2}\boxtimes\mathcal{K}_{1,2}\boxtimes\mathcal{K}_{r,s}\xrightarrow{ e_{\mathcal{K}_{1,2}}\boxtimes\mathrm{Id}} \mathcal{K}_{1,2}\boxtimes\mathcal{K}_{r,s}.
\end{align*}
After exchanging the $F_{r,s-1}^+$ and evaluation in the last two arrows (which involves naturality of the associativity and unit isomorphisms that have been suppressed from the notation), and then applying rigidity of $\mathcal{K}_{1,2}$, this composition is just
\begin{equation*}
\mathcal{K}_{1,2}\boxtimes\mathrm{Ker}\,F_{r,s}^-\xrightarrow{\mathrm{Id}\boxtimes k}\mathcal{K}_{1,2}\boxtimes\mathcal{K}_{r,s-1}\xrightarrow{F_{r,s-1}^+} \mathcal{K}_{r,s}.
\end{equation*}
Since $F_{r,s-1}^+$ is defined using a non-zero Heisenberg Fock module intertwining operator $\mathcal{Y}$ of type $\binom{\mathcal{F}_{r,s}}{\mathcal{F}_{1,2}\,\mathcal{F}_{r,s-1}}$, this means that $\mathcal{Y}\vert_{\mathcal{K}_{1,2}\otimes\mathrm{Ker}\,F_{r,s}^-}=0$. Then because $\mathcal{Y}$ is a non-zero $\mathcal{F}_0$-module intertwining operator among irreducible $\mathcal{F}_0$-modules, \cite[Proposition 11.9]{DL} implies that $\mathrm{Ker}\,F_{r,s}^-=0$, that is, $F_{r,s}^-$ is injective.

We also need to show that $F_{r,s}^+\circ F_{r,s}^-=0$. By the definition of $F_{r,s}^-$ and the commutative diagram in Theorem \ref{thm:Fock_module_associativity}, this composition is a non-zero multiple of
\begin{equation*}
\mathcal{K}_{r,s-1}\xrightarrow{i_{\mathcal{K}_{1,2}}\boxtimes\mathrm{Id}}  \mathcal{K}_{1,2}\boxtimes\mathcal{K}_{1,2}\boxtimes\mathcal{K}_{r,s-1}\xrightarrow{F_{1,2}^+\boxtimes\mathrm{Id}} \mathcal{K}_{1,3}\boxtimes\mathcal{K}_{r,s-1}\xrightarrow{F_{1,3;r,s-1}} \mathcal{K}_{r,s+1}.
\end{equation*}
The first two arrows involve a map $\mathcal{K}_{1,1}\rightarrow\mathcal{K}_{1,3}$ which must vanish because $0$ is not a conformal weight of $\mathcal{K}_{1,3}$, and therefore this map sends the generating vacuum vector in $\mathcal{K}_{1,1}$ to $0$. As a result, the entire composition above vanishes, showing $\mathrm{Im}\,F_{r,s}^-\subseteq\mathrm{Ker}\,F_{r,s}^+$.
\end{proof}

\begin{Remark}\label{rem:partial_exact_for_general_rs}
By $c_{p,q}= c_{q,p}$ symmetry, we also have for any $r,s\in\mathbb{Z}_{\geq 1}$ an injection $G_{r,s}^-:\mathcal{K}_{r-1,s}\rightarrow\mathcal{K}_{2,1}\boxtimes\mathcal{K}_{r,s}$ and a surjection $G_{r,s}^+:\mathcal{K}_{2,1}\boxtimes\mathcal{K}_{r,s}\rightarrow\mathcal{K}_{r+1,s}$ such that $\mathrm{Im}\,G_{r,s}^-\subseteq\mathrm{Ker}\,G_{r,s}^+$.
\end{Remark}

\begin{Remark}
Note that Proposition \ref{prop:partial_exact_for_gen_rs} and Remark \ref{rem:partial_exact_for_general_rs} are consistent with the results of Corollary \ref{cor:1st_Kac_module_fusion} for the special cases of $r$ and $s$ considered there.
\end{Remark}

The main result of this section will be a strengthening of Proposition \ref{prop:partial_exact_for_gen_rs} and Remark \ref{rem:partial_exact_for_general_rs}: we will show that in fact $\mathrm{Im}\,F_{r,s}^-=\mathrm{Ker}\,F_{r,s}^+$ and $\mathrm{Im}\,G_{r,s}^-=\mathrm{Ker}\,G_{r,s}^+$. By Corollary \ref{cor:1st_Kac_module_fusion} and Remark \ref{rem:1st_Kac_module_fusion}, the first of these claims already holds for $r\leq p$ and $q\nmid s$, and the second already holds for $s\leq q$ and $p\nmid r$. The next proposition extends these results to the $r\leq p$, $q\mid s$ and $s\leq q$, $p\mid r$ cases:
\begin{Proposition}\label{prop:K12_times_simple_indecomp}
If $1\leq r\leq p$ and $n\geq 1$, then there is a non-split exact sequence
\begin{equation*}
0\longrightarrow\mathcal{K}_{r,nq-1}\longrightarrow\mathcal{K}_{1,2}\boxtimes\mathcal{K}_{r,nq}\longrightarrow\mathcal{K}_{r,nq+1}\longrightarrow 0.
\end{equation*}
Similarly, if $1\leq s\leq q$ and $m\geq 1$, then there is a non-split exact sequence
\begin{equation*}
0\longrightarrow\mathcal{K}_{mp-1,s}\longrightarrow\mathcal{K}_{2,1}\boxtimes\mathcal{K}_{mp,s}\longrightarrow\mathcal{K}_{mp+1,s}\longrightarrow 0.
\end{equation*}
\end{Proposition}
\begin{proof}
Because $c_{p,q}=c_{q,p}$, it is enough to prove the first assertion of the proposition. We take $r=1$ first.
Since $\mathcal{K}_{1,nq}$ is the simple quotient of $\mathcal{V}_{1,nq}$, Proposition \ref{prop:lowest_weight_spaces} implies that the lowest conformal weight(s) of $\mathcal{K}_{1,2}\boxtimes\mathcal{K}_{1,nq}$ are $h_{1,nq\pm 1}$. We have $h_{1,nq+1}-h_{1,nq-1}=pn-1\in\mathbb{Z}_{\geq 1}$. Thus we may write $\mathcal{K}_{1,2}\boxtimes\mathcal{K}_{1,nq}=\bigoplus_{k=0}^\infty (\mathcal{K}_{1,2}\boxtimes\mathcal{K}_{1,nq})(k)$ where $(\mathcal{K}_{1,2}\boxtimes\mathcal{K}_{r,nq})(k)$ is the conformal weight space with (generalized) $L_0$-eigenvalue $h_{1,nq-1}+k$. 

Since $h_{1,nq-1}<h_{1,nq+1}$, Proposition \ref{prop:lowest_weight_spaces} implies that $\dim\,(\mathcal{K}_{1,2}\boxtimes\mathcal{K}_{1,nq})(0)\leq 1$. Thus the submodule $\big\langle(\mathcal{K}_{1,2}\boxtimes\mathcal{K}_{1,nq})(0)\big\rangle$ generated by the lowest weight space is some Verma module quotient $\mathcal{V}_{1,nq-1}/\mathcal{J}^-$. Moreover, the quotient
\begin{equation}\label{eqn:quotient}
\mathcal{K}_{1,2}\boxtimes\mathcal{K}_{1,nq}/\big\langle(\mathcal{K}_{1,2}\boxtimes\mathcal{K}_{1,nq})(0)\big\rangle
\end{equation}
has conformal weights in $h_{1,nq-1}+\mathbb{Z}_{\geq 1}$, and the tensor product intertwining operator
\begin{equation*}
\mathcal{Y}_\boxtimes :\mathcal{K}_{1,2}\otimes\mathcal{K}_{1,nq}\longrightarrow(\mathcal{K}_{1,2}\boxtimes\mathcal{K}_{1,nq})[\log z]\lbrace z\rbrace
\end{equation*}
induces a surjective intertwining operator of type $\binom{\mathcal{K}_{1,2}\boxtimes\mathcal{K}_{1,nq}/\left\langle(\mathcal{K}_{1,2}\boxtimes\mathcal{K}_{1,nq})(0)\right\rangle}{\mathcal{K}_{1,2}\quad\mathcal{K}_{1,nq}}$.  Thus by Proposition \ref{prop:lowest_weight_spaces} again, the lowest conformal weight of the quotient \eqref{eqn:quotient} is $h_{1,nq+1}$, and the dimension of its lowest weight space is at most $1$. Consequently, the lowest weight space of \eqref{eqn:quotient} generates some Verma module quotient $\mathcal{V}_{1,nq+1}/\mathcal{J}^+$. Moreover, the quotient of \eqref{eqn:quotient} by $\mathcal{V}_{1,nq+1}/\mathcal{J}^+$ vanishes because otherwise its conformal weights would be strictly larger than $h_{1,nq+1}$, contradicting Proposition \ref{prop:lowest_weight_spaces}. Thus there is a short exact sequence
\begin{equation}\label{eqn:short_ex_seq}
0\longrightarrow\mathcal{V}_{1,nq-1}/\mathcal{J}^-\longrightarrow\mathcal{K}_{1,2}\boxtimes\mathcal{K}_{1,nq}\longrightarrow\mathcal{V}_{1,nq+1}/\mathcal{J}^+\longrightarrow 0.
\end{equation}

We need to show that $\mathcal{V}_{1,nq\pm 1}/\mathcal{J}^\pm\cong\mathcal{K}_{1,nq\pm 1}$. First, note that both $\mathcal{V}_{1,nq-1}/\mathcal{J}^-$ and $\mathrm{Im}\,F_{1,nq}^-\cong\mathcal{K}_{1,nq-1}$ are generated by the one-dimensional lowest conformal weight space of $\mathcal{K}_{1,2}\boxtimes\mathcal{K}_{1,nq}$ (with conformal weight $h_{1,nq-1}$). Thus $\mathcal{V}_{1,nq-1}/\mathcal{J}^-\cong\mathcal{K}_{1,nq-1}$.

Next, we have now shown $\mathcal{V}_{1,nq-1}/\mathcal{J}^-\subseteq\mathrm{Ker}\,F_{1,nq}^+$, so $F_{1,nq}^+:\mathcal{K}_{1,2}\boxtimes\mathcal{K}_{1,nq}\rightarrow\mathcal{K}_{1,nq+1}$ descends to a surjective map $\mathcal{V}_{1,nq+1}/\mathcal{J}^+\rightarrow\mathcal{K}_{1,nq+1}$. We need to show that this map is an isomorphism. Recalling that $\mathcal{K}_{1,nq+1}\cong\mathcal{V}_{1,nq+1}/\langle\tilde{v}_{1,nq+1}\rangle$ where $\tilde{v}_{1,nq+1}$ is a singular vector of conformal weight $h_{1,nq+1}+nq+1$, we have $\mathcal{J}^+\subseteq\langle\tilde{v}_{1,nq+1}\rangle$, and we need to show the opposite inclusion. Indeed, since $\mathcal{V}_{1,nq+1}/\mathcal{J}^+$ is a quotient of $\mathcal{K}_{1,2}\boxtimes\mathcal{K}_{1,nq}$, cofinite dimensions satisfy
\begin{equation*}
\dim_{C_1}(\mathcal{V}_{1,nq+1}/\mathcal{J}^+)\leq\dim_{C_1}(\mathcal{K}_{1,2}\boxtimes\mathcal{K}_{1,nq}) \leq 2nq,
\end{equation*}
using Propositions \ref{prop:Miyamoto}, \ref{prop:C1_quotient} and \ref{prop:C1-cofinite_dimension}. Thus using the PBW basis of $\mathcal{V}_{1,nq+1}$ as in the proof of Proposition \ref{prop:C1-cofinite_dimension}, we see that $\mathcal{J}^+$ contains a vector of conformal weight at most $h_{1,nq+1}+2nq$. Consulting the Verma module embedding diagrams in Section \ref{subsec:Vir_and_Verma}, we see that the maximal proper submodule of $\langle\tilde{v}_{1,nq+1}\rangle\cong\mathcal{V}_{p-1,(n+1)q+1}\subseteq\mathcal{V}_{1,nq+1}$ is generated by two singular vectors with conformal weights $h_{1,(n+2)q+1}$, $h_{p-1,(n+3)q-1}$. We calculate
\begin{align*}
h_{1,(n+2)q+1}-h_{1,nq+1} & = npq+(p-1)q+p > 2nq,\nonumber\\
h_{p-1,(n+3)q-1}-h_{1,nq+1} & = (n+1)(pq+q-p)\nonumber\\
& =2nq +n(p-2)q+pq+(n+1)(q-p)\nonumber\\
& =2nq+n((p-1)(q-1)-1)+(p+1)(q-1)+1 > 2nq,
\end{align*}
using $p,q\geq 2$. Since $\mathcal{J}^+$ has a vector of conformal weight at most $h_{1,nq+1}+2nq$, this shows that $\mathcal{J}^+$ is not contained in the maximal proper submodule of $\langle\tilde{v}_{1,nq+1}\rangle$, that is, $\mathcal{J}^+=\langle\tilde{v}_{1,nq+1}\rangle$. This proves $\mathcal{V}_{1,nq+1}/\mathcal{J}^+\cong\mathcal{K}_{1,nq+1}$.

We have now proved that for $r=1$, there is an exact sequence 
\begin{equation*}
0\longrightarrow\mathcal{K}_{1,nq-1}\longrightarrow\mathcal{K}_{1,2}\boxtimes\mathcal{K}_{1,nq}\longrightarrow\mathcal{K}_{1,nq+1}\longrightarrow 0.
\end{equation*}
For $2\leq r\leq p$, since $\mathcal{K}_{r,1}$ is rigid by Theorem \ref{thm:Krs_rigid}, this sequence remains exact after tensoring with $\mathcal{K}_{r,1}$. Thus we get an exact sequence
\begin{equation*}
0\longrightarrow\mathcal{K}_{r,nq-1}\longrightarrow\mathcal{K}_{1,2}\boxtimes\mathcal{K}_{r,nq}\longrightarrow\mathcal{K}_{r,nq+1}\longrightarrow 0
\end{equation*}
by Proposition \ref{prop:Kr1_K1s_small_r_or_s}. To show that this sequence does not split, we use the rigidity of $\mathcal{K}_{1,2}$ and Corollary \ref{cor:1st_Kac_module_fusion} to calculate
\begin{align*}
\dim\mathrm{Hom}(\mathcal{K}_{r,nq+1},\mathcal{K}_{1,2}\boxtimes\mathcal{K}_{r,nq}) & =\dim\mathrm{Hom}(\mathcal{K}_{1,2}\boxtimes\mathcal{K}_{r,nq+1},\mathcal{K}_{r,nq})\nonumber\\
& =\dim\mathrm{Hom}(\mathcal{K}_{r,nq}\oplus\mathcal{K}_{r,nq+2},\mathcal{K}_{r,nq})=1.
\end{align*}
Thus every homomorphism $\mathcal{K}_{r,nq+1}\rightarrow\mathcal{K}_{1,2}\boxtimes\mathcal{K}_{r,nq}$ is a multiple of the non-zero composition
\begin{equation*}
\mathcal{K}_{r,nq+1}\twoheadrightarrow\mathcal{L}_{r,nq+1}\hookrightarrow\mathcal{K}_{r,nq-1}\xrightarrow{F_{r,nq}^-}\mathcal{K}_{1,2}\boxtimes\mathcal{K}_{r,nq}
\end{equation*}
(recall the Kac module structures in Section \ref{subsec:FF_and_K_modules}). In particular, since $\mathrm{Im}\,F_{r,nq}^-\subseteq\mathrm{Ker}\,F_{r,nq}^+$ by Proposition \ref{prop:partial_exact_for_gen_rs}, there is no homomorphism $\sigma:\mathcal{K}_{r,nq+1}\rightarrow{K}_{1,2}\boxtimes{K}_{r,nq}$ such that $F_{r,nq}^+\circ\sigma =\mathrm{Id}$. This completes the proof of the theorem.
\end{proof}

We can now generalize Corollary \ref{cor:1st_Kac_module_fusion} and Proposition \ref{prop:K12_times_simple_indecomp} to all $r,s\in\mathbb{Z}_{\geq 1}$; recall that $\mathcal{K}_{r,0}=\mathcal{K}_{0,s}=0$:
\begin{Theorem}\label{thm:main_exact_sequences}
For any $r,s\in\mathbb{Z}_{\geq 1}$, the sequence
\begin{equation}\label{eqn:12_exact_seq}
0\longrightarrow\mathcal{K}_{r,s-1}\xrightarrow{F_{r,s}^-}\mathcal{K}_{1,2}\boxtimes\mathcal{K}_{r,s}\xrightarrow{F_{r,s}^+}\mathcal{K}_{r,s+1}\longrightarrow 0
\end{equation}
is exact, and this sequence splits if and only if $q\nmid s$. Similarly, the sequence
\begin{equation}\label{eqn:21_exact_seq}
0\longrightarrow\mathcal{K}_{r-1,s}\xrightarrow{G_{r,s}^-}\mathcal{K}_{2,1}\boxtimes\mathcal{K}_{r,s}\xrightarrow{G_{r,s}^+}\mathcal{K}_{r+1,s}\longrightarrow 0
\end{equation}
is exact, and this sequence splits if and only if $p\nmid r$.
\end{Theorem}

\begin{proof}
By $c_{p,q}=c_{q,p}$ symmetry, it is enough to prove the first assertion. To show that the sequence \eqref{eqn:12_exact_seq} is exact, Proposition \ref{prop:partial_exact_for_gen_rs} implies that we just need to show $\mathrm{Im}\,F_{r,s}^-=\mathrm{Ker}\,F_{r,s}^+$. Proposition \ref{prop:partial_exact_for_gen_rs} also implies that the composition factors of the finite-length module $\mathcal{K}_{1,2}\boxtimes\mathcal{K}_{r,s}$ are obtained by combining the composition factors of $\mathcal{K}_{r,s-1}\cong\mathrm{Im}\,F_{r,s}^-$, $\mathcal{K}_{r,s+1}=\mathrm{Im}\,F_{r,s}^+$, and $\mathrm{Ker}\,F_{r,s}^+/\mathrm{Im}\,F_{r,s}^-$. Thus we just need to show that $\mathcal{K}_{1,2}\boxtimes\mathcal{K}_{r,s}$ has the same number of composition factors as $\mathcal{K}_{r,s-1}\oplus\mathcal{K}_{r,s+1}$ (though the tensor product need not be isomorphic to the direct sum as $\mathfrak{Vir}$-modules). We will first prove this when either $r\leq p$ or $s\leq q$, and then inductively show that if the $(r,s)$ case holds, then so does the case $(p+r,q+s)$.

To handle the cases where either $r\leq p$ or $s\leq q$, Corollary \ref{cor:1st_Kac_module_fusion} shows that $\mathcal{K}_{1,2}\boxtimes\mathcal{K}_{r,s}$ has the same composition factors as $\mathcal{K}_{r,s-1}\oplus\mathcal{K}_{r,s+1}$ when $r\leq p$, $q\nmid s$ and when $s\leq q-1$. Then Proposition \ref{prop:K12_times_simple_indecomp} gives the same result when $r\leq p$, $q\mid s$. It remains to consider $\mathcal{K}_{1,2}\boxtimes\mathcal{K}_{mp+r,q}$ where $m\geq 1$ and $1\leq r\leq p-1$ (since $\mathcal{K}_{mp,q}\cong\mathcal{K}_{p,mq}$ for $m\geq 1$). From the Kac module structures in Section \ref{subsec:FF_and_K_modules}, we have an exact sequence
\begin{equation*}
0\longrightarrow\mathcal{K}_{r,(m+1)q}\longrightarrow\mathcal{K}_{mp+r,q}\longrightarrow\mathcal{K}_{p-r,mq}\longrightarrow 0,
\end{equation*}
where the submodule and quotient Kac modules are simple. Because $\mathcal{K}_{1,2}$ is rigid, this sequence remains exact after tensoring with $\mathcal{K}_{1,2}$, so $\mathcal{K}_{1,2}\boxtimes\mathcal{K}_{mp+r,q}$ has the same composition factors as
\begin{equation*}
(\mathcal{K}_{1,2}\boxtimes\mathcal{K}_{r,(m+1)q})\oplus(\mathcal{K}_{1,2}\boxtimes\mathcal{K}_{p-r,mq}),
\end{equation*}
which by Proposition \ref{prop:K12_times_simple_indecomp} has the same composition factors as
\begin{equation*}
\mathcal{K}_{r,(m+1)q-1}\oplus\mathcal{K}_{r,(m+1)q+1}\oplus\mathcal{K}_{p-r,mq-1}\oplus\mathcal{K}_{p-r,mq+1}.
\end{equation*}
Thus $\mathcal{K}_{1,2}\boxtimes\mathcal{K}_{mp+r,q}$ has $2+2+2+2=8$ composition factors, and from the Kac module structures in Section \ref{subsec:FF_and_K_modules}, $\mathcal{K}_{mp+r,q-1}\oplus\mathcal{K}_{mp+r,q+1}$ has $2+6=8$ composition factors as well. This proves that \eqref{eqn:12_exact_seq} is exact when either $r\leq p$ or $s\leq q$.

We now prove by induction on $\min(m,n)$ that $\mathcal{K}_{1,2}\boxtimes\mathcal{K}_{mp+r,nq+s}$ has the same composition factors as 
$$\mathcal{K}_{mp+r,nq+s-1}\oplus\mathcal{K}_{mp+r,nq+s+1}$$
 for any $1\leq r\leq p$, $1\leq s\leq q$, and $m,n\geq 0$. We have just completed proving the base case where either $m=0$ or $n=0$, and we need to show that if the $(m,n)$ case holds, then so does the $(m+1,n+1)$ case.
  
  To do so, note that $\mathcal{K}_{mp+r,nq+s}$ is a submodule of $\mathcal{K}_{(m+1)p+r,(n+1)q+s}$, and the structure of $\mathcal{K}_{(m+1)p+r,(n+1)q+s}/\mathcal{K}_{mp+r,nq+s}$ can be deduced from the Kac module structures in Section \ref{subsec:FF_and_K_modules}. In particular, if $1\leq r\leq p-1$, then this quotient of Kac modules has the following structure when $1\leq s\leq q-1$ or $s=q$, respectively:
  \begin{equation*}
  \begin{matrix}
  \xymatrix{
  & \mathcal{L}_{r,(m+n+4)q-s} \\
\mathcal{L}_{r,(m+n+2)q+s}  \ar[ru] & \mathcal{L}_{p-r,(m+n+3)q-s} \ar[u]\\
\mathcal{L}_{p-r,(m+n+1)q+s} \ar[u] \ar[ru] & \\
  }
  \end{matrix},
  \begin{matrix}\qquad\qquad
  \xymatrix{
  \mathcal{L}_{r,(m+n+3)q}\\
  \mathcal{L}_{p-r,(m+n+2)q} \ar[u]\\
  }
  \end{matrix}.
  \end{equation*}
 In both cases, these structure diagrams show that $\mathcal{K}_{(m+1)p+r,(n+1)q+s}/\mathcal{K}_{mp+r,nq+s}$ has a submodule (of length $2$ or $1$) generated by a singular vector of conformal weight $h_{r,(m+n+2)q+s}$. Recalling Remarks \ref{rem:Krs_Verma_quot} and \ref{rem:K_small_rs_2}, this submodule is isomorphic to $\mathcal{K}_{r,(m+n+2)q+s}$, and then similarly, the quotient of $\mathcal{K}_{(m+1)p+r,(n+1)q+s}/\mathcal{K}_{mp+r,nq+s}$ by this Kac submodule is $\mathcal{K}_{p-r,(m+n+1)q+s}$. That is,
   if $1\leq r\leq p-1$, then there is an exact sequence
 \begin{equation}\label{eqn:Kac_quotient_seq}
 0\longrightarrow\mathcal{K}_{r,(m+n+2)q+s}\longrightarrow\mathcal{K}_{(m+1)p+r,(n+1)q+s}/\mathcal{K}_{mp+r,nq+s}\longrightarrow\mathcal{K}_{p-r,(m+n+1)q+s}\longrightarrow 0,
 \end{equation}
 regardless of whether $1\leq s\leq q-1$ or $s=q$. If $r=p$, then a similar but easier analysis of the Kac module structures in Section \ref{subsec:FF_and_K_modules} shows that
 \begin{equation}\label{eqn:Kac_quotient_p}
\mathcal{K}_{(m+2)p,(n+1)q+s}/\mathcal{K}_{(m+1)p,nq+s}\cong \mathcal{K}_{p,(m+n+2)q+s}, 
 \end{equation}
again regardless of whether $1\leq s\leq q-1$ or $s=q$.

Now because $\mathcal{K}_{1,2}\boxtimes\bullet$ is exact, 
\begin{equation*}
\mathcal{K}_{1,2}\boxtimes(\mathcal{K}_{(m+1)p+r,(n+1)q+s}/\mathcal{K}_{mp+r,nq+s})\cong (\mathcal{K}_{1,2}\boxtimes\mathcal{K}_{(m+1)p+r,(n+1)q+s})/(\mathcal{K}_{1,2}\boxtimes\mathcal{K}_{mp+r,nq+s}).
\end{equation*}
Thus for $1\leq r\leq p-1$, we can use the exact sequence \eqref{eqn:Kac_quotient_seq} and exactness of $\mathcal{K}_{1,2}$ again to obtain the composition factors of $\mathcal{K}_{1,2}\boxtimes\mathcal{K}_{(m+1)p+r,(n+1)q+s}$ by combining the composition factors of $\mathcal{K}_{1,2}\boxtimes\mathcal{K}_{mp+r,nq+s}$, $\mathcal{K}_{1,2}\boxtimes\mathcal{K}_{r,(m+n+2)q+s}$, and $\mathcal{K}_{1,2}\boxtimes\mathcal{K}_{p-r,(m+n+1)q+s}$. By the induction assumption and the theorem in $r\leq p$ case, this means that $\mathcal{K}_{1,2}\boxtimes\mathcal{K}_{(m+1)p+r,(n+1)q+s}$ has the same composition factors as
%
%
\begin{align*}
(\mathcal{K}_{mp+r,nq+s-1} & \oplus\mathcal{K}_{r,(m+n+2)q+s-1}\oplus\mathcal{K}_{p-r,(m+n+1)q+s-1})\nonumber\\
&\oplus(\mathcal{K}_{mp+r,nq+s+1}\oplus\mathcal{K}_{r,(m+n+2)q+s+1}\oplus\mathcal{K}_{p-r,(m+n+1)q+s+1}),
\end{align*}
which by the $s\pm 1$ cases of \eqref{eqn:Kac_quotient_seq} has the same composition factors as
\begin{equation*}
\mathcal{K}_{(m+1)p+r,(n+1)q+s-1}\oplus\mathcal{K}_{(m+1)p+r,(n+1)q+s+1},
\end{equation*}
as required. The case $r=p$ is similar: by \eqref{eqn:Kac_quotient_p} and the exactness of $\mathcal{K}_{1,2}\boxtimes\bullet$, $\mathcal{K}_{1,2}\boxtimes\mathcal{K}_{(m+2)p,(n+1)q+s}$ has the same composition factors as 
\begin{equation*}
(\mathcal{K}_{1,2}\boxtimes\mathcal{K}_{(m+1)p,nq+s})\oplus(\mathcal{K}_{1,2}\boxtimes\mathcal{K}_{p,(m+n+2)q+s}),
\end{equation*}
which by the inductive hypothesis and the $r=p$ case of the theorem has the same composition factors as
\begin{equation*}
(\mathcal{K}_{(m+1)p,nq+s-1}\oplus\mathcal{K}_{p,(m+n+2)q+s-1})\oplus(\mathcal{K}_{(m+1)p,nq+s+1}\oplus\mathcal{K}_{p,(m+n+2)q+s+1}).
\end{equation*}
By the $s\pm 1$ cases of \eqref{eqn:Kac_quotient_p}, this in turn has the same composition factors as
\begin{equation*}
\mathcal{K}_{(m+2)p,(n+1)q+s-1}\oplus\mathcal{K}_{(m+2)p,(n+1)q+s+1}.
\end{equation*}
This completes the proof that \eqref{eqn:12_exact_seq} is exact for all $r,s\in\mathbb{Z}_{\geq 1}$.

To determine when \eqref{eqn:12_exact_seq} splits, first note that the conformal weights of $\mathcal{K}_{r,s\pm 1}$ are congruent to $h_{r,s\pm1}$ modulo $\mathbb{Z}$ and that $h_{r,s+1}-h_{r,s-1}\in\mathbb{Z}$ if and only if $q\mid s$. Thus if $q\nmid s$, then the conformal weights of $\mathcal{K}_{1,2}\boxtimes\mathcal{K}_{r,s}$ are contained in two disjoint cosets of $\mathbb{C}/\mathbb{Z}$. Since the Virasoro operators $L_n$ change conformal weights by integers, the direct sum of all weight spaces of $\mathcal{K}_{1,2}\boxtimes\mathcal{K}_{r,s}$ with conformal weights in $h_{r,s+1}+\mathbb{Z}$ form a submodule which must be isomorphic to $\mathcal{K}_{r,s+1}$. This means that \eqref{eqn:12_exact_seq} splits when $q\nmid s$. For the $q\mid s$ case, consider $\mathcal{K}_{1,2}\boxtimes\mathcal{K}_{mp+r,nq}$ for $1\leq r\leq p$, $m\geq 0$, and $n\geq 1$. Note that the submodule of $\mathcal{K}_{mp+r,nq+1}$ generated by its lowest conformal weight space is $\mathcal{K}_{r,(n-m)q+1}$ if $m\leq n$ and $\mathcal{K}_{(m-n)p+r,1}$ if $m\geq n$. Then using the rigidity of $\mathcal{K}_{1,2}$ and Corollary \ref{cor:1st_Kac_module_fusion} as in the proof of Proposition \ref{prop:K12_times_simple_indecomp}, we get
\begin{align*}
\dim\mathrm{Hom}(\mathcal{K}_{r,(n-m)q+1},\mathcal{K}_{1,2}\boxtimes\mathcal{K}_{mp+r,nq}) =\dim\mathrm{Hom}(\mathcal{K}_{r,(n-m)q}\oplus\mathcal{K}_{r,(n-m)q+2},\mathcal{K}_{mp+r,nq}) =1
\end{align*}
if $m<n$; the one-dimensional space on the left is thus spanned by the non-injective map
\begin{equation*}
\mathcal{K}_{r,(n-m)q+1}\twoheadrightarrow\mathcal{L}_{r,(n-m)q+1}\hookrightarrow\mathcal{K}_{mp+r,nq-1}\xrightarrow{F_{mp+r,nq}^-} \mathcal{K}_{1,2}\boxtimes\mathcal{K}_{mp+r,nq}.
\end{equation*}
Similarly, if $m\geq n$, then
\begin{align*}
\dim\mathrm{Hom}(\mathcal{K}_{(m-n)p+r,1},\mathcal{K}_{1,2}\boxtimes\mathcal{K}_{mp+r,nq}) =\dim\mathrm{Hom}(\mathcal{K}_{(m-n)p+r,2},\mathcal{K}_{mp+r,nq}) =0.
\end{align*}
In either case, any map $\mathcal{K}_{mp+r,nq+1}\rightarrow\mathcal{K}_{1,2}\boxtimes\mathcal{K}_{mp+r,nq}$ is non-injective, because it must restrict to a non-injective map on the submodule generated by the lowest conformal weight space of $\mathcal{K}_{mp+r,nq+1}$. Thus the exact sequence \eqref{eqn:12_exact_seq} does not split when $q\mid s$.
\end{proof}

\begin{Remark}
The exact sequences \eqref{eqn:12_exact_seq} and \eqref{eqn:21_exact_seq} agree with the Grothendieck fusion rules predicted using a conjectural Verlinde formula in \cite[Equation 4.16]{MRR}. Note however that the conjectural Verlinde formula does not predict whether or not these exact sequences split.
\end{Remark}

We can use Theorem \ref{thm:main_exact_sequences} to prove the fusion product formula conjectured in \cite[Equation 4.34]{MRR}; this also generalizes Proposition \ref{prop:Kr1_K1s_small_r_or_s} to all $r,s\in\mathbb{Z}_{\geq 1}$:

\begin{Theorem}\label{thm:Kr1_times_K1s}
For all $r, s \geq 1$, the surjective map $F_{r,1;1,s}:\mathcal{K}_{r,1}\boxtimes\mathcal{K}_{1,s}\rightarrow\mathcal{K}_{r,s}$ of Theorem \ref{thm:Fock_module_associativity} is an isomorphism.
\end{Theorem}
\begin{proof}
The $r = 1$ case is clear. We can prove the general case by induction on $r$. Consider the following diagram, where for simplicity we suppress associativity isomorphisms:
\begin{equation*}
\xymatrixcolsep{2.5pc}
\xymatrix{
 & \mathcal{K}_{r-1,1}\boxtimes\mathcal{K}_{1,s} \ar[r]^(.44){G_{r,1}^-\boxtimes\mathrm{Id}} \ar[d]^{F_{r-1,1;1,s}} & \mathcal{K}_{2,1}\boxtimes\mathcal{K}_{r,1}\boxtimes\mathcal{K}_{1,s} \ar[r]^(.55){\mathcal{G}_{r,1}^+\boxtimes\mathrm{Id}} \ar[d]^{\mathrm{Id}\boxtimes F_{r,1;1,s}} & \mathcal{K}_{r+1,1}\boxtimes\mathcal{K}_{1,s} \ar[r] \ar[d]^{F_{r+1,1;1,s}} & 0\\
 0 \ar[r] & \mathcal{K}_{r-1,s} \ar[r]^(.48){G_{r,s}^-} & \mathcal{K}_{2,1}\boxtimes\mathcal{K}_{r,s} \ar[r]^(.52){G_{r,s}^+} & \mathcal{K}_{r+1,s} \ar[r] & 0\\
}
\end{equation*}
The bottom row is exact by Theorem \ref{thm:main_exact_sequences}, and the top row is right exact because $\bullet\boxtimes\mathcal{K}_{1,s}$ is a right exact functor (see \cite[Proposition 4.26]{HLZ3}). Moreover, by Theorem \ref{thm:Fock_module_associativity}, the right square of the diagram commutes if we add a suitable non-zero multiple of the associativity isomorphism $(\mathcal{K}_{2,1}\boxtimes\mathcal{K}_{r,1})\boxtimes\mathcal{K}_{1,s}\rightarrow\mathcal{K}_{2,1}\boxtimes(\mathcal{K}_{r,1}\boxtimes\mathcal{K}_{1,s})$ to the middle vertical arrow. As for the left square of the diagram, the upper right composition along the left square is a non-zero multiple of
\begin{align*}
\mathcal{K}_{r-1,1}\boxtimes\mathcal{K}_{1,s} & \xrightarrow{i_{\mathcal{K}_{2,1}}\boxtimes\mathrm{Id}} \mathcal{K}_{2,1}\boxtimes\mathcal{K}_{2,1}\boxtimes\mathcal{K}_{r-1,1}\boxtimes\mathcal{K}_{1,s}\nonumber\\ &\xrightarrow{\mathrm{Id}\boxtimes G_{r-1,1}^+\boxtimes\mathrm{Id}} \mathcal{K}_{2,1}\boxtimes\mathcal{K}_{r,1}\boxtimes\mathcal{K}_{1,s}\xrightarrow{\mathrm{Id}\boxtimes F_{r,1;1,s}}\mathcal{K}_{2,1}\boxtimes\mathcal{K}_{r,s}
\end{align*}
(suppressing unit and associativity isomorphisms for simplicity). By Theorem \ref{thm:Fock_module_associativity}, this equals a non-zero multiple of
\begin{align*}
\mathcal{K}_{r-1,1}\boxtimes\mathcal{K}_{1,s} & \xrightarrow{i_{\mathcal{K}_{2,1}}\boxtimes\mathrm{Id}} \mathcal{K}_{2,1}\boxtimes\mathcal{K}_{2,1}\boxtimes\mathcal{K}_{r-1,1}\boxtimes\mathcal{K}_{1,s}\nonumber\\ 
& \xrightarrow{\mathrm{Id}\boxtimes F_{r-1,1;1,s}} \mathcal{K}_{2,1}\boxtimes\mathcal{K}_{2,1}\boxtimes\mathcal{K}_{r-1,s}\xrightarrow{\mathrm{Id}\boxtimes G_{r-1,s}^+} \mathcal{K}_{2,1}\boxtimes\mathcal{K}_{r,s}.
\end{align*}
By naturality of the implicit unit and associativity isomorphisms, this is the same as 
\begin{equation*}
\mathcal{K}_{r-1,1}\boxtimes\mathcal{K}_{1,s}\xrightarrow{F_{r-1,1;1,s}} \mathcal{K}_{r-1,s}\xrightarrow{i_{\mathcal{K}_{2,1}}\boxtimes\mathrm{Id}}\mathcal{K}_{2,1}\boxtimes\mathcal{K}_{2,1}\boxtimes\mathcal{K}_{r-1,s}\xrightarrow{\mathrm{Id}\boxtimes G_{r-1,s}^+} \mathcal{K}_{2,1}\boxtimes\mathcal{K}_{r,s},
\end{equation*}
which is $G_{r,s}^-\circ F_{r-1,1;1,s}$. Thus the left square of the diagram also commutes if we replace $F_{r-1,1;1,s}$ by a suitable non-zero multiple.

We now assume by induction on $r$ that $F_{r-1,1;1,s}$ and $F_{r,1;1,s}$ are isomorphisms. Then $G_{r,1}^-\boxtimes\mathrm{Id}$ is a non-zero multiple of $(\mathrm{Id}\boxtimes F_{r,1;1,s})^{-1}\circ G_{r,s}^-\circ F_{r-1,1;1,s}$ and thus is injective because $G_{r,s}^-$ is. That is, the top row of the diagram is actually exact, and we conclude that $F_{r+1,1;1,s}$ is an isomorphism from the Short Five Lemma.
\end{proof}

\subsection{Tensor products of \texorpdfstring{$\mathcal{K}_{1,2}$}{K12} and \texorpdfstring{$\mathcal{K}_{2,1}$}{K21} with simple modules}

Finally, we determine how the Kac module $\mathcal{K}_{1,2}$ tensors with all simple objects of $\mathcal{O}_{c_{p,q}}$. The corresponding results for $\mathcal{K}_{2,1}$ then follow by $c_{p,q}=c_{q,p}$ symmetry, so for brevity we will not state them explicitly here. Since $\mathcal{L}_{r,nq}=\mathcal{K}_{r,nq}$ for $1\leq r\leq p$ and $n\geq 1$, we have already largely determined $\mathcal{K}_{1,2}\boxtimes\mathcal{L}_{r,nq}$ in Proposition \ref{prop:K12_times_simple_indecomp}, but we will say a little more about these tensor products after the next theorem; in particular, we will show that they are indecomposable and that $L_0$ acts on them non-semisimply. For more uniform formulas, we use the notation $\mathcal{L}_{r,0}=0$.

\begin{Theorem}\label{thm:K12_times_simple}
Let $1\leq r\leq p$, $1\leq s\leq q-1$, and $n\geq 0$, with $n\geq 1$ if $r=p$. Then

\begin{equation*}
\mathcal{K}_{1,2}\boxtimes\mathcal{L}_{r,nq+s}\cong\left\lbrace\begin{array}{lll}
\mathcal{L}_{r,nq+s-1}\oplus\mathcal{L}_{r,nq+s+1} & \text{if} & s\leq q-2\\
\mathcal{L}_{r,(n+1)q-2} & \text{if} & s=q-1
\end{array}\right. .
\end{equation*}



\end{Theorem}
\begin{proof}
From the structures in Section \ref{subsec:FF_and_K_modules}, $\mathcal{K}_{r,np+s}$ for $1\leq r\leq p$, $1\leq s\leq q-1$, and $n\geq 0$ (with $n\geq 1$ if $r=p$) fits into a short exact sequence
\begin{equation*}
0\longrightarrow\mathcal{L}_{r,(n+2)q-s}\longrightarrow\mathcal{K}_{r,nq+s}\longrightarrow\mathcal{L}_{r,nq+s}\longrightarrow 0.
\end{equation*}
We can combine such exact sequences into a resolution of $\mathcal{L}_{r,nq+s}$ by Kac modules:
\begin{equation*}
\cdots\longrightarrow\mathcal{K}_{r,(n+4)q-s}\longrightarrow\mathcal{K}_{r,(n+2)q+s}\longrightarrow\mathcal{K}_{r,(n+2)q-s}\longrightarrow\mathcal{K}_{r,nq+s}\longrightarrow\mathcal{L}_{r,nq+s}\longrightarrow 0
\end{equation*}
(compare with Felder complexes \cite{Fe} of Feigin-Fuchs modules).
Since $\mathcal{K}_{1,2}$ is rigid and thus flat, this resolution remains exact after tensoring with $\mathcal{K}_{1,2}$. Then also using Corollary \ref{cor:1st_Kac_module_fusion}, we get a resolution of $\mathcal{K}_{1,2}\boxtimes\mathcal{L}_{r,nq+s}$:
\begin{equation*}
\cdots\longrightarrow\mathcal{K}_{r,(n+2)q-s-1}\oplus\mathcal{K}_{r,(n+2)q-s+1}\longrightarrow\mathcal{K}_{r,nq+s-1}\oplus\mathcal{K}_{r,nq+s+1}\longrightarrow\mathcal{K}_{1,2}\boxtimes\mathcal{L}_{r,nq+s}\longrightarrow 0.
\end{equation*}
Since the conformal weights of the two direct summands in each term of this resolution are non-congruent modulo $\mathbb{Z}$, we see that
\begin{equation*}
\mathcal{K}_{1,2}\boxtimes\mathcal{L}_{r,nq+s}=\mathcal{W}_{-}\oplus\mathcal{W}_{+}
\end{equation*}
where $\mathcal{W}_{\pm}$ have the following resolutions by Kac modules:
\begin{align*}
\cdots\longrightarrow\mathcal{K}_{r,(n+4)q-(s\pm 1)}\longrightarrow\mathcal{K}_{r,(n+2)q+s\pm 1}\longrightarrow\mathcal{K}_{r,(n+2)q-(s\pm 1)}\longrightarrow\mathcal{K}_{r,nq+s\pm 1}\longrightarrow\mathcal{W}_{\pm}\longrightarrow 0.
\end{align*}
We consider the possibilities for $s$ individually.

If $2\leq s\leq q-2$, then the resolutions for $\mathcal{W}_{\pm}$ are the same as those for $\mathcal{L}_{r,nq+s\pm 1}$, which implies that as graded vector spaces, $\mathcal{W}_{\pm}$ have the same graded dimensions as $\mathcal{L}_{r,nq+s\pm 1}$. Moreover, $\mathcal{W}_{\pm}$ are quotients of $\mathcal{K}_{r,nq+s\pm 1}$, as are $\mathcal{L}_{r,nq+s\pm 1}$, so it follows that $\mathcal{W}_{\pm}\cong\mathcal{L}_{r,nq+s\pm 1}$ when $2\leq s\leq q-2$. Similarly, assuming $q\geq 3$, $\mathcal{W}_+\cong\mathcal{L}_{r,nq+2}$ when $s=1$ and $\mathcal{W}_-=\mathcal{L}_{r,(n+1)q-2}$ when $s=q-1$.

Now consider $\mathcal{W}_-$ when $s=1$: it has a resolution
\begin{equation*}
\cdots\longrightarrow\mathcal{K}_{r,(n+4)q}\longrightarrow\mathcal{K}_{r,(n+2)q}\longrightarrow\mathcal{K}_{r,(n+2)q}\longrightarrow\mathcal{K}_{r,nq}\longrightarrow\mathcal{W}_-\longrightarrow 0.
\end{equation*}
If $n=0$, then $\mathcal{K}_{r,0}=0$ so $\mathcal{W}_-=0$ as well. Now assume $n>0$: Then from the structures in Section \ref{subsec:FF_and_K_modules}, $\mathcal{K}_{r,nq}$ is simple and there is no non-zero homomorphism $\mathcal{K}_{r,(n+2)q}\rightarrow\mathcal{K}_{r,nq}$. Thus $\mathcal{W}_-\cong\mathcal{K}_{r,nq}\cong\mathcal{L}_{r,nq}$.

Now consider $\mathcal{W}_+$ when $s=q-1$: it has resolution
\begin{equation*}
\cdots\longrightarrow\mathcal{K}_{r,(n+3)q}\longrightarrow\mathcal{K}_{r,(n+3)q}\longrightarrow\mathcal{K}_{r,(n+1)q}\longrightarrow\mathcal{K}_{r,(n+1)q}\longrightarrow\mathcal{W}_+\longrightarrow 0.
\end{equation*}
In this case, $\mathcal{K}_{r,(n+1)q}$ is simple and there is no non-zero map $\mathcal{K}_{r,(n+3)q}\rightarrow\mathcal{K}_{r,(n+1)q}$, so we get $\mathcal{W}_+=0$. This completes the proof of the theorem for all cases of $s\in\lbrace 1,2,\ldots, q-1\rbrace$; in particular, the case $q=2$, $s=1$ follows by considering $\mathcal{W}_-$ for $s=1$ and $\mathcal{W}_+$ for $s=q-1$.
\end{proof}

An easy consequence of this theorem is that $\mathcal{K}_{1,2}\boxtimes\mathcal{L}_{r,s}$ is indecomposable when $q\mid s$:
\begin{Corollary}
For $1\leq r\leq p$ and $n\geq 1$, the module $\mathcal{K}_{1,2}\boxtimes\mathcal{K}_{r,nq}$ is indecomposable.
\end{Corollary}
\begin{proof}
We will show that $\mathcal{K}_{1,2}\boxtimes\mathcal{K}_{r,nq}$ contains $\mathcal{L}_{r,nq+1}$ (with multiplicity one) as its unique simple submodule. This will imply that $\mathcal{K}_{1,2}\boxtimes\mathcal{K}_{r,nq}$ is indecomposable, since if we write $\mathcal{K}_{1,2}\boxtimes\mathcal{K}_{r,nq}=W_1+W_2$ as the sum of two non-zero submodules, then $\mathcal{L}_{r,nq+1}\subseteq W_1\cap W_2$, that is, the sum is not direct.

To determine the simple submodules of $\mathcal{K}_{1,2}\boxtimes\mathcal{K}_{r,nq}$, rigidity of $\mathcal{K}_{1,2}$ says that
\begin{equation*}
\mathrm{Hom}(\mathcal{L}_{r',s'},\mathcal{K}_{1,2}\boxtimes\mathcal{K}_{r,nq}) \cong\mathrm{Hom}(\mathcal{K}_{1,2}\boxtimes\mathcal{L}_{r',s'},\mathcal{K}_{r,nq})
\end{equation*}
for any $1\leq r'\leq p$ and $s'\in\mathbb{Z}_{\geq 1}$. If $q\nmid s'$, Theorem \ref{thm:K12_times_simple} says that this space is non-zero (of dimension one) if and only if $r'=r$ and $s'=nq+1$. If $q\mid s'$ on the other hand, then Proposition \ref{prop:K12_times_simple_indecomp} shows that $\mathcal{K}_{1,2}\boxtimes\mathcal{K}_{r,nq}$ does not contain any composition factors isomorphic to $\mathcal{L}_{r',s'}$. Thus the unique simple submodule of $\mathcal{K}_{1,2}\boxtimes\mathcal{K}_{r,nq}$ is isomorphic to $\mathcal{L}_{r,nq+1}$.
\end{proof}

We present one more result on $\mathcal{K}_{1,2}\boxtimes\mathcal{K}_{r,nq}$ (where we now allow any $r\in\mathbb{Z}_{\geq 1}$): we will show that it is a logarithmic $\mathfrak{Vir}$-module, that is, $L_0$ acts non-semisimply on $\mathcal{K}_{1,2}\boxtimes\mathcal{K}_{r,nq}$:
\begin{Proposition}\label{prop:log_modules}
For $r\in\mathbb{Z}_{\geq 1}$ and $n\geq 1$, $\mathcal{K}_{1,2}\boxtimes\mathcal{K}_{r,nq}$ is a logarithmic $\mathfrak{Vir}$-module.
\end{Proposition}
\begin{proof}
We will show that the twist automorphism $\theta=e^{2\pi i L_0}$ acts non-semisimply on $\mathcal{K}_{1,2}\boxtimes\mathcal{K}_{r,nq}$. The twist on $\mathcal{K}_{1,2}\boxtimes\mathcal{K}_{r,nq}$ satisfies the \textit{balancing equation}
\begin{equation}\label{eqn:balancing}
e^{2\pi i L_0} =\mathcal{R}^2_{\mathcal{K}_{1,2},\mathcal{K}_{r,nq}}\circ(e^{2\pi i L_0}\boxtimes e^{2\pi i L_0}),
\end{equation}
where
\begin{equation*}
\mathcal{R}_{\mathcal{K}_{1,2},\mathcal{K}_{r,nq}}^2: \mathcal{K}_{1,2}\boxtimes\mathcal{K}_{r,nq}\xrightarrow{\cong}\mathcal{K}_{r,nq}\boxtimes\mathcal{K}_{1,2}\xrightarrow{\cong}\mathcal{K}_{1,2}\boxtimes\mathcal{K}_{r,nq}
\end{equation*}
is the double braiding in $\mathcal{O}_{c_{p,q}}$. The balancing equation follows from the formula \eqref{eqn:braiding} for the braiding in $\mathcal{O}_{c_{p,q}}$ together with the $L_0$-conjugation formula
\begin{equation*}
y^{L_0}\mathcal{Y}(w_1,z)y^{-L_0} =\mathcal{Y}(y^{L_0} w_1,yz)
\end{equation*}
which gives the action of scaling transformations on intertwining operators (see for example \cite[Proposition 3.36(b)]{HLZ2}).

Now assume for the sake of contradiction that $e^{2\pi i L_0}$ is semisimple on $\mathcal{K}_{1,2}\boxtimes\mathcal{K}_{r,nq}$. By Theorem \ref{thm:main_exact_sequences}, the conformal weights of $\mathcal{K}_{1,2}\boxtimes\mathcal{K}_{r,nq}$ are contained in $(h_{r,nq-1}+\mathbb{Z}_{\geq 0})\cup(h_{r,nq+1}+\mathbb{Z}_{\geq 0})$. Moreover, $h_{r,nq+1}-h_{r,nq-1}=-r+pn\in\mathbb{Z}$, so if $e^{2\pi i L_0}$ were semisimple, then it would act on $\mathcal{K}_{1,2}\boxtimes\mathcal{K}_{r,nq}$ by the scalar $e^{2\pi i h_{r,nq+1}}$. Then the balancing equation \eqref{eqn:balancing}
would imply
\begin{equation*}
\mathcal{R}_{\mathcal{K}_{1,2},\mathcal{K}_{r,nq}}^2=e^{2\pi i L_0}\circ(e^{-2\pi i L_0}\boxtimes e^{-2\pi i L_0}) = e^{2\pi i(h_{r,nq+1}-h_{1,2}-h_{r,nq})}\mathrm{Id}_{\mathcal{K}_{1,2}\boxtimes\mathcal{K}_{r,nq}}.
\end{equation*}
Since
\begin{equation*}
h_{r,nq+1}-h_{1,2}-h_{r,nq}=\frac{pn-r+1}{2}-\frac{p}{2q},
\end{equation*}
this would yield $\mathcal{R}_{\mathcal{K}_{1,2},\mathcal{K}_{r,nq}}^2 =(-1)^{pn-r+1} e^{-\pi i p/q}\mathrm{Id}_{\mathcal{K}_{1,2}\boxtimes\mathcal{K}_{r,nq}}$.

Consider now the double braiding isomorphism
\begin{equation*}
\mathcal{R}_{\mathcal{K}_{1,2}\boxtimes\mathcal{K}_{1,2},\mathcal{K}_{r,nq}}^2: (\mathcal{K}_{1,2}\boxtimes\mathcal{K}_{1,2})\boxtimes\mathcal{K}_{r,nq}\xrightarrow{\cong}\mathcal{K}_{r,nq}\boxtimes(\mathcal{K}_{1,2}\boxtimes\mathcal{K}_{1,2})\xrightarrow{\cong}(\mathcal{K}_{1,2}\boxtimes\mathcal{K}_{1,2})\boxtimes\mathcal{K}_{r,nq}.
\end{equation*}
Since we are assuming $\mathcal{R}_{\mathcal{K}_{1,2},\mathcal{K}_{r,nq}}^2$ is a scalar multiple of the identity, we can use the hexagon axiom for braided tensor categories (with associativity isomorphisms suppressed from the notation for simplicity) to calculate that
\begin{align*}
\mathcal{R}_{\mathcal{K}_{1,2}\boxtimes\mathcal{K}_{1,2},\mathcal{K}_{r,nq}}^2 & = (\mathrm{Id}_{\mathcal{K}_{1,2}}\boxtimes\mathcal{R}_{\mathcal{K}_{1,2},\mathcal{K}_{r,nq}})\circ(\mathcal{R}_{\mathcal{K}_{1,2},\mathcal{K}_{r,nq}}\boxtimes\mathrm{Id}_{\mathcal{K}_{1,2}})^2\circ(\mathrm{Id}_{\mathcal{K}_{1,2}}\boxtimes\mathcal{R}_{\mathcal{K}_{1,2},\mathcal{K}_{r,nq}})\nonumber\\
& = (-1)^{pn-r+1} e^{-\pi ip/q}(\mathrm{Id}_{\mathcal{K}_{1,2}}\boxtimes\mathcal{R}_{\mathcal{K}_{1,2},\mathcal{K}_{r,nq}})^2\nonumber\\
& = e^{-2\pi ip/q}\mathrm{Id}_{(\mathcal{K}_{1,2}\boxtimes\mathcal{K}_{1,2})\boxtimes\mathcal{K}_{r,nq}}.
\end{align*}
However, recalling that $\mathcal{K}_{1,2}$ is rigid with coevaluation $i_{\mathcal{K}_{1,2}}: \mathcal{K}_{1,1}\rightarrow\mathcal{K}_{1,2}\boxtimes\mathcal{K}_{1,2}$, this leads to
\begin{align*}
e^{-2\pi ip/q} (i_{\mathcal{K}_{1,2}}\boxtimes\mathrm{Id}_{\mathcal{K}_{r,nq}}) & = \mathcal{R}_{\mathcal{K}_{1,2}\boxtimes\mathcal{K}_{1,2},\mathcal{K}_{r,nq}}^2\circ (i_{\mathcal{K}_{1,2}}\boxtimes\mathrm{Id}_{\mathcal{K}_{r,nq}})\nonumber\\
& = (i_{\mathcal{K}_{1,2}}\boxtimes\mathrm{Id}_{\mathcal{K}_{r,nq}})\circ\mathcal{R}^2_{\mathcal{K}_{1,1},\mathcal{K}_{r,nq}} = i_{\mathcal{K}_{1,2}}\boxtimes\mathrm{Id}_{\mathcal{K}_{r,nq}},
\end{align*}
since the braiding isomorphisms are natural and since double braiding with the unit object in a braided tensor category is always the identity. This is a contradiction since $e^{-2\pi i p/q}\neq 1$ (because $q>1$ and $\mathrm{gcd}(p,q)=1$) and since $i_{\mathcal{K}_{1,2}}\boxtimes\mathrm{Id}_{\mathcal{K}_{r,nq}}\neq 0$ (because the composition
\begin{align*}
\mathcal{K}_{1,2}\boxtimes\mathcal{K}_{r,nq}\xrightarrow{\mathrm{Id}\boxtimes i_{\mathcal{K}_{1,2}}\boxtimes\mathrm{Id}} \mathcal{K}_{1,2}\boxtimes\mathcal{K}_{1,2}\boxtimes\mathcal{K}_{1,2}\boxtimes\mathcal{K}_{r,nq}\xrightarrow{e_{\mathcal{K}_{1,2}}\boxtimes\mathrm{Id}\boxtimes\mathrm{Id}} \mathcal{K}_{1,2}\boxtimes\mathcal{K}_{r,nq}
\end{align*}
is non-zero by the rigidity of $\mathcal{K}_{1,2}$). Thus $L_0$ in fact acts non-semisimply on $\mathcal{K}_{1,2}\boxtimes\mathcal{K}_{r,nq}$.
\end{proof}

\begin{Remark}\label{rem:log_modules}
By $c_{p,q}=c_{q,p}$ symmetry, $\mathcal{K}_{2,1}\boxtimes\mathcal{K}_{mp,s}$ is also indecomposable for $m\geq 1$ and $1\leq s\leq q$, and logarithmic for all $s\in\mathbb{Z}_{\geq 1}$.
\end{Remark}

\begin{Remark}
Logarithmic modules which are extensions of highest-weight quotients by highest-weight submodules are called ``staggered modules'' in the physics literature; staggered modules for the Virasoro algebra have been studied comprehensively in \cite{KyRi}. Propositions \ref{prop:K12_times_simple_indecomp} and \ref{prop:log_modules} show that $\mathcal{K}_{1,2}\boxtimes\mathcal{K}_{r,nq}$ is a staggered module when $1\leq r\leq p$ and $n\geq 1$.
\end{Remark}

\section{Conclusion and outlook}\label{sec:conclusion}

In this paper, we have explored the structure of the tensor category $\mathcal{O}_{c_{p,q}}$ \cite{CJORY} of $C_1$-cofinite modules for the universal Virasoro vertex operator algebra $V_{c_{p,q}}$ at central charge $c_{p,q}=1-\frac{6(p-q)^2}{pq}$ for integers $p,q\geq 2$ with $\mathrm{gcd}(p,q)=1$. We have especially focused on the Kac modules $\mathcal{K}_{r,s}$ for $r,s\in\mathbb{Z}_{\geq 1}$ defined in \cite{MRR}. We have shown that $\mathcal{K}_{1,2}$ and $\mathcal{K}_{2,1}$ are rigid and self-dual (though the full tensor category $\mathcal{O}_{c_{p,q}}$ is not rigid), and that the tensor subcategory of $\mathcal{O}_{c_{p,q}}$ generated by $\mathcal{K}_{1,2}$ and $\mathcal{K}_{2,1}$ (closed under subquotients as well as tensor products)  contains all Kac modules and irreducible modules in $\mathcal{O}_{c_{p,q}}$. We have also determined the tensor products of $\mathcal{K}_{1,2}$ and $\mathcal{K}_{2,1}$ with all Kac modules and irreducible modules of $\mathcal{O}_{c_{p,q}}$, in the process proving some fusion product conjectures from \cite{MRR}.

An obvious next step would be to calculate all fusion tensor products involving all Kac modules and irreducible modules. We expect that many of these calculations would be relatively straightforward consequences of our Theorems \ref{thm:main_exact_sequences}, \ref{thm:Kr1_times_K1s}, and \ref{thm:K12_times_simple}. However, it will be necessary to construct more indecomposable modules, similar to $\mathcal{K}_{1,2}\boxtimes\mathcal{K}_{r,nq}$, which will surely appear in  the general fusion tensor product formulas. Such modules should be constructible as indecomposable summands of repeated tensor products of $\mathcal{K}_{r,nq}$ with $\mathcal{K}_{1,2}$ and $\mathcal{K}_{2,1}$, similar to the constructions of \cite[Section 5.2]{MY-cp1-vir}, for example. Detailed conjectures for general fusion product formulas have been derived in \cite{RP}.

Closely related to the problem of constructing indecomposable modules in $\mathcal{O}_{c_{p,q}}$ is the problem of constructing and classifying modules which are projective in some suitable natural subcategory of $\mathcal{O}_{c_{p,q}}$. The category $\mathcal{O}_{c_{p,q}}$ itself does not have non-zero projective objects. Indeed, if a projective object $\mathcal{P}$ in $\mathcal{O}_{c_{p,q}}$ were non-zero, it would surject onto some simple quotient $\mathcal{L}_{r,s}$. But for every non-zero submodule $\mathcal{J}\subseteq\mathcal{V}_{r,s}$, the Verma module quotient $\mathcal{V}_{r,s}/\mathcal{J}$ is a finite-length object of $\mathcal{O}_{c_{p,q}}$ that also surjects onto $\mathcal{L}_{r,s}$. So by projectivity there would be a map $f_\mathcal{J}:\mathcal{P}\rightarrow\mathcal{V}_{r,s}/\mathcal{J}$ for any non-zero $\mathcal{J}$ whose image would contain the generating vector $v_{r,s}+\mathcal{J}$. Thus $f_\mathcal{J}$ would be surjective, but this is impossible because $\mathcal{P}$ has fixed finite length and so cannot surject onto all finite-length quotients of $\mathcal{V}_{r,s}$. 

Thus an outstanding problem is to identify a natural tensor subcategory of $\mathcal{O}_{c_{p,q}}$ which contains all simple objects and in which every simple object has a projective cover. In particular, such a subcategory must exclude most finite-length quotients of any Verma module. In the case of central charge $c_{p,1}$ studied in \cite{MY-cp1-vir}, the correct subcategory is the M\"{u}ger centralizer of $\mathcal{K}_{3,1}$, which also turns out to be the tensor subcategory of $\mathcal{O}_{c_{p,1}}$ generated by $\mathcal{K}_{1,2}$. In the general $c_{p,q}$ case, it seems likely that the tensor subcategory generated by both $\mathcal{K}_{1,2}$ and $\mathcal{K}_{2,1}$ is at least close to being the correct subcategory of $\mathcal{O}_{c_{p,q}}$ for obtaining projective objects. However, at the moment it is not clear if this subcategory would contain the correct projective covers of $\mathcal{L}_{r,s}$ for $r< p$ and $s< q$.

One possible application of finding a subcategory of $\mathcal{O}_{c_{p,q}}$ with enough projectives might be in constructing a full (bulk) logarithmic conformal field theory from the chiral LCFT $\mathcal{LM}(p,q)$. There is a vertex operator algebra approach to constructing a full CFT from a chiral one due to Huang and Kong called full field algebras \cite{HK1, HK2, Ko}. However, the main constructions of full field algebras so far apply to rational vertex operator algebras rather than to those appearing in LCFT. There is also a tensor-categorical approach to constructing the bulk space of a full CFT from a chiral one; see for example \cite{RFFS, GaRu, GRW2}. This approach does apply to at least some chiral LCFTs, but a key input in \cite{GRW2}, for example, is a braided tensor category with enough projective objects. Thus extending the methods of \cite{GRW2} to examples such as $\mathcal{LM}(p,q)$ would probably at least require constructing projective covers in an appropriate subcategory of $\mathcal{O}_{c_{p,q}}$.

Another research direction is to apply vertex operator algebra extension theory \cite{HKL,CKM-exts} to our results on $\mathcal{O}_{c_{p,q}}$ to study the singlet and triplet $W$-algebras, which are non-simple vertex operator algebras containing $V_{c_{p,q}}$ as vertex operator subalgebra. The singlet and triplet extensions of the Virasoro algebra provide some of the basic examples of logarithmic conformal field theories, and there are higher-rank generalizations which extend higher-rank principal affine $W$-algebras \cite{FT, Su}. The representation theories of the singlet and triplet extensions of the Virasoro algebra at central charge $c_{p,1}$ are by now well-understood mathematically, thanks largely to the works \cite{Ad-sing, AM-trip, NT, TW, MY-cp1-vir, CMY-singlet, CMY-sing2, GN, CLR}. However, questions and conjectures remain about the singlet and triplet algebras at central charge $c_{p,q}$, although these algebras have also received extensive study in the mathematical physics literature \cite{FGST1, FGST2, RP4, Ra1, Ra2, Ra3, Ra4, GRW, Wo, RW, TW2, AM2, AM3, AM4, AM5, AM6}. We expect it should be possible to use the induction functors \cite{KO, CKM-exts} from $\mathcal{O}_{c_{p,q}}$ to the categories of singlet and triplet modules to rigorously determine fusion tensor products and prove (non)-rigidity for at least some modules of the singlet and triplet algebras.

We also expect that methods similar to those applied to the Virasoro algebra in this paper will be useful for studying the $N=1$ super Virasoro algebras \cite{NS, Ramond} at suitable central charges, because the representation theories of these superalgebras are analogous to that of the Virasoro algebra in many respects. For example, Verma modules for the $N=1$ super Virasoro algebra (at least in the Neveu-Schwarz sector) have structures similar to Virasoro Verma modules \cite{As, IK-N=1-1}, and the $N=1$ super Virasoro algebra admits free field representations analogous to the Heisenberg Fock modules for the Virasoro algebra \cite{MR-C, IK-N=1-2}. Using these free field representations, Kac modules for both the Neveu-Schwarz and Ramond $N=1$ super Virasoro algebras were introduced in \cite{CRR, CR}, and fusion rules for these Kac modules were conjectured. In upcoming work \cite{CMOHY}, it will be shown that the category of $C_1$-cofinite modules for the Neveu-Schwarz $N=1$ super Virasoro algebra at any central charge admits vertex algebraic braided tensor category structure, analogous to the corresponding result for the Virasoro algebra in \cite{CJORY}. Thus it may soon be possible to apply the techniques of the present paper to prove some of the conjectures in \cite{CRR}. We also hope it will be possible to apply our methods to the $N=2$ superconformal algebras (see for example \cite{Gray, CLRW, LPX, RRR} for some recent work on the representation theory of these algebras, as well as the related work \cite{So} of the second author of this paper).

Finally, we note that there are analogues of the Virasoro Kac modules studied in this paper for the affine Lie algebra $\widehat{\mathfrak{sl}}_2$: they are defined in \cite{Ra-sl2} as certain finitely-generated submodules of Wakimoto modules. This is one of many links between the representation theories of $\widehat{\mathfrak{sl}}_2$ and the Virasoro algebra (recall that  the Virasoro algebra is the affine $W$-algebra associated to $\widehat{\mathfrak{sl}}_2$ by quantum Drinfeld-Sokolov reduction \cite{FFr}). In fact, upcoming work of the first author of this paper with Jinwei Yang will use Proposition \ref{prop:intrinsic_dim} to show that there are braided tensor functors from the Kazhdan-Lusztig category \cite{KL1} of $\widehat{\mathfrak{sl}}_2$ at levels $-2+\frac{p}{q}$ and $-2+\frac{q}{p}$ to the Virasoro category $\mathcal{O}_{c_{p,q}}$. These functors can be viewed as tensor-categorical versions of quantum Drinfeld-Sokolov reduction for $\widehat{\mathfrak{sl}}_2$.

\end{document}